\documentclass{amsart}

\usepackage[mathscr]{eucal}
\usepackage[normalem]{ulem}
\usepackage[usenames,dvipsnames]{xcolor}
\usepackage[utf8]{inputenc}
\usepackage{amsmath}
\usepackage{amssymb}
\usepackage{amsthm}
\usepackage{array}
\usepackage{bbold}
\usepackage{enumerate}
\usepackage{etoolbox}
\usepackage{xspace}

\numberwithin{equation}{section}
\usepackage[all]{xy}
\SelectTips{cm}{}
\newdir{ >}{{}*!/-10pt/\dir{>}}

\usepackage[colorlinks=true,linkcolor={Brown},citecolor={Brown},urlcolor={Brown}]{hyperref}
\usepackage{bookmark} 
\usepackage{cleveref}

\swapnumbers 
\newtheorem{Cor}[equation]{Corollary}
\newtheorem{Lem}[equation]{Lemma}
\newtheorem{KLem}[equation]{Key Lemma}
\newtheorem{Prop}[equation]{Proposition}
\newtheorem{Thm}[equation]{Theorem}
\newtheorem{MainThm}[equation]{Main Theorem}

\theoremstyle{remark}
\newtheorem{Def}[equation]{Definition}
\newtheorem{Not}[equation]{Notation}
\newtheorem{Exa}[equation]{Example}
\newtheorem{Exas}[equation]{Examples}

\newtheorem{Rem}[equation]{Remark}
\newtheorem{Remind}[equation]{Reminder}
\newtheorem{Cons}[equation]{Construction}

\newcommand{\nc}{\newcommand}
\nc{\dmo}{\DeclareMathOperator}

\dmo{\Ab}{Ab}
\dmo{\add}{add}
\dmo{\AM}{AM}
\dmo{\ATM}{ATM}
\dmo{\Chain}{Ch}
\dmo{\Chow}{Chow}
\dmo{\coker}{coker}
\dmo{\cone}{cone}
\dmo{\DAM}{DAM^{\geom}}%
\dmo{\DATM}{DATM^{\geom}}%
\dmo{\Der}{D}
\dmo{\DM}{DM^{\geom}}
\dmo{\DTM}{DTM^{\geom}}%
\dmo{\End}{End}
\dmo{\ev}{ev}
\dmo{\Ext}{Ext}
\dmo{\geom}{gm}
\dmo{\grmodname}{grmod}%
\dmo{\gr}{gr}
\dmo{\Hm}{H}
\dmo{\Hom}{Hom}
\dmo{\Hty}{\Kb}
\dmo{\Id}{Id}
\dmo{\id}{id}
\dmo{\Img}{Im}
\dmo{\ind}{ind}
\dmo{\Ker}{Ker}
\dmo{\Komp}{K}
\dmo{\Kos}{Kos}
\dmo{\modname}{mod}%
\dmo{\Mod}{Mod}
\dmo{\mot}{M}
\dmo{\opname}{op}
\dmo{\Proj}{Proj} 
\dmo{\proj}{proj}
\dmo{\res}{res}
\dmo{\SHmot}{SH^{\mathrm{c}}_{\bbA^{\!1}}}
\dmo{\SH}{SH^{\mathrm{c}}}
\dmo{\smallb}{b}
\dmo{\smallperf}{perf}
\dmo{\Spc}{Spc}
\dmo{\Spech}{Spec^h}
\dmo{\Spec}{Spec}
\dmo{\stabname}{stab}
\dmo{\stab}{stab}
\dmo{\st}{st}
\dmo{\supp}{supp}
\dmo{\thick}{thick}
\dmo{\TM}{TM}
\dmo{\triv}{triv}

\nc{\AbGrps}{\MMod{\bbZ}}
\nc{\Ac}{\mathrm{Ac}}
\nc{\adhpt}[1]{\adh{\{#1\}}}
\nc{\adh}[1]{\overline{#1}}
\nc{\adjto}{\rightleftarrows}
\nc{\adj}{\dashv}
\nc{\Afilmax}{\Afil_{\mathrm{q.ab}}}
\nc{\Afil}{\fil{\cat{A}}}
\nc{\afortiori}{{\sl a fortiori}}
\nc{\aka}{{a.\,k.\,a.}\ }
\nc{\ala}{{\`a la}\ }
\nc{\Aqab}{\Afilmax}
\nc{\Asplit}{\cA_{\mathrm{split}}}
\nc{\barpwzinv}{\barpwz\inv}
\nc{\barpwz}{\overline{\pwz}}
\nc{\cat}[1]{\mathscr{#1}}
\nc{\cA}{\cat{A}}
\nc{\cB}{\cat{B}}
\nc{\Cb}{\Chain_{\smallb}}
\nc{\cE}{\cat{E}}
\nc{\cF}{\cat{F}}
\nc{\cf}{{\sl cf.}\ }
\nc{\Ch}{\Cb}
\nc{\cJ}{\cat{J}}
\nc{\cK}{\cat{K}}
\nc{\cL}{\cat{L}}
\nc{\cM}{\cat{M}}
\nc{\cN}{\cat{N}}
\nc{\colim}{\mathop{\mathrm{colim}}}
\nc{\cP}{\cat{P}}
\nc{\cQ}{\cat{Q}}
\nc{\cR}{\cat{R}}
\nc{\cT}{\cat{T}}
\nc{\DABfil}{\Db(\Efil)}
\nc{\DA}{\Db(\cA)}
\nc{\Db}{\Der_{\smallb}}
\nc{\DEfil}{\Db(\Efil)}
\nc{\DETM}{\mathrm{D}\EE\mathrm{TM}(\FF;\bbZ/m)}%
\nc{\DF}{\mathbb{D}(\cT)}
\nc{\Displ}{\displaystyle}
\nc{\Dperf}{\Der^{\smallperf}}
\nc{\DT}{\mathbb{D}(\cT)}
\nc{\D}{\Der}
\nc{\Efil}{\Afil_{\mathrm{ex}}}
\nc{\eg}{{\sl e.g.}\xspace}
\nc{\eps}{\epsilon}
\nc{\et}{\textrm{\'et}}
\nc{\FF}{F} 
\nc{\fil}[2][]{\ifblank{#1}{#2^{\mathrm{fil}}}{(#2,#1)^{\mathrm{fil}}}}
\nc{\grmod}[1]{#1\text{-}\kern-0.1em\grmodname}%
\nc{\Homcat}[1]{\Hom_{\cat #1}}
\nc{\Homfil}{\Hom_{\Afil}}
\nc{\hook}{\hookrightarrow}
\nc{\ideal}[1]{\langle #1\rangle}
\nc{\ie}{{\sl i.e.}\ }
\nc{\injres}{\mathbb{J}}
\nc{\into}{\mathop{\rightarrowtail}}
\nc{\inv}{^{-1}}
\nc{\isofrom}{\overset{\sim}{\,\leftarrow\,}}
\nc{\isoto}{\overset{\sim}{\,\to\,}}
\nc{\Kbac}{\Komp_{\smallb,\mathrm{ac}}}
\nc{\Kbm}{\Komp_{\mathrm{b},\le0}}
\nc{\Kb}{\Komp_{\smallb}}
\nc{\kGfil}{\Efil}
\nc{\kk}{k}
\nc{\Km}{\Komp_{-}}
\nc{\Kp}{\Komp_{+}}
\nc{\loccit}{{\sl loc.\ cit.}\xspace}
\nc{\Mid}{\,\big|\,}
\nc{\mmod}[1]{#1\text{-}\kern-0.1em\modname}%
\nc{\MMod}[1]{#1\text{-}\kern-0.1em\Mod}%
\nc{\onto}{\mathop{\twoheadrightarrow}}
\nc{\op}{^{\opname}}
\nc{\oursetminus}{\smallsetminus}
\nc{\pos}{\mathrm{pos}}
\nc{\potimes}[1]{^{\otimes #1}}
\nc{\pproj}[1]{#1\text{-}\kern-0.1em\proj}%
\nc{\qquadtext}[1]{\qquad\textrm{#1}\qquad}
\nc{\quadtext}[1]{\quad\textrm{#1}\quad}
\nc{\reet}{\mathrm{R\acute{e}t}}
\nc{\restr}[1]{_{|_{\scriptstyle #1}}}
\nc{\rwz}{\mathrm{rwz}}
\nc{\sbull}{{\scriptscriptstyle\bullet}}
\nc{\seq}[2][]{\ifblank{#1}{#2^{\bbZ\op}}{(#2,#1)^{\bbZ\op}}}
\nc{\SET}[2]{\big\{\,#1\Mid#2\,\big\}}
\nc{\smat}[1]{\left(\begin{smallmatrix} #1 \end{smallmatrix}\right)}
\nc{\SpcK}{\Spc(\cK)}
\nc{\sstab}{\stabname\kern-0.1em\text{-}}%
\nc{\stabfil}{\stab(\Efil)}
\nc{\Tate}{\hat{\Hm}}
\nc{\tfgt}{\mathrm{f\widetilde{g}t}}
\nc{\too}{\mathop{\longrightarrow}\limits}
\nc{\To}{\Rightarrow}
\nc{\tristar}{\begin{center}*\ *\ *\end{center}}
\nc{\unit}{\mathbb{1}}
\nc{\via}{{\sl via}\xspace}
\nc{\wzp}{weight-zero part}
\nc{\xinto}[1]{\overset{#1}{\,\into\,}}
\nc{\xonto}[1]{\overset{#1}{\,\onto\,}}
\nc{\xto}[1]{\xrightarrow{#1}}

\nc{\pLs}{\cat{L}_1}
\nc{\pL}{\cat{L}_{0}}
\nc{\pMs}{\cat{M}_1}
\nc{\pM}{\cat{M}_0}
\nc{\pNs}{\cat{N}_1}
\nc{\pN}{\cat{N}_0}

\let\prLs\pLs
\let\prL\pL
\let\prMs\pMs
\let\prM\pM
\let\prNs\pNs
\let\prN\pN

\nc{\DE}{\DEfil}
\nc{\DEU}{\DEfil\restr{U}}
\nc{\KA}{\Kb(\cA)}
\nc{\KE}{\Kb(\Efil)}

\nc{\bamma}{\gamma}%
\nc{\E}{\mathbb{E}}
\nc{\Epur}{\mathrm{E}}
\nc{\fund}{\mathbb{S}}
\nc{\fundpur}{\mathrm{S}}
\nc{\T}{\mathbb{T}}
\nc{\invert}{\mathbb{L}}
\nc{\invertpur}{\mathrm{L}}

\nc{\chgcoef}{\mathrm{cc}}
\dmo{\quo}{quo}
\dmo{\sta}{sta}
\dmo{\Sta}{Sta}
\dmo{\fgt}{fgt}
\dmo{\pwz}{pwz}
\nc{\rsd}[1]{\mathrm{rsd}_{#1}}

\usepackage{mathtools}
\usepackage{xspace}
\makeatletter
\let\ea\expandafter
\newcount\foreachcount

\def\foreachLetter#1#2#3{\foreachcount=#1
  \ea\loop\ea\ea\ea#3\@Alph\foreachcount
  \advance\foreachcount by 1
  \ifnum\foreachcount<#2\repeat}

\def\definebb#1{\ea\gdef\csname bb#1\endcsname{\ensuremath{\mathbb{#1}}\xspace}}
\foreachLetter{1}{27}{\definebb}

\makeatother

\begin{document}


\title{Three real Artin-Tate motives}
\author{Paul Balmer}
\date{\today}

\address{Paul Balmer, Mathematics Department, UCLA, Los Angeles, CA 90095-1555, USA}
\email{balmer@math.ucla.edu}
\urladdr{http://www.math.ucla.edu/~balmer}

\author{Martin Gallauer}
\address{Martin Gallauer, Max-Planck-Institut f\"ur Mathematik, 53111 Bonn, Germany}
\email{gallauer@mpim-bonn.mpg.de}
\urladdr{https://guests.mpim-bonn.mpg.de/gallauer}

\begin{abstract}
We analyze the spectrum of the tensor-triangulated category of Artin-Tate motives over the base field~$\mathbb{R}$ of real numbers, with integral coefficients.
Away from~$2$, we obtain the same spectrum as for complex Tate motives, previously studied by the second-named author. So the novelty is concentrated at the prime~$2$, where modular representation theory enters the picture via work of Positselski, based on Voevodsky's resolution of the Milnor Conjecture. With coefficients in~$k=\mathbb{Z}/2$, our spectrum becomes homeomorphic to the spectrum of the derived category of filtered $kC_2$-modules with a peculiar exact structure, for the cyclic group~$C_2=\mathrm{Gal}(\mathbb{C}/\mathbb{R})$. This spectrum consists of six points organized in an interesting way. As an application, we find exactly fourteen classes of mod-2 real Artin-Tate motives, up to the tensor-triangular structure. Among those, three special motives stand out, from which we can construct all others. We also discuss the spectrum of Artin motives and of Tate motives.
\end{abstract}


\thanks{First-named author supported by NSF grant~DMS-1901696. Second-named author partially supported by IMPA, a Titchmarsh Fellowship of the University of Oxford and the Lockey Fund.}

\maketitle

\tableofcontents

\section{Introduction}


We explore motives and representation theory, from the perspective of tensor-triangular geometry. We first present our results in the language of motives and then turn to representation theory in the second part of this introduction.

\subsection*{Motivic result}

Consider the tensor-triangulated category (\emph{tt-category} for short)
\[
\cK=\DATM(\bbR;\bbZ/2)
\]
of geometric mixed Artin-Tate motives over the base field~$\FF=\bbR$ of real numbers, with coefficients modulo~$2$. (\textsl{Mutatis mutandis},~$\FF$ could be any real closed field.) In Voevodsky's category~$\DM(\FF;k)$ of geometric motives over~$\FF$, with coefficients in a commutative ring~$k$ (see~\cite{Voevodsky00}), the tt-subcategory $\DATM(\FF;k)$ is generated by Tate objects~$k(i)$, $i\in \bbZ$, and Artin motives $\mot(\Spec(E))$ for finite separable extensions~$E/\FF$. Our base field~$\FF=\bbR$ is the simplest non-trivial case, involving only one separable extension~$E=\bbC$. And yet, we shall see that this case is already interesting, particularly for coefficients in~$k=\bbZ/2$.

Every tt-category~$\cK$ admits a spectrum~$\SpcK$, a space that classifies objects of~$\cK$ up to the tensor-triangular structure; see~\cite{balmer:spectrum,balmer:icm} or \Cref{Rmd:Spc}. For `tensor-triangular geometry' and its relevance beyond motives and representation theory, the reader is referred to the surveys~\cite{Stevenson18,balmer:guide-HT-handbook}. Here, our goal is:
\begin{MainThm}[{\Cref{Thm:main}}]
\label{Thm:main-intro}%
The spectrum of the tensor-triangulated category $\cK$ of real Artin-Tate motives with $\bbZ/2$-coefficients is the six-point space
\[
\SpcK\ =
\qquad\qquad
\vcenter{\xymatrix@R=.01em{
\overset{\pLs}{\bullet}
&& \overset{\pNs}{\bullet}
\\
\underset{\pL}{\bullet} \ar@{-}[u]
& \overset{\pMs}{\bullet} \ar@{-}[lu] \ar@{-}[ru]
& \underset{\pN}{\bullet} \ar@{-}[u]
\\
& \underset{\pM}{\bullet} \ar@{-}[u] \ar@{-}[lu] \ar@{-}[ru]
}}\qquad
\]
where a line $\bullet - \bullet$ indicates that the higher point lies in the closure of the lower one.
\end{MainThm}

Before discussing consequences, let us unpack the above result. The space $\SpcK$ has exactly fourteen closed subsets, ordered by inclusion as follows:
\begin{equation}
\label{eq:lattice-closed}\vcenter{
\xymatrix@C=.1em@R=.01em{
&& \varnothing
\\
& {\{\pLs\}} \ar@{-}[ru]
&& {\{\pNs\}} \ar@{-}[lu]
\\
{\color{Orange}\{\pL,\pLs\}} \ar@{-}[ru]
&& \{\pLs\}\sqcup\{\pNs\} \ar@{-}[lu] \ar@{-}[ru]
&& {\color{Orange}\{\pN,\pNs\}} \ar@{-}[lu]
\\\\
& \kern-2em \{\pL,\pLs\}\sqcup\{\pNs\} \kern-1em \ar@{-}[ruu] \ar@{-}[luu]
& {\color{Orange}\{\pLs,\pMs,\pNs\}} \ar@{-}[uu]
& \kern-1em \{\pLs\}\sqcup\{\pN,\pNs\} \kern-2em \ar@{-}[ruu] \ar@{-}[luu]
&&
\\\\
\kern-1em\{\pL,\pLs,\pMs,\pNs\} \kern-.3em \ar@{-}[uuuu] \ar@{-}@<-0.2em>[rruu]|!{[rr];[ruu]}{\hole}
&& \{\pL,\pLs\}\sqcup\{\pN,\pNs\} \ar@{-}[luu] \ar@{-}[ruu]
&& \kern-.3em \{\pLs,\pMs,\pN,\pNs\} \kern-1em \ar@{-}[uuuu] \ar@{-}@<0.2em>[lluu]|!{[ll];[luu]}{\hole}
\\\\
&& \{\pL,\pLs,\pMs,\pN,\pNs\} \ar@{-}[lluu] \ar@{-}[uu] \ar@{-}[rruu]
\\\\
&& {\SpcK} \ar@{-}[uu]
}}\kern-2em
\end{equation}
The space $\SpcK$ has Krull dimension two and admits six irreducible closed subsets. In addition to~$\SpcK=\adhpt{\pM}$ itself and the two closed points $\{\pLs\}=\adhpt{\pLs}$ and $\{\pNs\}=\adhpt{\pNs}$, there are three irreducibles (highlighted in orange in~\eqref{eq:lattice-closed} above)
\begin{equation}
\label{eq:3-irreduciblos}%
{
{\xymatrix@R=.2em@C=1em@H=.5em{
\bullet \ar@{-}[d]\ar@[Gray]@{-}[rd] && {\color{Gray}\circ} \ar@[Gray]@{-}[ld]\ar@[Gray]@{-}[d]
\\
\bullet \ar@[Gray]@{-}[rd] & {\color{Gray}\circ} \ar@[Gray]@{-}[d] & {\color{Gray}\circ} \ar@[Gray]@{-}[ld]
\\
& {\color{Gray}\circ}
}}
\atop
\adhpt{\pL}=\{\pL,\pLs\}
}
\qquad\quad
{
{\xymatrix@R=.2em@C=1em@H=.5em{
\bullet\ar@[Gray]@{-}[d]\ar@{-}[rd] && \bullet \ar@{-}[ld]\ar@[Gray]@{-}[d]
\\
{\color{Gray}\circ} \ar@[Gray]@{-}[rd] & \bullet \ar@[Gray]@{-}[d] & {\color{Gray}\circ}\ar@[Gray]@{-}[ld]
\\
& {\color{Gray}\circ}
}}
\atop
\adhpt{\pMs}=\{\pLs,\pMs,\pNs\}
}
\qquad\quad
{
{\xymatrix@R=.2em@C=1em@H=.5em{
{\color{Gray}\circ} \ar@[Gray]@{-}[d]\ar@[Gray]@{-}[rd] && \bullet \ar@[Gray]@{-}[ld]\ar@{-}[d]
\\
{\color{Gray}\circ} \ar@[Gray]@{-}[rd] & {\color{Gray}\circ} \ar@[Gray]@{-}[d] & \bullet\ar@[Gray]@{-}[ld]
\\
& {\color{Gray}\circ}
}}
\atop
\adhpt{\pN}=\{\pN,\pNs\}
}
\end{equation}
that are `special' in that they are not intersections of larger irreducibles.
\goodbreak

By~\cite{balmer:spectrum}, the lattice~\eqref{eq:lattice-closed} of closed subsets of~$\SpcK$ classifies the \emph{tt-ideals} of~$\cK$, \ie the thick triangulated $\otimes$-ideal subcategories of~$\cK$. Consequently:
\begin{Cor}
\label{Cor:14-tt-ideals}%
There are precisely fourteen tt-ideals in $\cK=\DATM(\bbR;\bbZ/2)$.
\end{Cor}
\goodbreak

Every object~$M$ in a tt-category~$\cK$ has a \emph{support}, $\supp(M)\subseteq\SpcK$, which is a closed subset of the spectrum. This yields an equivalence relation on objects: $M\sim M'$ when $\supp(M)=\supp(M')$. This happens if and only if $M$ and $M'$ generate the same tt-ideal $\ideal{M}=\ideal{M'}$, that is, $M$ and~$M'$ can be constructed from one another using the tensor-triangular structure of~$\cK$. In this light, the above corollary implies that there are precisely 14 equivalence classes of real Artin-Tate motives with $\bbZ/2$-coefficients, one for each closed subset given in~\eqref{eq:lattice-closed}.

For some of those 14 closed subsets~$Z$ of~$\SpcK$ it is easy to construct a representative~$M\in\cK$ whose support is~$Z$. As always, $\varnothing$ is the support of zero and $\SpcK$ is the support of the $\otimes$-unit $\unit=\mot(\Spec(\bbR))$. Also, if we have objects~$M_1$ and~$M_2$ realizing $Z_1=\supp(M_1)$ and~$Z_2=\supp(M_2)$, then we immediately have objects realizing their union $Z_1\cup Z_2=\supp(M_1\oplus M_2)$ and their intersection $Z_1\cap Z_2=\supp(M_1\otimes M_2)$. Hence there are three `special' Artin-Tate motives to describe, namely motives whose supports are the three special irreducible closed subsets of~\eqref{eq:3-irreduciblos}, those that are not intersections of larger irreducibles.

\subsection*{Three special Artin-Tate motives}

(See \Cref{sec:Spc-DATM(R;Z/2)}.) The right-most irreducible $\adhpt{\pN}=\{\pN,\pNs\}$ in~\eqref{eq:3-irreduciblos} is the support of the motive of the complex numbers
\[
\mot(\Spec(\bbC)).
\]
In other words, the two points $\{\pN,\pNs\}$ are exactly (the homeomorphic image of) the spectrum of the complex version of~$\cK$, that is, the tt-category of mod-2 complex Tate motives~$\DTM(\bbC;\bbZ/2)$; the spectrum of the latter was shown in~\cite{gallauer:tt-dtm-algclosed} to be a Sierpi\'nski space, \ie a space with two points, one closed~($\pNs$) and one open~($\pN$). This `geometric' part is marked by the blue box in~\eqref{eq:subdivision} below. The four `non-geometric' points $\pM,\pMs,\pL,\pLs$ can be called `arithmetic'.

\begin{equation}
\label{eq:subdivision}%
\vcenter{\xy
(0,0)*{\pLs};
(0,-14)*{\pL};
(15,-8)*{\pMs};
(15,-20)*{\pM};
(30,0)*{\pNs};
(30,-14)*{\pN};
{\ar@{-} (0,-12)*{};(0,-2)*{}};
{\ar@{-} (15,-18)*{};(15,-10)*{}};
{\ar@{-} (30,-12)*{};(30,-2)*{}};
{\ar@{-} (11,-7)*{};(3,-2)*{}};
{\ar@{-} (11,-19)*{};(3,-14)*{}};
{\ar@{-} (18,-7)*{};(27,-2)*{}};
{\ar@{-} (18,-19)*{};(27,-14)*{}};
{\ar@[Green]@{-} (-5,3)*{};(35,3)*{};};
{\ar@[Green]@{-} (-5,3)*{};(-5,-11)*{};};
{\ar@[Green]@{-} (35,3)*{};(35,-11)*{};};
{\ar@[Green]@{-} (-5,-11)*{};(35,-11)*{};};
{\ar@{..} (-5,-5)*{};(-14,-5)*{\mathrm{Pure\ }};};
{\ar@[Blue]@{-} (23,5)*{};(23,-17)*{};};
{\ar@[Blue]@{-} (23,5)*{};(37,5)*{};};
{\ar@[Blue]@{-} (23,-17)*{};(37,-17)*{};};
{\ar@[Blue]@{-} (37,5)*{};(37,-17)*{};};
{\ar@{..}@/^.5em/ (30,-17)*{};(45,-20)*{\mathrm{\ Geometric}};};
\endxy}
\end{equation}
\smallskip

The second irreducible $\adhpt{\pMs}=\{\pLs,\pMs,\pNs\}$ appearing in~\eqref{eq:3-irreduciblos} and isolated in the horizontal green box of~\eqref{eq:subdivision} is the support of a generalized Koszul object
\[
\Kos(\beta,\rho):=\cone(\beta)\otimes \cone(\rho).
\]
Here $\beta\colon \unit\to \unit(1)$ is the (motivic) Bott element of~\cite{levine:bott,haesemeyer-hornbostel:bott}, that is, the non-trivial element $-1$ in the motivic cohomology group $\Hm^{0,1}(\bbR;\bbZ/2)\cong \mu_2(\bbR)=\{\pm1\}$. The map $\rho\colon \unit\to \unit(1)[1]$ is the non-trivial element in the other weight-one motivic cohomology group $\Hm^{1,1}(\bbR;\bbZ/2)\cong K^{\mathrm{M}}_1(\bbR)/2=\bbR^\times/(\bbR^\times)^2$, induced by a morphism $\Spec(\bbR)\to\bbG_{\mathrm{m}}$ corresponding to a negative real number, see~\cite{bachmann:real}. The 3-point irreducible subset $\adhpt{\pMs}$ is also the image on spectra of a `semi-simplification' functor $\cK\to \Kb(\AM(\bbR;\bbZ/2))$ taking values in the homotopy category of complexes of pure Artin motives, and closely related to the weight complex functor of Bondarko~\cite{bondarko-weight,wildeshaus:artin-tate}. Hence the label `pure' in~\eqref{eq:subdivision}. The other three points $\pL,\pM,\pN$ are genuinely `mixed'.

Finally, consider the left-most irreducible $\adhpt{\pL}=\{\pL,\pLs\}$ in~\eqref{eq:3-irreduciblos}. In the category of finite correspondences over~$\Spec(\bbR)$, the object $\Spec(\bbC)$ is $\otimes$-selfdual and admits a map $\eta\colon\unit\to \Spec(\bbC)$ dual to the structure morphism $\epsilon\colon\Spec(\bbC)\to\Spec(\bbR)=\unit$. The composition $\epsilon\circ\eta$ is multiplication by~\mbox{$[\bbC:\bbR]$}, hence vanishes with $\bbZ/2$-coefficients. The corresponding complex
\begin{equation}
  \label{eq:fund-motivic}
  \fund_0=\qquad\cdots\to 0\to \unit\xto{\eta}\Spec(\bbC)\xto{\epsilon}\unit\to 0\to\cdots
\end{equation}
can be viewed as an object of~$\cK$ and its support is $\{\pL,\pLs\}$. This~$\fund_0$ is closely related to the non-trivial element of the Picard group discovered in~\cite{pohu:pic-motivic}. Namely, consider the affine quadric $Q \colon x^2 + y^2 = 1$ and its reduced motive $\tilde{\mot}(Q)$. This invertible motive is Artin-Tate and comes with a morphism $\tilde{\epsilon}:\tilde{\mot}(Q)\to \tilde{\mot}(\bbG_\mathrm{m})$, whose cone is nothing but the complex~\eqref{eq:fund-motivic}, up to one Tate twist.

\medbreak

Several interesting computations of spectra follow from \Cref{Thm:main-intro}.

\subsection*{Artin motives and Tate motives}

In Voevodsky's category $\DM(\bbR;\bbZ/2)$, we have the tt-subcategories of Tate motives $\DTM(\bbR;\bbZ/2)$, or of Artin motives $\DAM(\bbR;\bbZ/2)$, \ie the tt-subcategories of our $\cK$ generated by only the Tate objects $\bbZ/2(i)$ for $i\in\bbZ$, or by only the Artin motive $\mot(\Spec(\bbC))$, respectively.
\begin{Cor}[\Cref{Thm:Spc-DTM} and \Cref{Thm:Spc-DAM}]%
  The spectra of the tt-categories of real Tate motives and of real Artin motives with $\bbZ/2$-coefficients are respectively
\[
\Spc(\DTM(\bbR;\bbZ/2))=\vcenter{\xymatrix@R=.3em@C=1em{
&{\bullet}
\\
{\bullet} \ar@{-}[ru]
&
& {\bullet} \ar@{-}[lu]
\\
& {\bullet} \ar@{-}[lu] \ar@{-}[ru]}
}\quadtext{and}
\Spc(\DAM(\bbR;\bbZ/2))=\vcenter{\xymatrix@R=.3em@C=1em{
{\bullet}
&& {\bullet}
\\
& {\bullet} \ar@{-}[lu] \ar@{-}[ru]
}}
\]
In particular, the former admits six tt-ideals and is a local tt-category (unique closed point) whereas the latter admits five tt-ideals. (For the continuous maps induced by inclusion, between those spectra and $\Spc(\DATM(\bbR;\bbZ/2))$, see \Cref{Rmd:Spc-inclusion-Artin-Tate}.)
\end{Cor}

\subsection*{Inverting $\beta$ or $\rho$}

We described above the `pure' irreducible $\adhpt{\pMs}=\{\pLs,\,\pMs,\,\pNs\}$ at the top of~$\SpcK$ as the support of a generalized Koszul object $\Kos(\beta,\rho)$. More precisely, this closed subset~$\adhpt{\pMs}$ is the intersection of the following two supports
\begin{align*}
  \supp(\cone(\beta))=
\vcenter{\xymatrix@R=.2em@C=1em@H=.5em{
\bullet \ar@{-}[d]\ar@{-}[rd] && \bullet \ar@{-}[ld]\ar@[Gray]@{-}[d]
\\
\bullet \ar@[Gray]@{-}[rd] & \bullet \ar@[Gray]@{-}[d] & {\color{Gray}\circ}\ar@[Gray]@{-}[ld]
\\
& {\color{Gray}\circ}
}}
&\qquad\textrm{and}&\supp(\cone(\rho))=\vcenter{\xymatrix@R=.2em@C=1em@H=.5em{
\bullet \ar@[Gray]@{-}[d]\ar@{-}[rd] && \bullet \ar@{-}[ld]\ar@{-}[d]
\\
{\color{Gray}\circ} \ar@[Gray]@{-}[rd] & \bullet \ar@[Gray]@{-}[d] & \bullet\ar@[Gray]@{-}[ld]
\\
& {\color{Gray}\circ}
}}
\end{align*}
Thus, inverting $\beta$ or $\rho$ yields in both cases a tt-category whose spectrum is a Sierpi\'nski space (the complements, marked~${\circ-\circ}$ above) corresponding to the points $\{\pM,\pN\}$ and $\{\pM,\pL\}$ respectively. The localization at $\beta$ is \'etale realization
\[
\mathrm{Re}_{\mathrm{\acute{e}t}}:\DATM(\bbR;\bbZ/2)\too\DATM(\bbR;\bbZ/2)[\beta^{-1}]\cong\Db(\bbZ/2[C_2])
\]
where $C_2=\mathrm{Gal}(\bbC/\bbR)$. On the other hand, we call the localization at~$\rho$ the \emph{real realization}, in reference to~Bachmann~\cite{bachmann:real} (who proves that in the context of $\bbA^{\!1}$-homotopy theory inverting~$\rho$ amounts to real realization):
\[
\mathrm{Re}_{\bbR}:\DATM(\bbR;\bbZ/2)\too\DATM(\bbR;\bbZ/2)[\rho^{-1}].
\]
We compute the target as the quotient of Artin motives by the motive of~$\bbC$
\[
  \DATM(\bbR;\bbZ/2)[\rho^{-1}]\simeq\frac{\DAM(\bbR;\bbZ/2)}{\ideal{\mot(\Spec(\bbC))}}\simeq\frac{\Kb(\bbZ/2[C_2])}{\ideal{\bbZ/2[C_2]}}\,.
\]
(The identification $\DAM(\bbR;\bbZ/2)\simeq\Kb(\bbZ/2[C_2])$ goes back to~\cite[\S\,3.4]{Voevodsky00}.) We will also give an arguably more explicit description of this quotient category in terms of filtered $C_2$-representations over $\bbZ/2$. For more details, we refer to \Cref{sec:applications}.

\subsection*{Integral coefficients}
So far, we only discussed real motives with mod-2 coefficients. With integral coefficients, the spectrum of $\DATM(\FF;\bbZ)$ for $\FF$ real closed is essentially determined by the spectrum of~$\DATM(\FF;\bbZ/2)$ and the analogous spectra \emph{over the algebraic closure}~$\bar\FF$. This algebraically closed case was handled in~\cite{gallauer:tt-dtm-algclosed,gallauer:tt-fmod} for finite coefficients, whereas the case of rational coefficients goes back to~\cite{peter:spectrum-damt}. Since the latter is unconditional only for small fields we deduce an unconditional statement only for $\FF$ the field of real algebraic numbers. The expectation is, however, that the same result holds in general, \cf \Cref{Rmd:Spc-DATM(R;Q)-conjecture}.
\begin{Cor}[{\Cref{Thm:Spc-DATM(R;Z)}, \Cref{Rmd:Spc-DATM(R;Q)-conjecture}}]
  \label{Cor:Spc-AT-integral}%
The spectrum of $\DATM(\bbR;\bbZ)$ is the following topological space
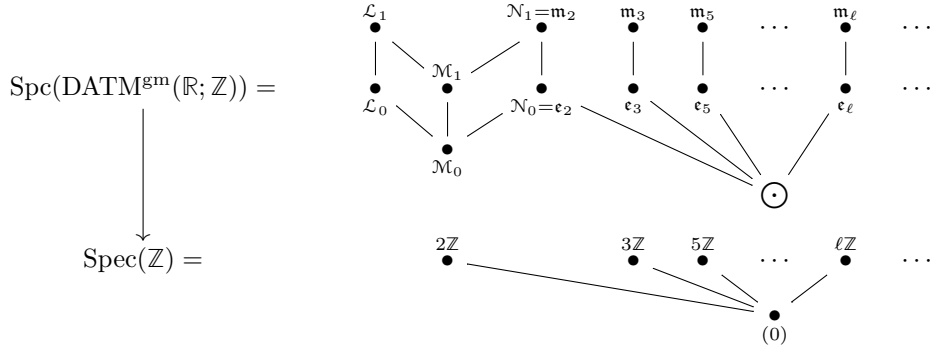
\[
\xymatrix@C=1em@R=.01em{
&& \overset{\pLs}{\bullet}&& \overset{\pNs=\mathfrak{m}_2}{\bullet}&\overset{\mathfrak{m}_3}{\bullet}&\overset{\mathfrak{m}_5}{\bullet}&\cdots&\overset{\mathfrak{m}_\ell}{\bullet}& \cdots
\\
\Spc(\DATM(\bbR;\bbZ))= \ar[dddd]
&&\underset{\pL}{\bullet}\ar@{-}[u]&\overset{\pMs}{\bullet}\ar@{-}[ru]\ar@{-}[lu]&\underset{\pN=\mathfrak{e}_2}{\bullet}\ar@{-}[u]&\underset{\mathfrak{e}_3}{\bullet}\ar@{-}[u]&\underset{\mathfrak{e}_5}{\bullet}\ar@{-}[u]&\cdots&\underset{\mathfrak{e}_\ell}{\bullet}\ar@{-}[u]
& \cdots
\\
&&&\underset{\pM}{\bullet}\ar@{-}[ru]\ar@{-}[lu]\ar@{-}[u]
\\
&&&&&&&\bigodot\ar@{-}[llluu]\ar@{-}[lluu]\ar@{-}[luu]\ar@{-}[ruu]
\\\\
\Spec(\bbZ)=
&&& \overset{2\bbZ}{\bullet} \ar@{-}[rrrrd]
&& \overset{3\bbZ}{\bullet} \ar@{-}[rrd]
& \overset{5\bbZ}{\bullet} \ar@{-}[rd]
& \cdots
& \overset{\ell\bbZ}{\bullet} \ar@{-}[ld]
& \cdots
\\
&&&&&&& \underset{(0)}{\bullet}
}
\]
where $\bigodot$ is the spectrum of the rational category~$\DATM(\bbR;\bbQ)$, which is conjectured to be a point. The same result holds for any real-closed field instead of~$\bbR$, in particular for $\bbR_{\mathrm{alg}}=\bar{\bbQ}\cap \bbR$ in which case $\bigodot$ is indeed known to be a single point.

The canonical comparison map of~\cite{balmer:sss} yields the vertical projection. The six points~$\pLs,\pMs,\pNs,\pL,\pM,\pN$ are mapped to~$2\bbZ$. These primes are the ones of \Cref{Thm:main-intro}, pulled-back under the tt-functor $\DATM(\bbR;\bbZ)\to \DATM(\bbR;\bbZ/2)$. The points~$\mathfrak{m}_\ell$ and $\mathfrak{e}_\ell$ are the ones of~\cite{gallauer:tt-dtm-algclosed}, pulled-back under the tt-functor $\DATM(\bbR;\bbZ)\to \DTM(\bbC;\bbZ/\ell)$. The primes $\mathfrak{m}_\ell$ and $\mathfrak{e}_\ell$ in $\DTM(\bbC;\bbZ/\ell)$ are the kernels of mod-$\ell$ motivic and mod-$\ell$ \'etale cohomology, respectively.
\end{Cor}
\goodbreak

\subsection*{Motivic tt-geometry}
Let us place the results discussed so far within the broader field of motivic tensor-triangular geometry. The goal of the latter is to understand the tt-geometry of motivic tt-categories in general, one prominent example of which is the category of Voevodsky motives $\DM(\FF;k)$. Even though the present work only handles Artin-Tate motives, it already provides a lower bound on the tt-geometric `complexity' of Voevodsky motives in general. Indeed, the inclusion $\DATM(\bbR;\bbZ/2)\hook\DM(\bbR;\bbZ/2)$ induces a \emph{surjection}
\begin{equation*}
  \Spc(\DM(\bbR;\bbZ/2))\onto\Spc(\DATM(\bbR;\bbZ/2))
\end{equation*}
onto the six-point space of \Cref{Thm:main-intro}, by~\cite[Corollary~1.8]{balmer:surjectivity}. In particular, there are at least fourteen tt-ideals in $\DM(\bbR;\bbZ/2)$. A similar surjection holds integrally, hence \Cref{Cor:Spc-AT-integral} sheds some light on the complexity of~$\DM(\bbR;\bbZ)$.

The base fields discussed in this work are real closed. They represent the first foray away from the algebraically closed case discussed in Gallauer~\cite{gallauer:tt-dtm-algclosed} and lead us to filtered representations of the Galois group in positive characteristic, by work of Positselski~\cite{positselski:artin-tate-motives}, as we explain next. Naturally, one may ask about Artin-Tate motives over base fields with more complicated Galois groups than~$C_2=\mathrm{Gal}(\bbC/\bbR)$, such as number fields or finite fields. We plan to attack this problem in future work and anticipate the answer to be substantially more involved. Even this first interaction with modular representation theory is non-trivial, as the reader will see, and produces the intriguing pictures discussed above.

\bigbreak
\tristar

\subsection*{Relation with modular representation theory}
Let us now turn our attention to the representation-theoretic facet of our work. The \'etale realization of real motives with $k=\bbZ/2$-coefficients takes values in $\Db(\cA)$, the bounded derived category of $kC_2$-modules, where $C_2=\mathrm{Gal}(\bbC/\bbR)$ is the absolute Galois group of $\bbR$. Artin-Tate motives admit a functorial weight filtration so that their \'etale realization is endowed with a filtration as well. A remarkable result of Positselski~\cite{positselski:artin-tate-motives} using the norm residue isomorphism theorem (\ie Milnor's Conjecture, here) establishes that, in a suitable sense, one recovers the motive from this purely algebraic information. It means that our triangulated category
\begin{equation}
\label{eq:Positselski-intro}%
\cK=\DATM(\bbR;k)\simeq \DEfil
\end{equation}
is equivalent to the bounded derived category of a slightly tricky \emph{exact} category~$\Efil$ of filtered $kC_2$-modules. The category~$\Efil$ is equivalent to the full subcategory of $\DATM(\bbR;k)$ closed under extensions and generated by the motives $\mot(\Spec(E))(i)$ for $E\in\{\bbR,\bbC\}$ and $i\in\bbZ$, with the exact structure induced from the triangulated structure of $\DATM(\bbR;k)$. Although Positselski's equivalence~\eqref{eq:Positselski-intro} is not known to preserve the tensor product, it preserves `enough' of the tensor structure to imply that the two tt-categories in~\eqref{eq:Positselski-intro} have the same spectrum (\Cref{Prop:positselski}). We discuss Positselski's results, and revisit their proof in our setting, in \Cref{sec:translation}.

\subsection*{Filtered representations}
Purely in representation-theoretic terms, the exact category $\Efil$ can be described as consisting of (finitely) filtered objects~$M$
\begin{equation}
\label{eq:M-filtered}%
\ldots 0 = 0  \subseteq M^n \subseteq M^{n-1} \subseteq \cdots \subseteq M^{m+1} \subseteq M^m = M^{m-1} = \ldots = M
\end{equation}
in the abelian category $\cat{A}=\mmod{kC_2}$ of finitely generated $kC_2$-modules; the exact structure on~$\Efil$ is pulled back from the \emph{split}-exact structure on~$\Asplit=\mmod{kC_2}$ (that is, $\cA$ viewed as only an additive category) via the total-graded functor $\gr\colon \Efil\to \Asplit$ mapping~$M^\sbull$ to $\oplus_{i}M^i/M^{i+1}$. This category~$\Efil$ turns out to be a Frobenius exact category (\Cref{Cor:Afil-Frobenius}), a type of category familiar to modular representation theorists. Details are given in \Cref{sec:filt-repr-setup,sec:filt-Frobenius}.

Thus, on the representation-theory side, \textbf{we now write $\cK$ to mean~$\DEfil$}. The technical heart of the paper consists in proving that $\Spc(\cK)$ has the structure described in \Cref{Thm:main-intro}, \ie the six points with the fourteen closed subsets. This will occupy the principal \Cref{part:I}. We show in particular that the `left-hand' irreducible $\adhpt{\pL}=\{\pL,\pLs\}$ is the support of the complex corresponding to~\eqref{eq:fund-motivic} in $\DEfil$, namely (with the usual non-trivial maps~$\eta$ and~$\eps$)
\begin{equation}
\label{eq:fund_0}%
\fund_0 = \qquad \cdots \to 0\to k \xto{\eta} kC_2 \xto{\eps} k \to 0 \to\cdots
\end{equation}
The above $kC_2$-modules all have the trivial one-step filtration, say, in filtration degree zero. Obviously this complex is exact in the abelian category~$\cat{A}=\mmod{kC_2}$. It is however not split exact, hence it is not exact in the `tricky' exact category~$\Efil$. Its non-exactness explains why the object~$\fund_0\in\cK$ has non-empty support in~$\SpcK$. The authors mistakenly believed for a while that this $\supp(\fund_0)$ was reduced to a single point and that $\SpcK$ had only five points. The discovery that $\supp(\fund_0)$ consists of two points was one of the most delicate parts of the work and led us to isolate the most evasive point~$\pL$.

\subsection*{Proof outline}

We build two tt-functors out of $\cK=\DEfil$ into the \emph{same} target category $\Kb(\cA)$, where $\cA=\mmod{kC_2}$ is now merely the additive category~$\Asplit$
\[
\gr\colon \DE\to \KA
\qquadtext{and}
\tfgt\colon \DE\to \KA\,.
\]
We prove in \Cref{Thm:Spc(Kb(kC_2))} that the spectrum of the target category $\KA$ is
\begin{equation}
\label{eq:Spc(KA)-intro}%
\Spc(\KA)\ = \qquad
\vcenter{\xymatrix@R=.1em{
\overset{\cL}{\bullet} && \overset{\cN}{\bullet}
\\
& \underset{\cM}{\bullet} \ar@{-}[lu] \ar@{-}[ru]
}}
\end{equation}
(The better-known $\Spc(\Db(\mmod{kC_2}))\simeq\Spech(\Hm^\sbull(C_2,\bbZ/2))$ appears as the open $\{\cM,\cN\}$, whereas the projective support variety of~$C_2$, which is well-known to be trivial, $\Spc(\stab(kC_2))=\mathcal{V}_{C_2}(k)=\ast$, appears as the point~$\cM$.) We were informed that \cite{dugger-et-al:C2} independently obtained~\eqref{eq:Spc(KA)-intro} through a different approach.

The images of the 3-point space $\Spc(\KA)$ in~$\SpcK$, under the maps $\Spc(\gr)$ and $\Spc(\tfgt)$, correspond respectively to the following two subsets of~$\SpcK$, one closed (at the top) and one open (at the bottom):
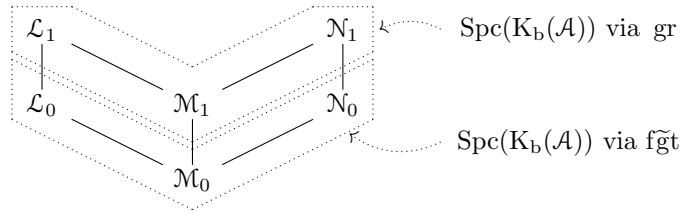
\[
{\xy
(0,0)*{\pLs};
(0,-10)*{\pL};
(20,-10)*{\pMs};
(20,-20)*{\pM};
(40,0)*{\pNs};
(40,-10)*{\pN};
{\ar@{-} (0,-8)*{};(0,-2)*{}};
{\ar@{-} (20,-18)*{};(20,-12)*{}};
{\ar@{-} (40,-8)*{};(40,-2)*{}};
{\ar@{-} (16,-8)*{};(4,-2)*{}};
{\ar@{-} (16,-18)*{};(4,-12)*{}};
{\ar@{-} (24,-8)*{};(36,-2)*{}};
{\ar@{-} (24,-18)*{};(36,-12)*{}};
(70,-15)*{\Spc(\KA)\textrm{ via }\tfgt};
{\ar@{..>}@/^.5em/ (53,-15)*{};(41,-14)*{}};
{\ar@{..} (-4,-4)*{};(20,-16)*{};};
{\ar@{..} (-4,-4)*{};(-4,-12)*{};};
{\ar@{..} (-4,-12)*{};(20,-24)*{};};
{\ar@{..} (44,-4)*{};(20,-16)*{};};
{\ar@{..} (44,-4)*{};(44,-12)*{};};
{\ar@{..} (44,-12)*{};(20,-24)*{};};
(70,0)*{\Spc(\KA)\textrm{ via }\gr};
{\ar@{..>}@/_.3em/ (53,0)*{};(45,0)*{}};
{\ar@{..} (-4,3)*{};(4,3)*{};};
{\ar@{..} (4,3)*{};(20,-5)*{};};
{\ar@{..} (20,-5)*{};(36,3)*{};};
{\ar@{..} (36,3)*{};(44,3)*{};};
{\ar@{..} (44,3)*{};(44,-3)*{};};
{\ar@{..} (44,-3)*{};(20,-15)*{};};
{\ar@{..} (20,-15)*{};(-4,-3)*{};};
{\ar@{..} (-4,-3)*{};(-4,3)*{};};
\endxy}
\]
Let us say a word about those two tt-functors~$\gr$ and~$\tfgt\colon \DE\to \KA$. First, $\gr$ is induced by the exact total-graded functor $\gr\colon \Efil\to \cA$ already discussed at the level of exact categories. Note how the special exact structure on~$\Efil$, pulled back from the split one on~$\cA$, allows $\gr$ to land in $\KA$ instead of the less informative~$\DA$. The second tt-functor $\tfgt\colon \DE\to \KA$ is more mysterious. As the notation suggests, it is related to the functor $\fgt\colon \Efil\to \cA$ that `forgets' the filtration (\ie takes $M^\sbull$ as in~\eqref{eq:M-filtered} to the underlying object~$M$) but this is only true with a twist. Indeed, $\fgt\colon \Efil\to \cA$ is only exact when $\cA$ is viewed as an abelian category, hence induces a functor $\fgt\colon \DE\to \DA$. Our functor $\tfgt\colon \DE\to \KA$ lifts this~$\fgt$ along $\KA\onto \DA$. Its construction involves twisting objects of~$\DE$ by sufficiently large powers of a special $\otimes$-invertible object, take the effective part, and untwist in~$\KA$. The precise definition is a little too technical for this introduction and will be explained in \Cref{sec:central}.

Having identified six points of~$\SpcK$, the remaining critical step consists in proving that these are indeed all the points (\Cref{Thm:Spc-DE}). This will rely on the tt-functor $\gr\colon \DEfil\to \Kb(kC_2)$ detecting the nilpotence of certain morphisms (\Cref{Lem:gr-nil}), expanding on the methods of~\cite{balmer:surjectivity} (\cf \Cref{Cor:surjectivity-support}).

As a final comment, we indicate that all points of $\SpcK$ come equipped with a `residue field functor' into a suitable `tt-field' in the sense of~\cite{balmer-krause-stevenson:ruminations}.

\subsection*{Acknowledgments}
We are thankful to Tom Bachmann, Bernhard Keller, Henning Krause, Peter Symonds and Burt Totaro for precious comments and references.

\bigbreak\goodbreak
\section{Background and notation}
\label{sec:basics}%
\medbreak

%
\begin{Remind}
\label{Rmd:general}%
By a \emph{tensor category}~$\cA$ we mean an additive symmetric monoidal category whose monoidal structure $\otimes\colon \cA\times\cA\to \cA$ is additive in each variable. The $\otimes$-unit is usually denoted~$\unit$. Such a $\otimes$-category~$\cA$ is called \emph{rigid} if every object~$x\in\cA$ admits a \emph{dual} $x^\vee\in\cA$ such that $x\otimes-\colon\cA\to \cA$ is left adjoint to~$x^\vee\otimes-\colon\cA\to \cA$. Any tensor-functor $F\colon \cA\to \cA'$ automatically preserves rigid objects, with $F(x)^\vee=F(x^\vee)$. Furthermore, if $F$ has a right adjoint~$G\colon \cA'\to \cA$ and $\cA$ is rigid then there is a projection formula (see~\cite[Prop.\,3.2]{fausk-hu-may:adjoints}, for instance)
\begin{equation}
\label{eq:projection-formula-general}%
x\otimes G(x')\cong G(F(x)\otimes x').
\end{equation}

An \emph{exact category} $\cE$ is an additive category together with a distinguished class of so-called \emph{admissible} short exact sequences $\xymatrix@C=1.5em{A \,\ar@{>->}[r]|(.45){f} & B \ar@{->>}[r]|(.45){g} & C}$; these sequences must be intrinsically exact (\ie $g$ is a cokernel of~$f$ and $f$ a kernel of~$g$), must contain all split exact sequences $A\into A\oplus C\onto C$, must be closed under push-out along any morphism $A\to A'$ (including existence of said push-out) and must be closed under pull-back along any morphism $C'\to C$, and finally must be such that admissible monomorphisms ($\into$) are closed under composition, and admissible epimorphisms ($\onto$) as well. See~\cite[App.\,A]{keller:chain-stable}. An \emph{exact functor} between exact categories is an additive functor which preserves admissible short exact sequences. An object $A$ is called \emph{projective} (respectively \emph{injective}) if the functor $\Hom_{\cE}(A,-):\cE\to\AbGrps$ (respectively the functor $\Hom_{\cE}(-,A):\cE\op\to\AbGrps$) is exact. A \emph{tensor-exact} category is one such that $\otimes\colon \cE\times\cE\to \cE$ is exact in each variable.

A \emph{Frobenius exact category} is an exact category with enough projectives (every object receives an admissible epimorphism from a projective), enough injectives (the dual notion), and the projective and injective objects coincide. Such a category $\cE$ has an associated \emph{stable category} $\stab(\cE)$ which is constructed as the additive quotient by the projective objects, \ie modding out maps that factor via a projective. It is canonically a triangulated category, see~\cite[Thm.\,2.6]{happel:stable-cat}. We write
\begin{equation}
\label{eq:stab}%
\sta\colon \cE\to \stab(\cE)
\end{equation}
for the canonical quotient functor (identity on objects and mapping a morphism to its class). If $\cE$ is moreover a tensor-exact category in which projective-injectives form a $\otimes$-ideal, then $\stab(\cE)$ inherits a unique tensor structure making $\sta\colon\cE\to \stab(\cE)$ into a tensor-functor. In that case, $\stab(\cE)$ is tensor-triangulated.
\end{Remind}

\begin{Not}
\label{Not:homological-algebra}%
Given an additive category~$\cA$, we denote the usual categories of complexes as follows (we use homological indexing because exponents will be reserved for filtration degrees):
  \begin{itemize}
  \item $\Cb(\cA)$: the category of bounded chain complexes in $\cA$; 
  \item $\Kb(\cA)$: the category with same objects but maps up to homotopy.
  \end{itemize}
When $\cA$ is a (rigid) tensor category, $\Kb(\cA)$ is naturally a (rigid) tt-category.
\end{Not}

\begin{Remind}
\label{Rmd:DbE}%
Let $\cE$ be an exact category. Its bounded derived category~\cite{neeman:D(exact)}
\[
\Db(\cE)=\frac{\Kb(\cE)}{\Kbac(\cE)}
\]
is the Verdier quotient of $\Kb(\cE)$ by the thick subcategory~$\Kbac(\cE)$ of \emph{acyclic complexes}, \ie those (complexes homotopy equivalent to\,(\footnote{\,Our exact categories $\cE$ will all be idempotent-complete making this point moot.})) complexes spliced together from admissible short exact sequences. Again, if $\cE$ is (rigid) tensor-exact then $\Db(\cE)$ is a (rigid) tt-category. One key advantage of~$\Db(\cE)$ over~$\Kb(\cE)$ is that a sequence $A\to B\to C$ of bounded complexes in~$\cE$ which is degreewise an admissible exact sequence yields a triangle $A\to B\to C\to A[1]$ in~$\Db(\cE)$.
\end{Remind}

\begin{Prop}
\label{Prop:Rickard}%
Let $\cE$ be a Frobenius exact category and $\ideal{\Proj}\subseteq\Db(\cE)$ the subcategory of perfect complexes (essentially complexes of projectives). Consider the Verdier quotient $\quo\colon \Db(\cE)\onto \Db(\cE)/\ideal{\Proj}$. Then there exists a canonical triangulated equivalence $\stab(\cE)\to\Db(\cE)/\ideal{\Proj}$ making this diagram commute\,:
\[
\xymatrix@R=2em{
\cE \ \ar@{^(->}[r]^-{[0]} \ar@{->>}[d]_-{\sta}
& \Db(\cE) \ar@{->>}[d]^-{\quo}
\\
\stab(\cE) \ar[r]^-{\cong}
& \Db(\cE)/\ideal{\Proj}\,.}
\]
When moreover $\cE$ is tensor-exact and $\Proj$ forms a $\otimes$-ideal in~$\cE$, then the equivalence $\stab(\cE)\isoto \Db(\cE)/\ideal{\Proj}$ is an equivalence of tensor-triangulated categories.
\end{Prop}

\begin{proof}
The composite $\cE\into\Db(\cE)\xrightarrow{\quo}\Db(\cE)/\ideal{\Proj}$ clearly passes to the stable category. The resulting functor is an equivalence by (the general form of) a result of Rickard~\cite{rickard:der-stab}. See for instance~\cite[Ex.\,2.3]{keller-vossieck:sous-der}. The `moreover' case is easy: all functors in sight are tensor-triangulated and $\ideal{\Proj}$ is a tt-ideal of~$\Db(\cE)$.
\end{proof}

\begin{Not}
\label{Not:Sta}%
We shall sometimes denote by $\Sta\colon \Db(\cE)\onto \stab(\cE)$ the composite functor $\Db(\cE)\xonto{\quo}\Db(\cE)/\ideal{\Proj}\overset{\simeq}{\,\leftarrow\,}\stab(\cE)$ of \Cref{Prop:Rickard}.
\end{Not}

\begin{Remind}[\cite{balmer:spectrum}]
\label{Rmd:Spc}%
The \emph{spectrum} $\SpcK$ of an essentially small tt-category~$\cK$ is the set of tt-ideals~$\cP\subsetneq\cK$ (triangulated subcategories closed under direct summands and tensoring with objects in $\cK$) which are \emph{prime}, meaning that $x\otimes y\in\cP$ forces $x\in\cP$ or~$y\in\cP$. The set~$\SpcK$ admits a topology whose basis of closed subsets are the \emph{supports} $\supp(x)=\SET{\cP}{x\notin\cP}$ of objects~$x\in\cK$. In that topological space~$\SpcK$, the closure of a point~$\cP$ is~$\adhpt{\cP}=\SET{\cQ}{\cQ\subseteq\cP}$. In our pictures, we denote the specialization relation $\cQ\in\adhpt{\cP}$ by a vertical(ish) line as we did in the introduction
\[
\xymatrix@R=1em{
\cQ\ar@{-}[d]
\\
\cP
}
\]
The support of a tt-ideal $\cJ\subseteq\cK$ is~$\supp(\cJ)=\cup_{a\in \cJ}\supp(a)=\SET{\cQ}{\cJ\not\subseteq\cQ}$. All (radical) tt-ideals $\cJ\subseteq\cK$ are classified by their support. Every tensor-triangulated functor $F\colon \cK\to \cL$ induces a continuous map $\varphi=\Spc(F)\colon \Spc(\cL)\to \SpcK$ sending~$\cQ$ to $F\inv(\cQ)$. It satisfies $\varphi\inv(\supp_{\cK}(x))=\supp_{\cL}(F(x))$ for all~$x\in\cK$.
\end{Remind}

\begin{Rem}
\label{Rmd:Spc-quo}%
We shall use several times that if $\cJ\subset\cK$ is a tt-ideal with Verdier quotient $\quo\colon \cK\onto \cK/\cJ$ then the map $\Spc(\quo)\colon \Spc(\cK/\cJ)\to \SpcK$ given by $\cQ\mapsto \quo\inv(\cQ)$ defines a homeomorphism between $\Spc(\cK/\cJ)$ and the subspace $\SET{\cP\in\SpcK}{\cJ\subseteq\cP}$. In particular, if $\cJ=\ideal{x}$ is the tt-ideal generated by $x$ then the subspace is the open $U(x)=\SET{\cP}{x\in \cP}$, complement of~$\supp(x)$. See~\cite[Prop.\,3.11]{balmer:spectrum}.
\end{Rem}

We shall use the following tt-geometric fact of independent interest.
\begin{Prop}
  \label{Prop:Spc-map-generators}%
  Let $F:\cK\to\cL$ be a tt-functor and let $\varphi=\Spc(F)$ the induced map on spectra~$\varphi\colon\Spc(\cL)\to\SpcK$. Let $\cQ\in\Spc(\cL)$ be a prime that is generated by the image under $F$ of a set $S$ of objects in $\cK$, \ie $\cQ=\ideal{F(S)}$.
  \begin{enumerate}[\rm(a)]
  \item\label{it:Spc-map-generators-a} If $\cQ'\in\Spc(\cL)$ and $\cQ\not\subseteq\cQ'$ then $\varphi(\cQ)\not\subseteq\varphi(\cQ')$.
  \item\label{it:Spc-map-generators-b} If $\varphi$ is surjective {\rm(}at least onto~$\SET{\cP\in\SpcK}{S\subseteq\cP}${\rm)} then $\varphi(\cQ)=\ideal{S}$.
  \end{enumerate}
\end{Prop}

\begin{proof}
For~\eqref{it:Spc-map-generators-a}, we have $F(S)\not\subseteq\cQ'$ by assumption and therefore $S\not\subseteq\varphi(\cQ')$. On the other hand, we do have $S\subseteq\varphi(\cQ)$. For~\eqref{it:Spc-map-generators-b}, since $S\subseteq\varphi(\cQ)$, it suffices to show that $\varphi(\cQ)\subseteq\ideal{S}$. By~\cite[Lemma~4.8, Theorem~4.10]{balmer:spectrum}, we have
  \begin{equation*}
    \ideal{S}=\bigcap_{S\subseteq\cP}\cP
  \end{equation*}
where the $\cP$ are prime ideals in $\cK$. So it suffices to prove the following claim:
  \begin{equation}
    \text{If }S\subseteq\cP\text{ then }\varphi(\cQ)\subseteq\cP.\label{eq:claim-PQ}\tag{$\ast$}
  \end{equation}
  So let $\cP\in\Spc(\cK)$ with $S\subseteq\cP$. By our assumption, $\cP=\varphi(\cQ')$ for some $\cQ'\in\Spc(\cL)$. We deduce from $S\subseteq\cP=F\inv(\cQ')$ that $F(S)\subseteq\cQ'$ and, as $\cQ$ is generated by $F(S)$, also $\cQ\subseteq\cQ'$. We conclude that $\varphi(\cQ)\subseteq\varphi(\cQ')=\cP$ as claimed in~$(*)$ above.
\end{proof}

\begin{Rem}
\label{Rem:linear-algebra}%
  Let $M$ be a complex in an additive category~$\cA$, of the following form
  \begin{equation*}
    M=\qquad \cdots\xto{d_{i+2}} M_{i+1}\xrightarrow{d_{i+1}} L\oplus M_i'\xrightarrow{\ d_i\ }L'\oplus M_{i-1}'\xrightarrow{d_{i-1}}M_{i-2}\xto{d_{i-2}}\cdots
  \end{equation*}
  where $d_i$ induces an isomorphism $L\isoto L'$ on the first summands. Then elementary operations show that $M$ is isomorphic to a complex of the form
  \begin{equation*}
    \cdots\xto{d_{i+2}} M_{i+1}\xrightarrow{\smat{0\\d'_{i+1}}} L\oplus M_i'\xrightarrow{\smat{1&0\\0& d'_i}}L\oplus M_{i-1}'\xrightarrow{\smat{0&d'_{i-1}}}M_{i-2}\xto{d_{i-2}}\cdots
  \end{equation*}
Consequently, in $\Hty(\cat{A})$, the complex~$M$ becomes isomorphic to
  \begin{equation*}
    \cdots\xto{d_{i+2}}M_{i+1}\xrightarrow{d'_{i+1}} M_i'\xrightarrow{d'_i}M_{i-1}'\xrightarrow{d'_{i-1}}M_{i-2}\xto{d_{i-2}}\cdots
  \end{equation*}
\end{Rem}

\begin{Rem}
\label{Rem:truncation}%
Recall that a complex~$M=\cdots \xto{d} M_{n+1}\xto{d} M_{n} \xto{d} M_{n-1}\xto{d} \cdots$ admits `stupid' truncations above and below homological degree~$n\in\bbZ$
\[
M_{\ge n}=\cdots \xto{d} M_{n+1}\xto{d} M_{n} \to 0 \to \cdots
\quadtext{and}
M_{\le n}=\cdots 0\to M_{n}\xto{d} M_{n-1} \xto{d} \cdots
\]
together with canonical maps $M_{\le n}\to M$ and $M\to M_{\ge n}$. The functors
\[
(-)_{\ge n}\colon \Chain(\cat{A})\to \Chain(\cat{A})
\qquadtext{and}
(-)_{\le n}\colon \Chain(\cat{A})\to \Chain(\cat{A})
\]
do \emph{not} descend to the homotopy category~$\Komp(\cA)$, for homotopies can cross `over' the degree where truncation occurs, but every~$M$ fits in an exact triangle in~$\Komp(\cat{A})$
\[
M_{\le n} \to M \to M_{\ge n+1} \to M_{\le n}[1]
\]
with the obvious morphisms. This triangle in~$\Komp(\cA)$ is natural in~$M\in\Chain(\cA)$. It will be convenient to rotate this triangle to express $M$ as the cone of a morphism~$\delta$\,:
\begin{equation}
\label{eq:trunc-triangle}%
M_{\ge n+1}[-1] \xto{\delta} M_{\le n} \to M \to M_{\ge n+1}
\end{equation}
where $\delta$ is simply~$d\colon M_{n+1}\to M_n$ in degree~$n$ and (necessarily) zero elsewhere.
\end{Rem}

\begin{Rem}
Notation can be overwhelming in this topic. We tried to be somewhat systematic. We typically use $M,N,\ldots$ for modules and $A,B,\ldots$ for filtered modules. We use $\E$, $\invert$, $\fund$, $\T$, \ldots for special objects or special complexes in the category of filtered modules. We typically use $\Epur$, $\invertpur$ and $\fundpur$, \ldots for their unfiltered (pure) analogues. We also tried to name the many functors that appear with two-to-three-letter names, like $\gr$, $\fgt$, $\quo$, $\sta$, \ldots so that the reader can more easily remember their meaning (`graded', `forget', `quotient', `stable') even under the stress of proof.
\end{Rem}

\bigbreak\goodbreak
\part{Filtered modular representations}
\label{part:I}%
\medbreak

\bigbreak\goodbreak
\section{Homotopy category of \texorpdfstring{$kC_2$}{kC2}-modules}
\label{sec:Spc-Kb(C_2)}%
\medbreak

Let $k$ be a field of characteristic~$2$ and $C_2=\langle{\sigma\mid\sigma^2=1}\rangle$ the group of order~2. Consider the rigid tensor-abelian category of finite-dimensional $kC_2$-modules
\[
\cA=\mmod{kC_2}.
\]
As usual, the tensor~$\otimes$ is over~$k$, with diagonal group action. This tensor is exact in each variable. So the homotopy category~$\Kb(\cA)$ of complexes in~$\cA$ and the derived category~$\Db(\cA)$ are rigid tt-categories (see~\Cref{sec:basics}). Our goal in this preparatory section (\Cref{Thm:Spc(Kb(kC_2))}) is to describe the tt-spectrum of~$\Kb(\cA)$.

\begin{Rem}\label{Rem:stab(A)}%
The indecomposable objects in the Krull-Schmidt category~$\cA$ are the trivial representation~$k$ and the free one~$kC_2$. The stable category $\stab(\cA)$ is tt-equivalent to $\mmod{k}$ via the composite $\mmod{k}\hook\cA\onto \stab(\cA)$. Those categories~$\mmod{k}$ and~$\stab(kC_2)$ are triangulated with the identity as suspension $\Sigma=\Id$, and trivial (semi-simple) triangulation. Their spectrum contains just one point,~$(0)$.
\end{Rem}

\begin{Rem}
\label{Rmd:Spc(Db(A))}%
The cohomology ring $\Hm^\sbull(C_2,k)=\Homcat{\Db(\cA)}(k,k[\bullet])=\Ext^\sbull_{kC_2}(k,k)$ is isomorphic to~$k[\fundpur]$, with generator $\fundpur\in \Ext^1_{kC_2}(k,k)$ the non-trivial extension:
  \begin{equation}
    \label{eq:extension-S}%
    k\xinto{\eta}kC_2\xonto{\epsilon}k.
  \end{equation}
When viewed as a complex in~$\cA$, we place $\fundpur$ in homological degrees 2, 1 and~0:
\begin{equation}
\label{eq:fund-pur}%
\fundpur = \qquad \cdots 0 \to 0\to k \xto{\eta} kC_2 \xto{\eps} k \to 0 \to 0 \cdots
\end{equation}

The spectrum $\Spc(\Db(\mmod{kG}))$ is known for any finite groups~$G$ to be homeomorphic to $\Spech(\Hm^\sbull(G,k))$, see~\cite[Prop.\,8.5]{balmer:sss}. For $G=C_2$, the spectrum $\Spech(k[\fundpur])$ has two points and one can easily prove directly that
\[
\Spc(\Db(\cA))=\{(0),\ideal{kC_2}\}.
\]
Indeed, the two tt-ideals $(0)$ and $\ideal{kC_2}$ are prime because they are the kernels of the following tt-functors (see \Cref{Not:Sta} for the second one):
\[
\res^{C_2}_1\colon \Db(\cA)\to \Db(k)
\qquadtext{and}
\Sta\colon \Db(\cA)\onto \stab(\cA)\cong\mmod{k}.
\]
If $\cJ\subseteq\Db(\cA)$ contains $M$ non-zero then $\res^{C_2}_1M$ remains non-zero in~$\Db(k)$, hence admits $k[i]$ as a direct summand, for some $i\in \bbZ$. Then $kC_2\otimes M\cong \ind_1^{C_2}\res^{C_2}_1 M$ admits $kC_2[i]$ as a summand, thus $\cJ\supseteq\ideal{kC_2}$. By \Cref{Rem:stab(A)}, $\Spc(\Db(\cA)/\ideal{kC_2})=\Spc(\stab(\cA))=\{(0)\}$. So $\ideal{kC_2}$ is indeed the only non-zero prime of~$\Db(\cA)$.
\end{Rem}

Applying \Cref{Rmd:Spc-quo} to $\cK=\Kb(\cA)$ and~$\cJ=\Kbac(\cA)$ the tt-ideal of acyclic complexes, we obtain by definition of~$\Db(\cA)=\Kb(\cA)/\Kbac(\cA)$ the following:

\begin{Prop}\label{Prop:Der-open-piece}
  The Verdier quotient functor $\Kb(\cA)\onto\Db(\cA)$ induces a homeomorphism between $\Spc(\Db(\cA))$ and the subspace $\{\cP\in\Spc(\Kb(\cA))\mid \Kbac(\cA)\subset\cP\}$ of $\Spc(\Kb(\cA))$. The complement of that subspace is~$\supp(\Kbac(\cA))$.\qed
\end{Prop}

So we want to understand $\supp(\Kbac(\cA))$, the support of acyclic complexes.

\begin{Def}\label{Def:simple-ideal}
  A tt-ideal $\cJ$ in a tt-category is \emph{simple} if any non-zero object $x\in\cJ$ generates $\cJ$ as a tt-ideal. In other words, its only sub-tt-ideals are~$0$ and~$\cJ$.
\end{Def}

\begin{Lem}\label{Lem:simple-ideal-support} Let $\cJ\subset\cK$ be a simple tt-ideal in a tt-category $\cK$. Then the support $\supp(\cJ)=\SET{\cP}{\cJ\not\subseteq\cP}$ of $\cJ$ is either empty or a single closed point.
\end{Lem}

\begin{proof}
Let $\cP_1,\cP_2\in\supp(\cJ)$. Pick $x\in \cJ\oursetminus\cP_2$. Since $\cJ$ is simple, $\cJ\not\subseteq\cP_1$ forces $\cJ\cap \cP_1=0$. For all $y\in \cP_1$ we have $x\otimes y\in \cJ\cap \cP_1=0\subseteq\cP_2$ and $x\notin\cP_2$ forces $y\in \cP_2$. So $\cP_1\subseteq\cP_2$ for any $\cP_1,\cP_2\in\supp(\cJ)$. Hence also $\cP_2\subseteq\cP_1$ and thus $\cP_1=\cP_2$.
\end{proof}

\begin{Prop}
\label{Prop:acyclics-simple}%
The tt-ideal $\Kbac(\cA)$ of acyclics in~$\Kb(\cA)$ is simple. In particular, $\Kbac(\cA)=\ideal{\fundpur}$ is generated by the complex~$\fundpur$ in~\eqref{eq:fund-pur}.
\end{Prop}

Let us begin with a preparation.

\begin{Lem}
\label{Lem:<t>=acyclics}%
Consider the following morphism of complexes $\tilde\eps$ in~$\cA$
\[
\xymatrix@R=1.5em{
\invertpur:=\fundpur_{\ge1}[-1]= \ar@<-.5em>[d]_-{\tilde\eps}
&& \cdots \ar[r]
& 0 \ar[r] \ar[d]
& k \ar[r]^-{\eta} \ar[d]
& kC_2 \ar[r] \ar[d]^-{\eps}
& 0\ar[r] \ar[d]
& \cdots
\\
\unit=
&& \cdots \ar[r]
& 0 \ar[r]
& 0 \ar[r]
& k \ar[r]
& 0\ar[r]
& \cdots
}
\]
Let $M\in\Kbac(\cat{A})$ be an acyclic complex. Then the map~$\tilde\eps\colon \invertpur\to \unit$ is $\otimes$-nilpotent on~$M$, that is, there exists $\ell\gg0$ such that $\tilde\eps\potimes{\ell}\otimes M=0$ in~$\Hty(\cat{A})$.
\end{Lem}

\begin{proof}
Since~$\invertpur=\cone(\eta\colon k\to kC_2)$, we have an exact triangle in~$\Hty(\cat{A})$
\begin{equation}
\label{eq:aux-V-triangle}%
k \xto{\eta} kC_2 \to \invertpur\to k[1]\,.
\end{equation}
For every $\ell\ge 1$, the morphism $\tilde\eps\potimes{\ell}\otimes M$ has the following source and target:
\begin{equation}
\label{eq:aux-alpha^l}%
\tilde\eps\potimes{\ell}\otimes M\colon \invertpur\potimes{\ell}\otimes M\to M.
\end{equation}
Since $\res^{C_2}_1M$ is acyclic over the field~$k$, it is zero in~$\Kb(k)$. Hence by Frobenius we have $kC_2\otimes M\cong \ind_1^{C_2}\res^{C_2}_1M=\ind_1^{C_2}0=0$. So tensoring the exact triangle~\eqref{eq:aux-V-triangle} with~$M$, we see that $\invertpur\otimes M\simeq M[1]$ in~$\Hty(\cat{A})$. By induction on~$\ell$, we have
\begin{equation}
\label{eq:aux-BM}%
\invertpur\potimes{\ell}\otimes M\simeq M[\ell].
\end{equation}
Now since $M$ is bounded, there exists $\ell_0$ large enough so that for all~$\ell\ge \ell_0$
\begin{equation}
\label{eq:aux-M[l]M}%
\Hom_{\Hty(\cat{A})}\big(M[\ell],M\big)=0.
\end{equation}
It then follows from~\eqref{eq:aux-alpha^l},~\eqref{eq:aux-BM} and~\eqref{eq:aux-M[l]M} that $\tilde\eps\potimes{\ell}\otimes M=0$ for all~$\ell\ge\ell_0$.
\end{proof}

\begin{proof}[Proof of \Cref{Prop:acyclics-simple}]
Let $M,N$ be acyclic complexes with $N$ homotopically non-trivial. We need to show that $M\in \ideal{N}$ in~$\Kb(\cA)$. We can assume that~$N$ lives in the following degrees $N=\cdots 0\to N_n\to \cdots \to N_1\to N_0\to 0\cdots$ with $N_0\neq 0$ and that $N$ contains no contractible direct summand in~$\Cb(\cA)$. Hence $N_0$ has no projective summand, for otherwise we could split off a summand isomorphic to $\cdots 0\to kC_2\xto{=}kC_2\to 0\cdots$ by \Cref{Rem:linear-algebra}, contradicting our assumption on~$N$. Similarly, if we decompose $N_1\simeq P\oplus T$ with $P$ projective and $T$ without projective summand then the differential $d_1\colon P\oplus T=N_1\to N_0$ is of the form~$(q\ 0)$. Indeed, here we use that the group is~$C_2$. Both~$T$ and~$N_0$ have no projective summand, hence must have trivial $C_2$-action. Any non-zero map $T\to N_0$ would then yield, by \Cref{Rem:linear-algebra} again, a summand of~$N$ isomorphic to $\cdots 0\to k\xto{=} k\to 0\cdots$, which we have excluded. In summary, our acyclic complex~$N$ has the following form:
\[
N=\qquad\cdots 0\to N_n\to \cdots \to N_2\to P\oplus T\xto{(q\ 0)} N_0 \to 0\cdots
\]
where $P$ is projective and where the $C_2$-action on $N_0\simeq k^s$ is trivial.

Consider the tensor of the exact complex~$\fundpur$ of~\eqref{eq:fund-pur} with the object~$N_0$. This yields the second row in the following commutative diagram, whose first row is~$N$\,:
\[
\xymatrix@R=2em{
N= \ar@{..>}@<-.7em>[d]^-{\phi}
& \cdots \ar[r]^-{} 
& N_3 \ar[r]^-{d} \ar[d]
& N_2 \ar[r]^-{d} \ar@{..>}[d]_-{\exists}^-{g}
& P \oplus T \ar[r]^-{\smat{q&0}} \ar@{..>}[d]_-{\exists}^-{\smat{f& 0}}
& N_0 \ar[r]^-{} \ar@{=}[d]
& 0 \ar[r]^-{} \ar[d]
& {\cdots}
\\
\fundpur\otimes N_0=
& {\cdots}\ar[r]
& 0 \ar[r]^-{}
&  N_0 \ar@{ >->}[r]^-{}
& kC_2\otimes N_0 \ar@{->>}[r]^-{\eps\otimes 1}
& N_0 \ar[r]^-{}
& 0 \ar[r]^-{}
& {\cdots}
}
\]
From the right, we construct a morphism $\phi\colon N\to \fundpur\otimes N_0$ in~$\Cb(\cA)$ between those two exact complexes. Since $\eps\otimes 1\colon kC_2\otimes N_0\to N_0$ is onto and $P$ is projective, there exists $f\colon P\to kC_2\otimes N_0$ that lifts~$q$, \ie such that $(\eps\otimes 1)f=q$. This makes the above right-hand square commute. The existence of~$g$ then simply follows from exactness of the bottom sequence. Applying the stupid truncations~\eqref{eq:trunc-triangle} to this morphism of complexes~$\phi$ in~$\Cb(\cA)$, we get two exact triangles in~$\Kb(\cA)$ and a morphism of triangles (where $N_0$ is $N_0[0]=N_{\le0}$ and $(\fundpur\otimes N_0)_{\ge 1}=\fundpur_{\ge 1}\otimes N_0$)\,:
\[
\xymatrix@R=2em{
N_{\ge 1}[-1] \ar[r]^-{\delta} \ar[d]^{\phi_{\ge 1}[-1]}
& N_0 \ar[r] \ar@{=}[d]
& N \ar[r] \ar[d]^{\phi}
& N_{\ge 1} \ar[d]^{\phi_{\ge 1}}
\\
\fundpur_{\ge 1}[-1]\otimes N_0 \ar[r]^-{\tilde\eps\otimes 1}
& N_0 \ar[r]
& \fundpur \otimes N_0 \ar[r]
& \fundpur_{\ge 1}\otimes N_0
}
\]
Now we have seen in \Cref{Lem:<t>=acyclics} that $\tilde\eps\colon \fundpur_{\ge1}[-1]\to \unit$ is $\otimes$-nilpotent on any acyclic. Hence so is $\tilde\eps\otimes 1$ and by the above left-hand commutative square, so is~$\delta$. Therefore for our acyclic~$M$ we have $\delta\potimes{\ell}\otimes M=0$ in~$\Kb(\cA)$ for some $\ell\gg0$. The cone of this zero morphism $\delta\potimes{\ell}\otimes M\colon N_{\ge1}[-1]\potimes{\ell}\otimes M \to N_0\potimes{\ell}\otimes M$ contains $N_0\potimes{\ell}\otimes M\simeq M^{\oplus s\ell}$ as a direct summand, hence~$M$ as well: $M\in\ideal{\cone(\delta\potimes{\ell})}\subseteq\ideal{\cone(\delta)}=\ideal{N}$.
\end{proof}

\begin{Thm}
\label{Thm:Spc(Kb(kC_2))}%
  The spectrum of $\Hty(\mmod{kC_2})$ is the following 3-point topological space, where the complex $kC_2$ is concentrated in degree zero and $\fundpur$ is as in~\eqref{eq:fund-pur}.
  \begin{equation*}
    \vcenter{\xymatrix@C=.5em@R=1em{\cL=\ideal{kC_2} \ar@{-}[rd]
    && \cN=\ideal{\fundpur} \ar@{-}[ld]
    \\
    & \cM=\ideal{\fundpur,kC_2}}}
  \end{equation*}
  So $\cM$ is a generic point, whereas $\{\cL\}=\supp(\fundpur)$ and $\{\cN\}=\supp(kC_2)$ are closed.
\end{Thm}

\begin{proof}
As before we write $\cA$ for~$\mmod{kC_2}$. In~\Cref{Prop:acyclics-simple}, we proved that $\Kbac(\cA)$ is a simple tt-ideal, hence has support a single closed point by \Cref{Lem:simple-ideal-support}, that we call~$\cL$. On the other hand we proved in~\Cref{Prop:Der-open-piece} that the complement of this single point was $\Spc(\Db(\cA))$, which has two points~$\cM,\cN$ with $\cN\in\adhpt{\cM}$, which means $\cN\subset\cM$ (\Cref{Rmd:Spc}). More precisely, $\cM$ and $\cN$ are the two primes containing~$\Kbac(\cA)$ and they correspond in the quotient~$\Kb(\cA)/\Kbac(\cA)=\Db(\cA)$ to the two primes~$0$ and~$\ideal{kC_2}$ of \Cref{Rmd:Spc(Db(A))}. This gives us the description of~$\cN=\Kbac(\cA)=\ideal{\fundpur}$ and~$\cM=\ideal{\fundpur,kC_2}$. It remains to see that $\cL=\ideal{kC_2}$ and that it is contained in~~$\cM$ but not in $\cN$. The object~$kC_2$ refers here to the complex concentrated in degree zero hence it is not acyclic. Thus its support is disjoint from~$\supp(\Kbac(\cA))=\{\cL\}$. This reads $kC_2\in\cL$ or $\ideal{kC_2}\subseteq\cL$. As $kC_2\notin\Kbac(\cA)=\cN$, this proves already that $\cL\not\subset\cN$. If we had $\cL\not\subset\cM$ as well then $\Spc(\Kb(\cA))$ would be disconnected, as $\{\cL\}\sqcup\{\cM,\cN\}$. This would force the rigid tt-category~$\Kb(\cA)$ to be the product of two tt-categories, which is excluded for many reasons, for instance $\End_{\Kb(\cA)}(\unit)\cong k$ being indecomposable. So we have indeed $\cL\subset\cM$. Finally let us show that $\ideal{kC_2}\subseteq \cL$ is an equality, by considering the supports of those two tt-ideals, \ie the primes not containing them. By inspection, we see that both have support~$\{\cN\}$, hence they are equal: $\ideal{kC_2}=\cL$.
\end{proof}

Let us describe `residue field functors' detecting the three points of~$\Spc(\Kb(\cA))$.
\begin{Cor}
\label{Cor:fields-KA}%
Let $\cA=\mmod{kC_2}$ and use notation of \Cref{Thm:Spc(Kb(kC_2))}.
\begin{enumerate}[\rm(a)]
\item
\label{it:res-field-N}%
The following restriction functor is a tt-functor whose kernel is~$\cN=\ideal{\fundpur}$:
\[
\xymatrix@C=4em{
\rsd{\cN}\colon\Kb(\cA) \ar[r]^-{\res^{C_2}_1} & \Kb(k)\cong\Db(k).
}
\]
\smallbreak
\item
\label{it:res-field-M}%
The following localization functor is a tt-functor whose kernel is~$\cM=\ideal{\fundpur,kC_2}$:
\[
\xymatrix{
\rsd{\cM}\colon\Kb(\cA) \ar@{->>}[r]^-{\quo} & \Db(\cA) \ar@{->>}[r]^-{\Sta} & \stab(kC_2)\cong \mmod{k}.
}
\]
\smallbreak
\item
\label{it:res-field-L}%
Consider the additive quotient $\sta\colon\cA=\mmod{kC_2}\onto \stab(kC_2)\cong \mmod{k}$. The following induced functor is a tt-functor whose kernel is the prime~$\cL=\ideal{kC_2}$:
\[
\xymatrix@C=4em{
\rsd{\cL}\colon\Kb(\cA) \ar[r]^-{\Kb(\sta)} & \Kb(\mmod{k})\cong\Db(k).
}
\]
\end{enumerate}
\end{Cor}

\begin{proof}
The verification that these are well-defined tt-functors is easy. As the target categories have spectra reduced to a single prime, namely zero, the kernels of those functors are primes in~$\Kb(\cA)$. To identify which prime it is exactly, it then suffices to compute the image of~$kC_2$ and of~$\fundpur$ under those functors, which is very easy.
\end{proof}

\begin{Rem}
We draw the reader's attention to the slightly unorthodox construction in~\eqref{it:res-field-L}. We consider the stable category $\stab(kC_2)=\mmod{kC_2}/\pproj{kC_2}$ but do not think of it as a triangulated category, just as an additive category, and take its homotopy category of complexes $\Kb(\stab(kC_2))$ as such.
\end{Rem}

\begin{Rem}
\label{Rem:localisations-of-KA}%
Consider two localizations of the homotopy category~$\Kb(\cA)$, namely $\Kb(\cA)/\ideal{kC_2}$ and $\Kb(\cA)/\ideal{\fundpur}$. We already know (\Cref{Prop:acyclics-simple}) that $\KA/\ideal{\fundpur}$ is simply the derived category~$\Db(\cA)$. Using \Cref{Rmd:Spc-quo}, we identify the spectra of those two localizations with open pieces of the spectrum (indicated by the~$\bullet$)
\[
{\xymatrix@C=2em@R=.5em{
\bullet && {\color{Gray}\circ}
\\
& \bullet \ar@{-}[lu] \ar@{-}@[Gray][ru]
}\atop
\Spc\big(\Kb(\cA)/\ideal{kC_2}\big)
}
\qquad\hook\qquad
{\xymatrix@C=2em@R=.5em{
\bullet && \bullet
\\
& \bullet \ar@{-}[lu] \ar@{-}[ru]
}\atop
\Spc\big(\Kb(\cA)\big)
}
\qquad\hookleftarrow\qquad
{\xymatrix@C=2em@R=.5em{
{\color{Gray}\circ} && \bullet
\\
& \bullet \ar@{-}@[Gray][lu] \ar@{-}[ru]
}\atop
\Spc\big(\Db(\cA)\big)
}
\]
or explicitly $\Spc(\KA/\ideal{kC_2})=\{\cL,\cM\}$ and $\Spc(\Db(\cA))=\{\cM,\cN\}$.
\end{Rem}

\begin{Rem}
\label{Rem:res-fields}%
With notation as in \Cref{Cor:fields-KA}, we have a commutative diagram
\begin{equation}
\label{eq:residue-fields-Kb(A)}%
\vcenter{\xymatrix@R=.7em{
&& \Kb(\cA) \ar@{->>}[ldd]^-{\quo} \ar@{->>}[rdd]_-{\quo}
 \ar@/_3em/[lldddd]_-{\Displ\rsd{\cL}} \ar[dddd]|-{\Displ\vphantom{I^I_j}\rsd{\cM}} \ar@/^3em/[rrdddd]^-{\Displ\rsd{\cN}}
\\\\
&\Displ\frac{\Kb(\cA)}{\ideal{kC_2}} \ar@{->>}[rdd] \ar[ldd]^-{\rsd{\cL}'} \ar[rdd]_-{\rsd{\cM}'}
&& \kern-2em \Displ\frac{\Kb(\cA)}{\ideal{\fundpur}}=\Db(\cA) \ar@{->>}[ldd]|-{\ \ \rsd{\cM}''=\Sta} \ar[rdd]_-{\rsd{\cN}''=\res^{C_2}_1}
\\\\
\Db(k)
&& \Displ\frac{\Kb(\cA)}{\ideal{kC_2,\fundpur}}\cong\mmod{k}\kern-3em
&& \Db(k)
\\
\kappa(\cL) \ar@{=}[u]
&& \kappa(\cM) \ar@{=}[u]
&& \kappa(\cN) \ar@{=}[u]
}}
\end{equation}
Let us explain this picture. The functors $\rsd{\cL}$ and $\rsd{\cM}$ vanish on~$kC_2$ hence induce functors $\rsd{\cL}'$ and $\rsd{\cM}'$ as in the left-hand side of~\eqref{eq:residue-fields-Kb(A)}. The functors $\rsd{\cM}$ and $\rsd{\cN}$ vanish on~$\fundpur$ hence induce functors $\rsd{\cM}''$ and $\rsd{\cN}''$ as in the right-hand side of~\eqref{eq:residue-fields-Kb(A)}. The latter coincide with the tt-functors of \Cref{Rmd:Spc(Db(A))}, under the identification $\KA/\ideal{\fundpur}=\DA$, namely $\rsd{\cM}''=\Sta$ and $\rsd{\cN}''=\res^{C_2}_1$.

At the bottom of~\eqref{eq:residue-fields-Kb(A)} we see the target categories $\kappa(\cL)=\Db(k)$, $\kappa(\cM)=\mmod{k}$ and $\kappa(\cN)=\Db(k)$ of the three `residue functors'~$\rsd{\cL}$, $\rsd{\cM}$ and~$\rsd{\cN}$. The derived category $\Db(k)$ that appears at~$\cL$ and $\cN$ is certainly a `tensor-triangular field'; see \cite{balmer-krause-stevenson:ruminations}. The third one, $\kappa(\cM)$, is more mysterious and comes into play as the stable category $\stab(kC_2)=\frac{\mmod{kC_2}}{\pproj{kC_2}}$. This is indeed one of the non-standard tt-fields identified in~\cite{balmer-krause-stevenson:ruminations}, namely $\stab(kC_p)$ for a prime number~$p$. Here however, because $p=2$, this stable category coincides with the very ordinary category of finite dimensional $k$-vector spaces. So the exotic nature of this `tt-field' $\kappa(\cM)$ is somewhat hidden, except for its suspension being the identity.
\end{Rem}

\begin{center}
*\ *\ *
\end{center}

We have achieved our goal of describing $\Spc(\Kb(\mmod{kC_2}))$. We end the section with some technical results about $\cA=\mmod{kC_2}$ that will come handy later on.

\begin{Remind}
\label{Rmd:Galois}%
The object $\Epur=kC_2$ admits a unique associative and commutative multiplication $\mu\colon kC_2\otimes kC_2\to kC_2$ in~$\cA$, mapping $1\otimes 1$ to~$1$ and $1\otimes\sigma$ to zero. (Recall that $C_2=\ideal{\sigma}$.) The map $\eta\colon \unit_{\cA}=k\xto{1+\sigma} \Epur=kC_2$
is a two-sided unit for~$\Epur$. Also $\sigma\in C_2$ acts as an automorphism via $\Epur\xto{\sigma\cdot} \Epur$. In fact, $\Epur$ is a `quasi-Galois' ring-object with Galois group~$C_2$ meaning that the following is an isomorphism
\begin{equation}
\label{eq:Galois}%
\xymatrix@C=6em{
\Epur\otimes \Epur \ar[r]^-{\big(\mu\quad\mu\circ(1\otimes\sigma)\big)}_-{\simeq}
& \Epur\oplus \Epur.
}
\end{equation}
See \cite{Pauwels17}. This isomorphism is just a permutation of the bases, as follows: $1\otimes 1\leftrightarrow(1,0)$, $1\otimes\sigma\leftrightarrow(0,1)$, $\sigma\otimes 1\leftrightarrow(0,\sigma)$ and $\sigma\otimes\sigma\leftrightarrow(\sigma,0)$. In other words, it $k$-linearly extends a bijection of $C_2$-sets $C_2\times C_2\isoto C_2\sqcup C_2$.
\end{Remind}

\begin{Prop}
\label{Prop:L^n}%
The object $\invertpur=\fundpur_{\ge1}[-1]=(\cdots 0\to k \xto{\eta} kC_2 \to 0 \cdots)$ with $k$ in homological degree~1 (see \Cref{Lem:<t>=acyclics}) is $\otimes$-invertible in $\Kb(\cA)$. More precisely, for every $n\in\bbZ$ we have canonical isomorphisms
\[
\invertpur\potimes{n}\cong\left\{
\begin{array}{cl}
\cdots 0 \to k \xto{\eta} kC_2\xto{\eta\eps} \cdots \xto{\eta\eps} kC_2\to 0\cdots \kern8.5em & \textrm{if }n\ge 0
\\
\kern9em \cdots 0 \to kC_2 \xto{\eta\eps} \cdots \xto{\eta\eps} kC_2 \xto{\eps} k\to 0\cdots & \textrm{if }n\le 0
\end{array}
\right.
\]
where in each case $k$ sits in homological degree~$n$ and there are $|n|$ copies of~$kC_2$.
\end{Prop}

\begin{proof}
This follows from the description of~$kC_2\otimes kC_2$ in \Cref{Rmd:Galois}. Alternatively, one can test invertibility in all residue fields of \Cref{Rem:res-fields}. If we pass via $\Kb(\cA)/\ideal{kC_2}$ then $\invertpur$ is just~$\unit[1]$. If on the other hand we pass via $\Db(\cA)$ then $\invertpur$, being a resolution of~$k$, becomes isomorphic to~$\unit$. In any case, it is invertible.
\end{proof}

\begin{Rem}
\label{Rem:v^n}%
To understand how unique the isomorphisms of \Cref{Prop:L^n} are, or how `coherent' they are, note that any two isomorphisms between invertibles differ by multiplication by an automorphism of~$\unit$. For $k=\bbF_2$ there is no non-trivial automorphism of~$\unit$ since $\bbF_2^\times=\{1\}$. So we can construct our isomorphisms canonically over the field $\bbF_2$ and then extend scalars to any field of characteristic~$2$.
\end{Rem}

\begin{Rem}
\label{Rem:Koszul-Kb(A)}%
Already in our toy example of a tt-category~$\cK=\KA$, the two codimension-one irreducible closed subsets $\{\cL\}$ and~$\{\cN\}$ of~$\SpcK$ are `cut out' by a single `equation' $s\colon \unit \to u$ for $u$ some $\otimes$-invertible, that is, they are of the form
\begin{equation}
\label{eq:Z(s)}%
Z(s)=\supp(\cone(s)).
\end{equation}
In algebro-geometric language, they correspond to a `Koszul object' of length one $\mathrm{Kos}(\alpha)=\cone(\unit \xto{\alpha} u)$. Specifically for $\{\cL\}=\supp(\fundpur)$, we have $\{\cL\}=\supp(\cone(\tilde\eta\colon \unit \to \invertpur\potimes{-1}))$ where $\tilde \eta$ is simply~$\eta\colon k \to kC_2$ in degree zero. The other closed point $\{\cN\}=\supp(kC_2)$ equals $\supp(\upsilon\colon \unit \to \invertpur\potimes{-1}[1])$ where $\upsilon\colon \unit \to \invertpur\potimes{-1}[1]$ is given by the identity $k\to k$ in degree zero.
\end{Rem}

\bigbreak\goodbreak
\section{Filtered \texorpdfstring{$kC_2$}{kC2}-modules}
\label{sec:filt-repr-setup}%
\medbreak

After describing in \Cref{sec:Spc-Kb(C_2)} the space $\Spc(\KA)$ for $\cA=\mmod{kC_2}$, we turn to the more substantial problem of understanding filtered objects in~$\cA$. In this section, we define the tensor category $\Afil$ of filtered objects, which is a Krull-Schmidt category, and we describe its indecomposable objects (\Cref{Cor:Afil-krull-schmidt}) and their tensor (\Cref{Prop:tens-e_l}). We discuss exact structures in the next section.

\begin{Def}
\label{Def:filtrations}%
We denote by $\fil{\cA}$ the category of \emph{filtered objects} in $\cA$, \ie sequences
\begin{equation}
\label{eq:filtered-object}
A= \qquad \cdots\into A^{n+1}\into A^n\into A^{n-1}\into\cdots
\end{equation}
of monomorphisms in $\mmod{kC_2}$ such that $\gr^n(A):=A^n/A^{n+1}$ is zero for all but finitely many $n$, and $A^n=0$ for $n\gg 0$. The \emph{underlying object} of $A$ is
\[
\fgt(A)=\colim_{n\to -\infty}A^n
\]
that is, up to isomorphism, $A^{n}$ for $n\ll0$. We often think of a filtered object~$A$ as the underlying object equipped with a finite filtration $\cdots A^{n+1}\subseteq A^n\subseteq \cdots \subseteq A$ as above. A morphism of filtered objects is the obvious degreewise notion, compatible with the inclusions. In other words, we can view $\Afil$ as a full subcategory of the category $\seq{\cA}=\mathrm{Fun}(\bbZ\op,\cA)$ of presheaves from the poset~$(\bbZ,\le)$ to~$\cA$. Alternatively, a morphism $f\colon A\to B$ is simply a morphism of underlying objects such that $f(A^n)\subseteq B^n$ for all~$n\in\bbZ$. We thus have a faithful functor
\[
\fgt\colon \Afil\to \cA
\]
that forgets the filtration. We also have functors $\gr^n\colon \Afil\to \cA$, which assemble to a functor $\gr(A)=\oplus_{n}\gr^n(A)$ called the \emph{total-graded functor}
\[
\gr\colon\Afil \to \cA.
\]

When we discuss complexes in~$\Afil$, we shall have (homological) degrees for complexes and (filtration) degrees of each term of the complex. To avoid confusion with the word `degree', and to follow the motivic tradition, we speak of \emph{weight} to refer to the \emph{filtration degree}. More precisely, we shall call~$A^n$ the elements \emph{of weight at least}~$n$ in~$A\in\fil{\cA}$. Also, we shall try to write complexes differentials horizontally and filtration inclusions vertically (with bigger weights below smaller weights).
\end{Def}

\begin{Exa}
\label{Exa:filt-0}%
Any $kC_2$-module $M\in \cA$ defines a filtered object~$\pwz(M)=M$ with $\pwz(M)^n=M$ for $n\le 0$ and $\pwz(M)^n=0$ for $n>0$. We call such a filtered object $\pwz(M)=\cdots 0=0\subseteq M= M=\cdots$ \emph{pure of weight zero}.
\end{Exa}

\begin{Rem}
\label{Rem:idempotent-complete-Afil}%
The subcategory $\Afil$ of~$\cA^{\bbZ\op}$ is closed under retracts. Since $\cA^{\bbZ\op}$ is idempotent-complete, so is~$\Afil$. Hence in the sequel, we tacitly use that a retracted monomorphism in~$\Afil$ is the inclusion of a direct summand.
\end{Rem}

\begin{Rem}\label{Rem:tensor-filtered}
There is an induced tensor product on $\fil{\cA}$ making $\fgt\colon\Afil\to \cA$ a tensor functor, \ie defined by tensoring underlying objects and filtering via
  \begin{equation*}
    (A\otimes B)^n:=\colim_{p+q\geq n}A^p\otimes B^q.
  \end{equation*}
In our case, the tensor $\otimes\colon \cA\times \cA\to \cA$ is exact, hence preserves monomorphisms. So we can think of~$(A\otimes B)^n=\sum_{p+q=n}A^p\otimes B^q$ as the sum of the subobjects $A^p\otimes B^q\into \fgt(A)\otimes \fgt(B)$, although the map $\oplus_{p+q=n}A^p\otimes B^q\to \fgt(A)\otimes \fgt(B)$ need not be a monomorphism. By definition, the functor $\fgt:\fil{\cA}\to \cA$ is a tensor functor. The same holds for $\gr\colon \Afil\to \cA$, as we now check.
\end{Rem}

\begin{Lem}
  \label{Lem:gr-tensor} We have in $\cA$ a canonical isomorphism
  \begin{equation*}
    \gr^n(A\otimes B)=\bigoplus_{p+q=n}\gr^p(A)\otimes \gr^q(B),
  \end{equation*}
  which is part of the structure making $\gr\colon \Afil\to \cA$ a tensor functor.
\end{Lem}

\begin{proof}
  As explained above, we may think of $A\otimes B$ as a filtration on $\fgt(A)\otimes\fgt(B)$ with elements of weight at least~$n$ given by
  \begin{equation*}
    (A\otimes B)^n=\sum_{p+q= n}A^p\otimes B^q.
  \end{equation*}
  Furthermore, for every pair $(p,q)$ such that $p+q=n$, the canonical map $A^p\otimes B^q\to\gr^n(A\otimes B)$ vanishes on $A^{p+1}\otimes B^q + A^p\otimes B^{q+1}$. This induces a map
  \begin{equation}
    \label{eq:component-monoidal-gr}
    \bigoplus_{p+q=n}\gr^p(A)\otimes \gr^q(B)\onto \gr^n(A\otimes B)
  \end{equation}
  which is part of a natural transformation making $\gr$ lax-monoidal. Note that~\eqref{eq:component-monoidal-gr} is surjective.
  Summing over all $n$, the $k$-dimensions of the domain and codomain of~\eqref{eq:component-monoidal-gr} are the $k$-dimensions of $\fgt(A)\otimes\fgt(B)$ and $\fgt(A\otimes B)$, respectively---which are equal. We conclude that for each~$n$ the map in~\eqref{eq:component-monoidal-gr} is an isomorphism.
\end{proof}

\begin{Exa}
The functor $\pwz\colon \cA\to \Afil$ of \Cref{Exa:filt-0} is a tensor-functor.
\end{Exa}

\begin{Lem}
  \label{Lem:fil-rigid}%
  The $\otimes$-category $\Afil$ is rigid, and the dual $A^\vee$ of $A$ is given by
  \begin{equation*}
    (A^\vee)^n:=\ker(\fgt(A)^\vee\to (A^{-n+1})^\vee)
  \end{equation*}
  with the canonical transition morphisms. Furthermore, $\gr^n(A^\vee)\cong (\gr^{-n}(A))^\vee$.
\end{Lem}
\begin{proof}
  The proof is straightforward.
\end{proof}

\begin{Not}
\label{Not:beta}%
  Consider the ``twist'' functor $(m):\fil{\cA}\to\fil{\cA}$ which keeps the same underlying object but shifts the filtration (or the weight) by $m\in\bbZ$:
  \begin{equation*}
    A(m)^n=A^{n-m}.
  \end{equation*}
  It comes with a canonical morphism $\beta\colon A\to A(1)$, given by the identity on the underlying object. This defines a natural transformation $\beta:\Id\To (1)\colon \Afil \to \Afil$.
\end{Not}

We want to describe all objects of~$\Afil$. Recall that the Krull-Schmidt category~$\cA=\mmod{kC_2}$ has two indecomposable objects, $k$ with trivial action and the free module~$kC_2$. Let us start by constructing filtrations on these two objects.

\begin{Cons}
\label{Cons:e_l}%
Consider the basic objects $\unit(m)=\pwz(k)(m)$ and $\E_0(m)=\pwz(kC_2)(m)$ in $\fil{\cA}$, for $m\in\bbZ$. For $\ell\in\bbZ_{\geq 1}$, we define the object $\E_\ell(m)$ as
\begin{equation}
\label{eq:e_l}%
  \E_\ell(m)=\qquad\cdots \subset 0\subset k = \cdots = k\xinto{\eta} kC_2 = kC_2 = \cdots
\end{equation}
where $kC_2$ occurs in filtration degrees~$\le m$ and $k$ appears in degrees from $m+\ell$ down to~$m+1$. So the underlying object of~$\E_\ell(m)$ is~$kC_2$. The homomorphisms out of $\E_\ell(m)$ can be described as follows:
\begin{equation}
\label{eq:Hom(e_l,-)}%
\begin{array}{rl}
  \Hom(\E_\ell(m),A)\kern-.7em
    &=\Big\{f\in\Hom_{kC_2}(kC_2,A)\,\Big|\,{\Displ f(kC_2)\subseteq A^m\atop\Displ\textrm{ and }f(k)\subseteq A^{m+\ell}}\Big\}
    \\[1em]
                &\cong\{x\in A^m\mid (1+\sigma)\,x\in A^{m+\ell}\}.
\end{array}\kern-1.5em
\end{equation}
(Recall that $\sigma$ is the name of the generator of~$C_2$ and that $\eta\colon k\into kC_2$ is $1\mapsto 1+\sigma$.)
\end{Cons}

\begin{Rem}
\label{Rem:a^>m}%
Continuing the analogy with motives, we shall say that a filtered object $A\in\Afil$ is \emph{effective} if~$\fgt(A)=A^0$, that is $A^0=A^{-1}=\cdots=A^n$ for all~$n\le 0$, or equivalently if~$A$ lives entirely in non-negative weights. For every $m\in\bbZ$, we shall also use the notation
\[
A^{\ge m}= \qquad \cdots \subseteq A^{m+1} \subseteq A^{m} = A^{m} = \cdots
\]
for the subspace of weight at least $m$, the filtered object~$A^m$ (still in weight~$m$) together with all higher weights. We have a monomorphism $A^{\ge m}\into A$. So an object $A\in\Afil$ is effective if and only if $A^{\ge0}=A$. For example, the objects~$\E_\ell(m)$ of \Cref{Cons:e_l} satisfy $(\E_\ell(m))^{\ge m'}=\E_\ell(m)$ whenever $m'\le m$. In particular, for $m\ge 1$, we have
\[
(\E_\ell(m))^{\ge 1}=\E_\ell(m).
\]
\end{Rem}

\begin{Lem}
\label{Lem:KRS-prepa}%
Let $A=A^{\ge0}$ be an effective object of~$\Afil$. Let $m\ge 1$ and $\ell\ge 0$ and $\alpha\colon \E_\ell(m)\to A$ be a morphism such that $\alpha^{\ge1}\colon (\E_\ell(m))^{\ge1}\to A^{\ge1}$ is a split monomorphism, \ie we give $\E_\ell(m)=(\E_\ell(m))^{\ge1}$ as direct summand of $A^{\ge1}$. Then $\alpha$ is a split monomorphism as well, \ie $\E_\ell(m)$ is a direct summand of $A$ itself.
\end{Lem}

\begin{proof}
Note that since $m\ge1$, the underlying morphism of $\alpha\colon \E_\ell(m)\to A$ `lands' in~$A^1$, that is, $\alpha(kC_2)\subseteq A^1$. Pick a retraction of~$\alpha^{\ge1}$, say $r^1\colon A^{\ge1}\to (\E_\ell(m))^{\ge1}=\E_\ell(m)$. This morphism $r^1$ consists of $r^1\colon A^1\to kC_2$ which maps the filtration of $A$ into the one for~$\E_\ell(m)$, for all weights\,$\geq1$, and satisfies $r^1\circ \alpha=\id_{kC_2}$ on underlying objects. The only question is to extend $r^1\colon A^1\to kC_2=(\E_\ell(m))^1=(\E_\ell(m))^0$ to the whole of~$A^0$
\[
\xymatrix@C=4em@R=.2em{
\E_\ell(m) \ar[r]^-{\alpha} & A \ar@{-->}[r]^-{\exists\,?\ r}
& \E_\ell(m)
\\
\vdots&\vdots&\vdots\\
kC_2 \ar[r]^-{\alpha} \ar@{=}[u] & A^0 \ar@{=}[u] \ar@{-->}[r]^-{\exists\,?\ r^0} & kC_2 \ar@{=}[u]
\\\\
kC_2 \ar[r]^-{\alpha} \ar@/_2em/[rr]_(.3){\id} \ar@{=}[uu] & A^1 \ar@{ >->}[uu] \ar[r]^-{r^1} & kC_2 \ar@{=}[uu]
\\\\
\vdots \ar@{>->}[uu] & \vdots \ar@{>->}[uu] & \vdots \ar@{>->}[uu]
}
\]
Such an extension $r^0$ of~$r^1$ exists because $A^1\into A^0$ is a monomorphism and $kC_2$ is injective in~$\cA$. This automatically defines a retraction~$r\colon A\to \E_\ell(m)$ of~$\alpha$.
\end{proof}

\begin{Prop}\label{Prop:KRS-Efil}
  Every object in $\fil{\cA}$ is isomorphic to one of the form
  \begin{equation*}
    \bigoplus_i\E_{\ell_i}(m_i)\oplus\bigoplus_j\unit(n_j).
  \end{equation*}
  for finitely many integers $\ell_i\ge 0$, $m_i\in\bbZ$ and $n_j\in \bbZ$. {\rm(}See~\eqref{eq:e_l} for $\E_{\ell}(m)$.{\rm)}
\end{Prop}
\begin{proof}
  Let $A\in\Afil$. Since twisting on $\Afil$ is an equivalence of categories, and by induction on the filtration amplitude, we may assume that $A$ is effective, $\fgt(A)=A^0$, and that the statement holds for $A^{\geq 1}$. Doing induction on the $k$-dimension of $A^0$ and using \Cref{Lem:KRS-prepa} (see \Cref{Rem:idempotent-complete-Afil}), we may assume that $A^{1}$ has trivial $C_2$-action: any $\E_\ell(m)$ summand of~$A^{\ge1}$ is already a summand of~$A$.

Let $\ell$ be the maximal weight in~$A$, that is, the largest integer~$\ell$ such that $A^\ell\neq 0$. If $\ell=0$, we have $A^1=0$ and $A=A^{\ge0}$ is simply a $kC_2$-module pure of weight zero, hence is of the form $\bigoplus_i \E_0\oplus\bigoplus_j\unit$, as wanted. So let us suppose $\ell>0$.

Let $x\in A^\ell$ be non-zero. Since $\ell>0$, we have $x\in A^1$ and thus $x$ is fixed by the $C_2$-action by the first reduction above. Let us define a sub-module $B^0$ of~$A^0$ by distinguishing two cases. If $x=y+\sigma y$ for some $y\in A^0$ let $B^0=ky+k\sigma y\simeq kC_2$, otherwise let $B^0=k x\simeq k$. Note that in both cases the inclusion $B^0\into A^0$ has a retraction $r^0:A^0\to B^0$ in the category of $kC_2$-modules. In the first case, $B_0$ is an injective $kC_2$-module. In the second case, it is an easy exercise to verify that if $i\colon k\into A$ in $\mmod{kC_2}$ does not factor through $kC_2$, then $i$ is a split monomorphism. (The assumption means that $i$ remains non-zero in~$\stab(\cA)\cong\mmod{k}$, hence splits there, hence splits in~$\cA$ as well since $\cA(k,k)\cong\Hom_{\stab \cA}(k,k)\cong k$.) Endowing $B^0$ with the filtration induced by the inclusion $B^0\into A^0$ we see that the resulting object $B$ is of the form $\E_\ell$ (in the first case) or $\unit(\ell)$ (in the second case).

Our claim is that $r^0:A^0\to B^0$ is compatible with the filtrations and thus defines a retraction $r:A\to B$ of the inclusion $B\into A$. We need to show that for $z\in A^{n}$ with $n\ge 1$, we have $r^0(z)\in B^n$. We know that $z\in A^{\ge 1}$ is fixed by the $C_2$-action hence so is its image $r^0(z)\in B^0$, thus $r^0(z)\in kx\subseteq A^\ell$ for the maximal weight~$\ell$ in~$A$. Hence, we see that $r^0\colon A^0\to B^0$ either maps $z\in A^n$ to zero (hence to $B^n$) or it maps $z$ into the highest possible weight~$\ell\ge n$ where non-zero elements exist. In any case, $r$ respects the filtration. So we can split off~$B$ as a direct summand of~$A$ and finish the proof by induction.
\end{proof}

\begin{Lem}
\label{Lem:hom-Afil}%
For every $n\in \bbZ$ and $\ell\ge 0$ the endomorphisms rings
\[
\End_{\Afil}(\unit(n)) = k\cdot \Id
\qquadtext{and}
\End_{\Afil}(\E_\ell(n)) = kC_2\cdot \Id
\]
are local rings.
\end{Lem}

\begin{proof}
Of course, as twisting is an auto-equivalence, we can assume $n=0$. It is clear that $\End_{\Afil}(\unit)=k\cdot\Id$. On the other hand, we deduce from~\eqref{eq:Hom(e_l,-)} that $\End_{\Afil}(\E_\ell)$ equals $\End_{kC_2}(kC_2)\cong kC_2\cong k[\sigma]/(\sigma^2-1)\cong k[s]/s^2$.
\end{proof}

\begin{Cor}
\label{Cor:Afil-krull-schmidt}%
  The category $\Afil$ is Krull-Schmidt. In particular, the decomposition in \Cref{Prop:KRS-Efil} is unique up to permutation (and isomorphism) of the indecomposable summands $\unit(n)$ and $\E_{\ell}(m)$.
\end{Cor}
\begin{proof}
The category is Krull-Schmidt because we are over a field and all objects are Noetherian and Artinian. We can also see this directly from \Cref{Prop:KRS-Efil}, which moreover describes the indecomposables. Indeed, it suffices to know that the endomorphism rings of $\unit(n)$ and $\E_{\ell}(m)$ are local rings, which is \Cref{Lem:hom-Afil}.
\end{proof}

Let us now discuss the tensor structure of~$\Afil$ on the indecomposables~$\E_\ell(m)$.

\begin{Lem}\label{Lem:special-dual}
  The $\otimes$-dual of $\E_\ell(m)$ is isomorphic to~$\E_\ell(-m-\ell)$.
\end{Lem}
\begin{proof}
 The dual of $kC_2$ in $\cA$ is $kC_2$ and the result follows from \Cref{Lem:fil-rigid}.
\end{proof}

\begin{Rem}
We now want to describe $\E_\ell\otimes \E_{\ell'}$. We have already described the underlying object $kC_2\otimes kC_2\cong kC_2\oplus kC_2$ in \Cref{Rmd:Galois}. However that isomorphism does not preserve weights; indeed, weights in~$\E_\ell$ are controlled by $\eta\colon k\into kC_2$ and this monomorphism, defined by~$\eta(1)=1+\sigma$, is not the $k$-linearization of a map of $C_2$-sets. Hence we compose the Galois isomorphism $kC_2\otimes kC_2\isoto kC_2\oplus kC_2$ of~\eqref{eq:Galois} with an automorphism of~$kC_2\oplus kC_2$, namely~$\smat{1&1\\0&\sigma}$.
\end{Rem}

\begin{Not}
\label{Not:bamma}%
Let $\Epur=kC_2$. Define an isomorphism of $kC_2$-modules $\bamma\colon \Epur\otimes \Epur\isoto \Epur\oplus \Epur$ and its inverse~$\bamma\inv$ as follows
\begin{equation}
\label{eq:bamma}%
\quad\vcenter{\xymatrix@C=1.5em@R=.3em{
\Epur\otimes \Epur \ar[r]^-{\bamma}
& \Epur\oplus \Epur
&& \Epur\oplus \Epur \ar[r]^-{\bamma\inv}
& \Epur\otimes \Epur
\\
1\otimes 1 \ar@{|->}[r] & (1,0)
&& (1,0) \ar@{|->}[r] & 1\otimes 1\
\\
1 \otimes \sigma \ar@{|->}[r] & (1,\sigma)
&& (\sigma,0) \ar@{|->}[r] & \sigma\otimes \sigma
\\
\sigma \otimes 1 \ar@{|->}[r] & (\sigma,1)
&& (0,1) \ar@{|->}[r] & \sigma\otimes \sigma + \sigma\otimes 1\kern-3.7em
\\
\sigma \otimes \sigma \ar@{|->}[r] & (\sigma,0)
&& (0,\sigma) \ar@{|->}[r] & 1\otimes 1 + 1\otimes\sigma.\kern-3.7em
}}
\end{equation}
\end{Not}

\begin{Rem}
\label{Rem:bamma-basics}%
For some computations, it can be useful to know what happens to~$\eta\colon k\to kC_2=\Epur$ and to $\eps\colon \Epur=kC_2\to k$ under tensorization with~$\Epur$ and the identification of~\eqref{eq:bamma}. It is easy to check that the following diagrams commute:
\[
\vcenter{\xymatrix{
\unit \otimes \Epur \ar[r]^-{}_-{\cong} \ar[d]_-{\eta\otimes 1}
& \Epur \ar[d]^-{\smat{\eta\eps\\\id}}
\\
\Epur \otimes \Epur \ar[r]^-{\bamma}_-{\simeq}
& \Epur\oplus \Epur
}}
\qquad
\vcenter{\xymatrix{
\Epur \otimes \Epur \ar[r]^-{\bamma}_-{\simeq} \ar[d]_-{\eps\otimes 1}
& \Epur\oplus \Epur \ar[d]^-{\smat{\id&\eta\eps}}
\\
\unit  \otimes \Epur \ar[r]^-{}_-{\cong}
& \Epur
}}
\qquad
\vcenter{\xymatrix{
\Epur \otimes \Epur \ar[r]^-{\bamma}_-{\simeq} \ar[d]_-{\eta\eps\otimes 1}
& \Epur\oplus \Epur \ar[d]^-{\smat{\eta\eps&0\\\id&\eta\eps}}
\\
\Epur \otimes \Epur \ar[r]^-{\bamma}_-{\simeq}
& \Epur\oplus \Epur
}}
\]
\[
\vcenter{\xymatrix{
\Epur \otimes \unit \ar[r]^-{}_-{\cong} \ar[d]_-{1\otimes \eta}
& \Epur \ar[d]^-{\smat{0\\\sigma}}
\\
\Epur \otimes \Epur \ar[r]^-{\bamma}_-{\simeq}
& \Epur\oplus \Epur
}}
\qquad
\vcenter{\xymatrix{
\Epur \otimes \Epur \ar[r]^-{\bamma}_-{\simeq} \ar[d]_-{1\otimes \eps}
& \Epur\oplus \Epur \ar[d]^-{\smat{\id&0}}
\\
\Epur \otimes \unit \ar[r]^-{}_-{\cong}
& \Epur
}}
\qquad
\vcenter{\xymatrix{
\Epur \otimes \Epur \ar[r]^-{\bamma}_-{\simeq} \ar[d]_-{1\otimes \eta\eps}
& \Epur\oplus \Epur \ar[d]^-{\smat{0&0\\\sigma &0}}
\\
\Epur \otimes \Epur \ar[r]^-{\bamma}_-{\simeq}
& \Epur\oplus \Epur
}}
\]
Finally, the swap of factors $(12)\colon \Epur\otimes \Epur\isoto \Epur\otimes \Epur$ becomes~$\smat{1&\eta\eps\\0&\sigma}\colon \Epur\oplus \Epur\isoto \Epur\oplus \Epur$.
\end{Rem}

\begin{Prop}
\label{Prop:tens-e_l}%
Let $\ell'\geq \ell\geq 0$ and $i,j\in\bbZ$. The isomorphism $\bamma$ of~\eqref{eq:bamma} induces an isomorphism in~$\Afil$ for the filtered objects defined in~\eqref{eq:e_l}
\[
\E_\ell (i) \otimes \E_{\ell'}(j)\isoto \E_\ell(i+j) \oplus \E_\ell(i+j+\ell').
\]
\end{Prop}

\begin{proof}
We easily reduce to the case $i=j=0$. It suffices to show that the explicit $\bamma$ and $\bamma\inv$ of~\eqref{eq:bamma} preserve the filtrations to induce maps in~$\Afil$:
\[
\bamma\colon \E_\ell\otimes \E_{\ell'}\to \E_\ell\oplus \E_\ell(\ell')
\quadtext{and}
\bamma\inv\colon \E_\ell\oplus \E_\ell(\ell')\to \E_\ell\otimes \E_{\ell'}.
\]
In each case, we need to trace what happens to higher-weight elements of the form $1+\sigma$. For instance $(1+\sigma)\otimes 1$ in $\E_\ell\otimes \E_{\ell'}$ must land under~$\bamma$ in weight at least~$\ell$ in both summands~$\E_\ell$ and $\E_\ell(\ell')$. This image is~$(1+\sigma,\sigma)$ which is indeed in weight~$\ell$ in the first summand and in weight at least~$\ell$ in the second because of the assumption~$\ell'\ge \ell$. For another instance, $\bamma\inv$ should map $(0,1+\sigma)$ to something in weight at least $\ell+\ell'$ in $\E_\ell\otimes \E_{\ell'}$. That image is equal to $(1+\sigma)\otimes (1+\sigma)$, which is in weight at least~$\ell$ in the first factor, $\ell'$ in the second, thus in weight at least~$\ell+\ell'$ in the tensor. The remaining verifications are left to the reader.
\end{proof}

\bigbreak\goodbreak
\section{Frobenius category of filtered \texorpdfstring{$kC_2$}{kC2}-modules}
\label{sec:filt-Frobenius}%
\medbreak

Let $\cA=\mmod{kC_2}$ as in \Cref{sec:filt-repr-setup}, where we described the Krull-Schmidt tensor-category $\Afil$ of filtered $kC_2$-modules. We now want to discuss its homological structure, in the form of a Frobenius exact category (see \Cref{Rmd:general}).

\begin{Def}
\label{Def:Afil}%
Let $\Asplit$ denote the category $\cA$ with the minimal exact structure: Admissible short exact sequences are precisely the \emph{split} exact ones. In $\Afil$, we define a sequence $(f,g)=\big(\xymatrix@C=1.5em{A \,\ar@{>->}[r]|(.45){f} & B \ar@{->>}[r]|(.45){g} & C}\big)$ to be \emph{admissible} if $g\circ f=0$ and
\[
\big(\gr(f),\gr(g)\big)=\big(\xymatrix{\gr(A) \,\ar@{>->}[r]^-{\gr(f)} &\  \gr(B)\ \ar@{->>}[r]^-{\gr(g)} & \ \gr(C)}\big)
\]
is admissible in $\Asplit$. Equivalently, this means that $g\circ f=0$ and the sequences $\big(\gr^n(f),\gr^n(g)\big)$ are split exact in~$\cA$ for all~$n\in\bbZ$. These admissible exact sequences define an exact structure on~$\Afil$, see~\cite[Prop.~1.4, Lem.~1.9, Prop.~1.10]{drss:exact-vector}, that we denote
\[
\Efil\,.
\]
\end{Def}

\begin{Rem}
\label{Rem:fgt-gr-exact}%
The two functors of \Cref{Def:filtrations} are exact:
 \[
   \fgt:\Efil\to\cA\qquadtext{and}\gr:\Efil\to\Asplit
 \]
where the left-hand~$\cA$ has the abelian structure and the right-hand one has the split exact structure.
\end{Rem}

\begin{Rem}
\label{Rem:quasi-abelian}%
At first, the reader might be puzzled by our notation $\Efil$ to denote the same category that we denoted $\Afil$ in \Cref{sec:filt-repr-setup}. We choose to emphasize this point since it touches the technical crux of many of our discussions below. Indeed, there is another (maximal) exact structure on the category $\Afil$ that we denote
\[
\Afilmax
\]
coming from the fact that $\Afil$ is \emph{quasi-abelian}, \cf~\cite{Schneiders:quasi-abelian}. One possible definition of a quasi-abelian category is as a pre-abelian category, \ie one with kernels and cokernels, such that the family of all kernel-cokernel sequences $(f,g)=\big(\xymatrix@C=1.5em{A \,\ar@{>->}[r]|(.45){f} & B \ar@{->>}[r]|(.45){g} & C}\big)$ defines an exact structure. The two functors
 \[
   \fgt:\Aqab\to\cA\qquadtext{and}\gr:\Aqab\to\cA
 \]
are now exact when the target has the abelian structure in both cases.

Note that a sequence $(f,g)$ is exact in~$\Aqab$ if and only if $g\circ f=0$ and $(\gr(f),\gr(g))$ is exact in the abelian category~$\cA$. (This implies, but is different from, $(\fgt(f),\fgt(g))$ being exact. For example, the sequence $\unit\xto{\beta}\unit(1)\to 0$ is exact on underlying vector spaces but not intrinsically exact, \ie not in~$\Aqab$.)

Several things we shall spend time proving might seem trivially true if one does not pay attention to the special exact sequences of~$\Efil$. Conversely, some things we shall say would be plain wrong with another exact structure on~$\Afil$. Also, the motivic result of Positselski that we connect with in \Cref{part:II} involves~$\Efil$, not~$\Afilmax$.
\end{Rem}

\begin{Rem}
\label{Rem:gr-split-mono}%
It is convenient to have the quasi-abelian structure $\Aqab$ on~$\Afil$ even to study~$\Efil$. For instance, a morphism $f\colon A\to B$ in~$\Afil$ such that $\gr(f)$ is a split monomorphism is necessarily an admissible monomorphism in~$\Efil$. Indeed, such an $f$ is intrinsically a monomorphism since $\gr\colon \Aqab\to \cA$ is exact and conservative. Thus $f$ fits in an intrinsically exact sequence $(f,g)=\big(\xymatrix@C=1.5em{A \,\ar@{>->}[r]|(.45){f} & B \ar@{->>}[r]|(.3){g} & \coker(f)}\big)$. Its image under~$\gr$ is split exact, hence $(f,g)$ is indeed admissible in~$\Efil$.
\end{Rem}

\begin{Rem}
\label{Rem:pwz-exact}%
The tensor-functor $\pwz$ of \Cref{Exa:filt-0} can be seen as exact in two ways, either as $\pwz\colon \Asplit\to \Efil$ or $\pwz\colon \cA\to \Aqab$ (with this $\cA$ abelian).
\end{Rem}

\begin{Lem}
  \label{Lem:gr-conservative-and-section}%
  The exact functor $\gr:\Efil\to\Asplit$ induces a conservative tt-functor
  \begin{equation*}
    \gr:\DEfil\to\Kb(\cA)
  \end{equation*}
  with a section tt-functor $\pwz:\Kb(\cA)\to \DEfil$ induced by~$\pwz$ (see \Cref{Rem:pwz-exact}).
\end{Lem}
\begin{proof}
Since $\gr$ is a tensor functor (\Cref{Lem:gr-tensor}), the induced $\gr:\DEfil\to\Kb(\cA)$ is indeed a tt-functor. For conservativity, let $A\in \DE$ be a complex in~$\Efil$ such that $\gr(A)=0$. We can assume $A=(\cdots 0\to A_n\to \cdots \to A_0\to 0\cdots)$ and proceed by induction on~$n$. The cases $n=0$ and~$n=1$ are trivial. By assumption $\gr(A)$ is homotopically trivial, hence $\gr(A_n\to A_{n-1})$ is a split monomorphism in~$\cA$. This means that $d\colon A_n\to A_{n-1}$ is an admissible monomorphism in~$\Efil$; see \Cref{Rem:gr-split-mono}. Therefore $A$ is spliced together from an admissible exact sequence and a shorter complex $A'=(\cdots 0 \to \coker(d_n)\to A_{n-2}\to \cdots \to A_0\to 0\cdots)$ as usual:
\[
\xymatrix@R=.5em{
\cdots 0 \ar[r]
& A_n \ar@{ >->}[r]^-{d_n}
& A_{n-1} \ar@{->>}[rd] \ar[rr]^-{d_{n-1}}
&& A_{n-2} \ar[r]
& \cdots
\\
&&& \coker(d_n) \ar[ru]
}
\]
Therefore, in $\DE$, we have $A\simeq A'$ and $\gr(A')=0$. By induction hypothesis, we have $A'=0$ and thus $A=0$. The last statement is easy from $\gr\circ\pwz=\Id_{\cA}$.
\end{proof}

\begin{Exa}
   \label{Exa:fundamental-sequences}%
   Recall the fundamental short exact sequence of $kC_2$-modules~\eqref{eq:extension-S}
   \begin{equation*}
     0\to k\xto{\eta}kC_2\xto{\epsilon}k\to 0.
   \end{equation*}
We may think of these objects as pure of weight zero, that is, we can apply $\pwz$ (\Cref{Exa:filt-0}) to the complex~$\fundpur$ of~\eqref{eq:fund-pur}, and thus obtain a sequence in~$\Afil$
\begin{equation}
  \label{eq:fundamental-nonexact-sequence}
  0\to \unit\xto{\eta}\E_0\xto{\epsilon}\unit\to 0
\end{equation}
which we denote by~$\fund_0=\pwz(\fundpur)$. (As a complex in~$\Afil$, it is still viewed as non-zero in homological degrees $2,1,0$.) As $\fund_0$ is \emph{not} an admissible sequence in $\Efil$, since $\gr^0(\fund_0)=\fundpur$ is not split, the complex $\fund_0$ is a non-zero object of~$\DE$. It would be acyclic in~$\Aqab$ though.

On the other hand, there is an infinite family of admissible exact sequences in~$\Efil$
   \begin{equation}
     \label{eq:fundamental-exact-sequences}
     \fund_\ell=\kern5em\unit(\ell)\xinto{\eta}\E_\ell\xonto{\epsilon}\unit\kern9em
   \end{equation}
for any $\ell\geq 1$, satisfying $\fgt(\fund_\ell)=\fundpur$. (Recall $\E_\ell$ from \Cref{Cons:e_l}.) We call these~$\fund_\ell$ the \emph{fundamental admissible exact sequences}. Of particular importance is
\begin{equation}
\label{eq:fund_1}%
\fund_1=\kern5em\unit(1)\xinto{\eta}\E_1\xonto{\epsilon}\unit\kern9em
\end{equation}
or in expanded form:
\vskip-\baselineskip\vskip-\baselineskip
\begin{equation*}
\kern5em\vcenter{\xymatrix@H=1.5em@R=.8em{
&&\ar@{}|-{\vdots}[d]&\ar@{}|-{\vdots}[d]&\ar@{}|-{\vdots}[d]\\
&&k\ar[r]^-{\eta}&kC_2\ar[r]^-{\epsilon}&k&&\textrm{(weight }-1)\\
&&k\ar[r]^-{\eta}\ar@{^(->}[u]_-{=}&kC_2\ar@{^(->}[u]_-{=}\ar[r]^-{\eps}&k\ar@{^(->}[u]_-{=}&&\textrm{(weight }0)\\
&&k\ar[r]^-{1}\ar@{^(->}[u]_-{=}&k\ar@{^(->}[u]_{\eta}\ar[r]&0\ar@{^(->}[u]&&\textrm{(weight }1)\\
&&0\ar[r]\ar@{^(->}[u]&0\ar@{^(->}[u]\ar[r]&0\ar@{^(->}[u]&&\textrm{(weight }2)\\
&&\ar@{}|(.7){\vdots}[u]&\ar@{}|(.7){\vdots}[u]&\ar@{}|(.7){\vdots}[u]}}
\end{equation*}
\vskip-\baselineskip\vskip-\baselineskip
\end{Exa}

 \begin{Prop}
   \label{Prop:Afil-flatness}%
   Every object in $\Efil$ is flat, \ie the category $\Efil$ is tensor-exact.
 \end{Prop}
 \begin{proof}
   This follows readily from \Cref{Lem:gr-tensor} and \Cref{Def:Afil}.
 \end{proof}

\begin{Rem}
\label{Rem:enough-proj}%
Tensoring~\eqref{eq:fund_1} with an~$A\in \Afil$ shows that every $A$ receives an admissible epimorphism from~$\E_1\otimes A$. The latter are the projectives in~$\Efil$:
\end{Rem}

\begin{Prop}
\label{Prop:proj-Efil}%
The subcategory of projective objects of~$\Efil$, as an exact category, coincides with the thick $\otimes$-ideal $\add^\otimes(\E_1)$ generated by~$\E_{1}$, namely it consists of all  direct sums of $\E_0(i)$ and $\E_1(j)$ for $i,j\in \bbZ$.\,{\rm(\footnote{\,We write $\add^{\otimes}(\E_1)$ for the thick $\otimes$-ideal of~$\Afil$ generated by~$\E_1$, instead of $\ideal{\E_1}$. We reserve $\ideal{\E_1}$ for the tt-ideal generated by~$\E_1$ in upcoming tt-categories, like $\DE$ for instance.})}
\end{Prop}

\begin{proof}
By rigidity (\Cref{Lem:fil-rigid}) and flatness of all objects~$A$ (\Cref{Prop:Afil-flatness}), the projectives~$P$ form a $\otimes$-ideal since $\Hom(A\otimes P,-)\cong \Hom(P,A^\vee\otimes-)$. In view of \Cref{Cor:Afil-krull-schmidt}, \Cref{Prop:tens-e_l} and \Cref{Rem:enough-proj}, it suffices to show that $\E_1$ is projective in~$\Efil$. To see that, we need to show that it has the lifting property with respect to admissible epimorphisms. Recall the description $\Homfil(\E_1,A)=\SET{x\in A^0}{(1+\sigma)x\in A^1}$ given in \Cref{Cons:e_l}. Consider now an admissible epimorphism $g\colon B\onto C$ in~$\Efil$ and specifically the part around $\gr^0$:
\[
\xymatrix@R=1em@H=1em{
\gr^0(B) \ar@{->>}[r]^-{g}
& \gr^0(C)
\\
B^0 \ar@{->>}[r]^-{g} \ar@{->>}[u]
& C^0 \ar@{->>}[u]
\\
B^1 \ar@{->>}[r]^-{g} \ar@{ >->}[u]
& C^1 \ar@{ >->}[u]
}
\]
in which the rows are epimorphisms and the columns exact in~$\cat{A}$, and the top row is furthermore a split epimorphism. Suppose given $z\in C^0$ such that $(1+\sigma)z\in C^1$. We need to find~$y\in B^0$ such that $g(y)=z$ and $(1+\sigma)y\in B^1$. Consider first $\bar z\in \gr^0(C)$ and note that $(1+\sigma)\bar z=0$, that is, $\bar z$ is $C_2$-fixed. Since the top epimorphism is split, we can lift $\bar z$ to some $\bar y\in \gr^0(B)$ still $C_2$-fixed. In other words, we have found $y\in B^0$ such that $(1+\sigma)y\in B^1$ and whose image in $\gr^0(C)$ is~$\bar z$, \ie the same as our initial~$z$. We do not know that $g(y)=z$, we only know this modulo~$C^1$. Hence there exists $z'\in C^1$ such that $z=g(y)+z'$. Since $g\colon B^1\onto C^1$ is an epimorphism, we can pick $y'\in B^1$ such that $g(y')=z'$. Direct verification shows that the element $y'':=y+y'\in B^0$ satisfies $(1+\sigma)y''\in B^1$ and $g(y'')=z$, hence is a lift of the initial~$z$ under $\Homfil(\E_1,B)\to \Homfil(\E_1,C)$.
\end{proof}

\begin{Cor}
\label{Cor:Afil-Frobenius}%
The exact category $\Efil$ is Frobenius. In particular, its injective-projective objects are sums of $\E_0(i)$ and $\E_1(j)$ for $i,j\in\bbZ$ as in \Cref{Prop:proj-Efil}, and they form a $\otimes$-ideal.
\end{Cor}

\begin{proof}
Rigidity (\Cref{Lem:fil-rigid}) provides an equivalence of exact categories $(-)^\vee:(\Efil)^{\op}\isoto\Efil$. As $\E_1^\vee\simeq \E_1(-1)$ by \Cref{Lem:special-dual}, it follows that injective and projective objects in $\Efil$ coincide. There are enough of them by \Cref{Rem:enough-proj}.
\end{proof}

\begin{Rem}
\label{Rem:Aqab-Frobenius}%
We can also consider the quasi-abelian structure~$\Aqab$ on~$\Afil$, as in \Cref{Rem:quasi-abelian}. Since it admits more exact sequences than $\Efil$, it will have less projectives and injectives. One easily verifies that the projectives and injectives coincide in~$\Aqab$ (using the same argument as above) and that they contain all sums of twists of~$\E_0$. Also, tensoring any object $A\in \Aqab$ with the intrinsically-exact sequence $\unit\into \E_0\onto \unit$, we see that $\Aqab$ is Frobenius with subcategory of projective-injective equal to the thick $\otimes$-ideal $\add^\otimes(\E_0)$ generated by~$\E_{0}$ (\cf \Cref{Prop:tens-e_l}).
\end{Rem}

We can now consider the derived category~$\DE$, which is tensor-triangulated, and whose spectrum we compute in \Cref{sec:Spc-DE}. We shall need the following fact which is direct from the exact structure discussed in the present section:

\begin{Lem}
\label{Lem:e_l-e_l+1}%
For $\ell\ge 1$, we have an isomorphism in $\DEfil$
\[
\cone(\E_\ell\xto{\iota} \E_{\ell+1})\simeq\cone(\beta\colon \unit(\ell)\to \unit(\ell+1))
\]
where $\iota\colon \E_\ell\to \E_{\ell+1}$ is underlain by $\id_{kC_2}$. (Recall $\E_\ell$ from \Cref{Cons:e_l}.)
\end{Lem}

\begin{proof}
Consider the following morphism of complexes~$s$, where morphisms in~$\Afil$ are described as usual by the underlying morphisms of $kC_2$-modules
\[
\xymatrix@R=1em{
\cone(\unit(\ell) \xto{\beta} \unit(\ell+1))= \ar[d]_-{s}
& \cdots 0 \ar[r] \ar@<.7em>[d]
& \unit(\ell) \ar[r]^-{1} \ar@{>->}[d]^-{\eta}
& \unit(\ell+1) \ar[r] \ar@{>->}[d]^-{\eta}
& 0 \cdots \ar@<-.7em>[d]
\\
\cone(\E_\ell\xto{\iota}\E_{\ell+1})=
& \cdots 0 \ar[r] \ar@<.7em>[d]
& \E_{\ell} \ar[r]^-{1} \ar@{->>}[d]^-{\eps}
& \E_{\ell+1} \ar[r] \ar@{->>}[d]^-{\eps}
& 0 \cdots \ar@<-.7em>[d]
\\
& \cdots 0 \ar[r]
& \unit \ar[r]^-{1}
& \unit \ar[r]
& 0 \cdots
}
\]
We complete vertically by using the fundamental admissible exact sequences~\eqref{eq:fundamental-exact-sequences}. Hence the cone of $s\colon \cone(\beta_{\unit(\ell)})\to \cone(\E_{\ell}\to \E_{\ell+1})$ is isomorphic in~$\DEfil$ to the bottom complex (\Cref{Rmd:DbE}) which is trivial. Thus $s$ is an isomorphism.
\end{proof}

\begin{Cor}
\label{Cor:generators-DE}%
The objects $\{\unit(1),\unit(-1),\E_0,\E_1\}$ generate $\DE$ as a tensor-triangulated category.
\end{Cor}

\begin{proof}
This is immediate from \Cref{Prop:KRS-Efil} and \Cref{Lem:e_l-e_l+1}.
\end{proof}

\bigbreak\goodbreak
\section{A central localization}
\label{sec:central}%
\medbreak

In Sections~\ref{sec:filt-repr-setup} and~\ref{sec:filt-Frobenius}, we turned the category $\Afil$ of filtered objects in~$\cA=\mmod{kC_2}$ into a tensor-exact Frobenius category~$\Efil$. The present section is dedicated to computing a central localization of its derived category~$\DE$. This will be a key ingredient in the computation of its spectrum. Along the way, we build a functor $\tfgt\colon \DE\to \KA$, different from the total-graded~$\gr$ of \Cref{Lem:gr-conservative-and-section}.

\begin{Rem}
\label{Rem:pwz-(-)^0}%
The idea is to discuss~$\Afil$ `around weight zero'. We have already seen in \Cref{Exa:filt-0} the inclusion of `pure-weight-zero' objects $\pwz\colon \cA\to \Afil$, mapping any $kC_2$-module to the filtered object pure in filtration degree zero. It admits a right adjoint $A\mapsto A^0$, taking the \wzp. Indeed, a morphism $f\colon \pwz(M)\to A$ is given by the underlying $f\colon M\to \fgt(A)$ which must land in~$A^0$ to respect the filtration, with no other condition. Furthermore, this adjunction
\[
\xymatrix@R=1.5em{\cA \ar@<-.3em>[d]_-{\pwz}
\\
\Afil  \ar@<-.3em>[u]_-{(-)^0}
}
\]
satisfies a projection formula, \ie there exists a natural isomorphism
\begin{equation}
\label{eq:projection-formula}%
(A\otimes \pwz(M))^0\cong A^0\otimes M
\end{equation}
for $M\in \cA$ and $A\in \Afil$. This holds for general reasons; see~\eqref{eq:projection-formula-general}. But~\eqref{eq:projection-formula} can also be seen as an \emph{equality} of submodules of~$\fgt(A)\otimes M$. Indeed, the \wzp\ of $A\otimes \pwz(M)$ consists of $\sum_{i+j=0}A^{i}\otimes \pwz(M)^{j}$ and we can replace $\pwz(M)^{j}=0$ for $j>0$ and $\pwz(M)^{j}=M$ for $j\le0$ and use $A^i\subseteq A^0$ for all~$i>0$.
\end{Rem}

\begin{Exa}
\label{Exa:(e_1(m))^0}%
Consider in~$\Afil$ the object~$\E_1= (\cdots \subset 0\subset k\xinto{\eta} kC_2 = kC_2 = \cdots)$ of \Cref{Cons:e_l}, with~$k$ in weight~$1$. By definition, we have
\[
(\E_1(m))^0=
\left\{
\begin{array}{cl}
kC_2 & \textrm{if }m\ge 0
\\
k & \textrm{if }m=-1
\\
0 & \textrm{if }m<-1.
\end{array}
\right.
\]
Furthermore, under the functor~$(-)^0\colon \Afil\to \cA$, the map $\eta\eps\colon \E_1(m)\to \E_1(m-1)$ goes to $\eta\eps\colon kC_2\to kC_2$ when $m>0$ and to $\eps\colon kC_2\onto k$ when $m=0$, and necessarily to zero otherwise. On the other hand, for $r\ge0$, the map $\beta^r\colon \E_1(m)\to \E_1(m+r)$ goes under~$(-)^0$ to $\id\colon kC_2\to kC_2$ when $m\ge 0$ and to $\eta\colon k\into kC_2$ when $m=-1$, and necessarily to zero otherwise. (For $\beta\colon \Id\to (1)$ see \Cref{Not:beta}.)
\end{Exa}

\begin{Rem}
\label{Rem:rwz}%
Let us add exact structures to the discussion of \Cref{Rem:pwz-(-)^0}. The left adjoint $\pwz\colon \Asplit\to \Efil$ is exact but its right adjoint~$(-)^0\colon \Efil\to \Asplit$ is not, since $(\fund_1)^0=\fundpur$ is not split exact. So we need to right-derive $(-)^0$ to obtain a well-defined functor on~$\DE$ taking values in~$\Db(\Asplit)=\Kb(\cA)$.

Every object $A\in\Efil$ admits a canonical injective resolution $\injres\otimes A$ where $\injres$ is the injective resolution of~$\unit$. (\Cref{Cor:Afil-Frobenius}.) The complex~$\injres$ in~$\Efil$ is obtained by splicing together the fundamental exact sequences $\unit(i)\into \E_1(i-1)\onto \unit(i-1)$ as in~\eqref{eq:fundamental-exact-sequences}, for $i\le 0$, and the resulting quasi-isomorphism $\tilde\eta\colon\unit\to \injres$ in~$\Efil$ is
\begin{equation}
\label{eq:I}%
\vcenter{\xymatrix@R=2em{
\unit= \ar@<-.5em>[d]^-{\tilde\eta}
&\cdots \ar[r]
& 0 \ar[r] \ar[d]
& \unit \ar[r] \ar[d]^-{\eta}
& 0 \ar[r] \ar[d]
& 0 \ar[d] \ar[r]
& \cdots
\\
\injres=
& \cdots \ar[r]
& 0 \ar[r]
& \E_{1}(-1) \ar[r]^-{\eta\eps}
& \E_{1}(-2) \ar[r]^-{\eta\eps}
& \E_{1}(-3) \ar[r]^-{\eta\eps}
& \cdots}}
\end{equation}
Consider the triangulated functor~$\mathrm{R}((-)^0)=(\injres\otimes-)^0\colon \Kb(\Afil)\to \Km(\cA)$.
\end{Rem}

\begin{Prop}
\label{Prop:rwz}%
The above functor~$(\injres\otimes -)^0$ takes values in the bounded subcategory $\Kb(\cA)$ of~$\Km(\cA)$ and yields a right adjoint to~$\pwz\colon\KA\to \DE$
\[
\xymatrix@R=1.5em{
\KA \ar@<-.5em>[d]_-{\pwz}
\\
\DEfil \ar@<-.5em>[u]_-{\Displ\rwz:=(\injres\otimes-)^0}
}
\]
called $\rwz$ for \emph{`right-derived weight zero'}. We have a projection formula
\begin{equation}
\label{eq:projection-formula-rwz}
\rwz(A\otimes \pwz(M))\cong \rwz(A)\otimes M
\end{equation}
for every $A\in \DEfil$ and $M\in \Kb(\cA)$, given degreewise by~\eqref{eq:projection-formula}. Furthermore $\rwz(\unit)\cong\unit$ and the unit $\Id_{\KA}\to \rwz\circ\pwz$ of the adjunction is an isomorphism, hence the pure-weight-zero functor $\pwz\colon \KA\to \DE$ is fully faithful.
\end{Prop}

\begin{proof}
The adjunction is a general fact about derived functors:
\[
\kern-.5em\begin{array}{rll}
\Hom_{\Der_{-}(\Efil)}(\pwz(M),A)\kern-.7em
& \cong \Hom_{\Der_{-}(\Efil)}(\pwz(M),\injres\otimes A)&\textrm{for $\tilde\eta\colon \unit \to \injres$ is an iso}
\\
& \cong \Hom_{\Km(\Efil)}(\pwz(M),\injres\otimes A)&\textrm{for $\injres\otimes A$ is degw.\ inj.}
\\
& \cong \Hom_{\Km(\cA)}(M,(\injres\otimes A)^0)&\textrm{by~\Cref{Rem:pwz-(-)^0}.}
\end{array}\]
The only specific claim here is that $(\injres\otimes-)^0\colon \Der_{-}(\Efil)\to \Km(\cA)$ restricts to bounded subcategories~$\DE\to \KA$. By exactness and by induction on the length of complexes, it suffices to show that if~$A\in\Efil$ then $(\injres\otimes A)^0\in \KA$. The term $(\injres\otimes A)^0_i=(\E_1(i-1)\otimes A)^0$ in degree~$i\le 0$ is the \wzp\ of~$B(i)$ for $B=\E_1(-1)\otimes A$, where~$B$ does not depend on~$i$. For $i\ll0$ the filtered object $B(i)$ is `pushed up' far enough so that its \wzp\ becomes trivial. Hence the claim.
The projection formula still holds by general principle~\eqref{eq:projection-formula-general} or simply because it holds degreewise.
A direct computation gives~$\rwz(\unit)=\injres^0=k[0]=\unit$. Combining with the projection formula, we have~$\rwz\circ\pwz(M)\cong \rwz(\unit\otimes\pwz(M))\cong \rwz(\unit)\otimes M\cong \unit\otimes M\cong M$. This isomorphism is the unit of the $\pwz\adj \rwz$ adjunction.
\end{proof}

We now define an invertible object~$\invert$ in~$\DE$ and a map $\omega\colon \unit\to \invert$, that will play an important role in the sequel.

\begin{Def}
\label{Def:omega}%
Recall the invertible object~$\invertpur=(\cdots \to 0\to k\xto{\eta}kC_2\to 0\to\cdots)$ in~$\KA$ from~\Cref{Prop:L^n}, with $k$ in degree one. Let
\begin{equation}
\label{eq:u}%
\invert:=\pwz(\invertpur)(1)=(\cdots0 \to \unit(1)\xto{\eta} \E_0(1)\to 0\cdots)
\end{equation}
be the twisted image of~$\invertpur$ in~$\DE$. Consider the morphism $\omega\colon \unit\to \invert$ in~$\DE$ given by the following fraction in~$\KE$:
\begin{equation}
\label{eq:omega}%
\vcenter{\xymatrix@C=1.8em@R=1em{
\unit=
&\cdots 0 \ar[r]
& 0 \ar[r]
& \unit \ar[r]^-{}
& 0 \cdots
\\
\tilde\unit:=\ar@<.6em>[u]_-{s} \ar@<-.6em>[d]^-{\tilde\omega}
&\cdots 0 \ar[r]
& \unit(1) \ar[r]^-{\eta} \ar[u] \ar@{=}[d]
& \E_1 \ar[r]^-{} \ar[u]_-{\eps} \ar[d]^-{\iota}
& 0 \cdots
\\
\invert= \ar@{<-} `l[u] `[uu]^-{\Displ\omega} [uu]
&\cdots 0 \ar[r]
& \unit(1) \ar[r]^-{\eta}
& \E_0(1) \ar[r]
& 0 \cdots
}}
\end{equation}
Here the quasi-isomorphism $s\colon\tilde\unit\to \unit$ corresponds to the fundamental exact sequence $\unit(1)\into \E_1\onto \unit$ in~$\Efil$ as in~\eqref{eq:fund_1}, and the map $\iota\colon \E_1\to \E_0(1)$ is the canonical morphism underlain by~$\id_{kC_2}$. We shall denote $\iota$ by~$\iota_1$ when we need to distinguish it from the similarly defined $\iota_0\colon \E_0\to \E_1$, as in the next lemma.
\end{Def}

\begin{Lem}
\label{Lem:cone-omega}%
In $\DE$, the following holds true.
\begin{enumerate}[\rm(a)]
\item
\label{it:cone-omega-a}%
The cone of~$\omega\colon \unit\to \invert$ is isomorphic to~$\cone(\iota_1\colon \E_1\to \E_0(1))$.
\smallbreak
\item
\label{it:cone-omega-b}%
The tt-ideal $\ideal{\cone(\omega)}$ contains~$\cone(\iota_0\colon \E_0\to \E_1)$, $\cone(\beta\colon \E_0\to \E_0(1))$ and $\cone(\beta\colon \E_1\to \E_1(1))$, where $\beta\colon\Id\to (1)$ is as in~\Cref{Not:beta}.
\smallbreak
\item
\label{it:cone-omega-c}%
The tt-ideal~$\ideal{\cone(\omega)}$ is equal to the tt-ideal~$\ideal{\cone(\beta\colon \E_1\to \E_1(1))}$.
\end{enumerate}
\end{Lem}

\begin{proof}
With notation as in~\eqref{eq:omega}, we have $\cone(\omega)\cong\cone(\tilde\omega)$ since~$s$ is an isomorphism in~$\DE$; furthermore we have $\cone(\tilde\omega)\cong$\,\mbox{$(\cdots 0\to \E_1 \xto{\iota_1} \E_0(1)\to 0\cdots)$}$=\cone(\iota_1\colon \E_1\to \E_0(1))$ in $\KE$ already. This gives~\eqref{it:cone-omega-a}. Tensoring this complex with $\E_0$ gives an object in~$\ideal{\cone(\omega)}$ which can be shown to be isomorphic to
\[
\xymatrix{
\E_0\otimes \cone(\iota_1) \ar@{}[r]\cong
& \cdots 0 \ar[r]
& \E_0\oplus \E_0(1) \ar[r]^-{\smat{\beta& 0\\0&\id}}
& \E_0(1) \oplus \E_0(1) \ar[r]
&
0 \cdots
}
\]
by \Cref{Prop:tens-e_l}. (On underlying objects, both $\E_0\otimes \E_1$ and $\E_0\otimes \E_0(1)$ are $\Epur\otimes \Epur$ for $\Epur=kC_2$ and we use $\bamma\colon \Epur\otimes \Epur\isoto \Epur\oplus \Epur$ from~\eqref{eq:bamma} to replace the tensor by the sum of filtered objects. The underlying map of the differential $\id_{\E_0}\otimes \iota_1$ is therefore $\bamma\circ \bamma\inv$, the identity of~$\Epur\oplus \Epur$. So the differential is indeed $\beta\colon \E_0\to \E_0(1)$ and $\id_{\E_0(1)}$ on the diagonal, when the weights are taken into account.) The above complex is isomorphic to $\cone(\beta\colon \E_0\to \E_0(1))$, which therefore belongs to~$\ideal{\cone(\omega)}$ in~$\DE$. Consider now the commutative diagram in~$\Afil$:
\begin{equation}
\label{eq:omega-iota}%
\xymatrix{
\E_0 \ar[r]_-{\iota_0} \ar@/^1.5em/[rr]^-{\beta_{\E_0}} & \E_1 \ar[r]^-{\iota_1} \ar@/_1.5em/[rr]_-{\beta_{\E_1}} & \E_0(1) \ar[r]^-{\iota_0(1)} & \E_1(1).
}
\end{equation}
Modulo the tt-ideal~$\ideal{\cone(\omega)}$ we have proved that $\beta_{\E_0}$ and~$\iota_1$ become isomorphisms. Hence so do~$\iota_0$ and $\iota_0(1)$ and~$\beta_{\E_1}$. This finishes the proof of~\eqref{it:cone-omega-b}. To prove~\eqref{it:cone-omega-c}, thanks to~\eqref{it:cone-omega-a} and~\eqref{it:cone-omega-b}, it only remains to show that $\cone(\iota_1\colon \E_1\to \E_0(1))$ belongs to $\ideal{\cone(\beta_{\E_1})}$. It suffices to prove that in the quotient~$\DE/\ideal{\cone(\beta_{\E_1})}$ the morphism $\iota_1\colon \E_1\to \E_0(1)$ is invertible. Using~\eqref{eq:omega-iota}, it reduces to proving that $\beta_{\E_0}$ is invertible in that quotient. This claim, that $\cone(\beta_{\E_0})$ belongs to~$\ideal{\cone(\beta_{\E_1})}$, is easy from \Cref{Prop:tens-e_l} again, which tells us that $\E_0\in\ideal{\E_1}$ and therefore $\cone(\beta_{\E_0})\cong \cone(\beta_\unit)\otimes \E_0\in\ideal{\cone(\beta_\unit)\otimes \E_0}\subseteq\ideal{\cone(\beta_\unit)\otimes \E_1}=\ideal{\cone(\beta_{\E_1})}$.
\end{proof}

\begin{Lem}
\label{Lem:omega-effective}%
Let $A\in\Ch(\Efil)$ be a complex of effective\,{\rm(\footnote{\,Recall from \Cref{Rem:a^>m} that $A\in\Afil$ is effective if $A^0=\fgt(A)$.})} objects in~$\Efil$. Let $n\ge 0$. Then the image of $\omega\potimes{n}\otimes 1_A\colon A\to \invert\potimes{n}\otimes A$ under $\rwz\colon \DE\to \KA$
\[
\rwz(\omega\potimes{n}\otimes 1_A)\colon \rwz(A)\isoto \rwz(\invert\potimes{n}\otimes A)
\]
is an isomorphism. In particular, $\rwz(\omega\potimes{n})\colon \unit\isoto \rwz(\invert\potimes{n})$ is an isomorphism.
\end{Lem}

\begin{proof}
The last statement is the case $A=\unit$. We want to reduce to the case $n=1$ but need to be careful since $\rwz$ is not a tensor functor. However, $\invert$ being degreewise effective, all objects in the following factorization of~$\omega\potimes{n}\otimes A$ are degreewise effective:
\[
\xymatrix{
A \ar[r]^-{\omega\otimes 1}
& \invert \otimes A \ar[r]^-{\omega\otimes 1}
& \cdots \ar[r]^-{\omega\otimes 1}
& \invert\potimes{n-1} \otimes A \ar[r]^-{\omega\otimes 1}
& \invert\potimes{n} \otimes A\,.
}
\]
So we can indeed assume $n=1$. Since $\rwz$ is triangulated, we need to show that $\rwz$ maps~$\cone(\omega\otimes 1_A)\cong\cone(\omega)\otimes A$ to zero. We have seen in \Cref{Lem:cone-omega}\,\eqref{it:cone-omega-a} that $\cone(\omega)\cong\cone(\iota_1)=(\cdots \to 0\to \E_1\xto{\iota_1}\E_0(1)\to 0\to \cdots)$ and \Cref{Cor:Afil-Frobenius} tells us that $\E_1$ and $\E_0(1)$ and all their $\otimes$-multiples in~$\Efil$ are injective. Consequently, $\cone(\iota_1)\otimes A$ is degreewise injective, hence its image under the right-derived functor $\rwz=\mathrm{R}(-)^0$ is $(\cone(\iota_1)\otimes A)^0$. Since $\cone(\iota_1)\otimes A$ is degreewise effective, we have
\[
\rwz(\cone(\omega)\otimes A)\cong\big(\cone(\iota_1)\otimes A\big)^0=\fgt(\cone(\iota_1)\otimes A)\cong\fgt(\cone(\iota_1))\otimes \fgt(A)
\]
since $\fgt$ is a tensor functor. But $\fgt(\cone(\iota_1))=(\cdots 0\to kC_2\xto{\id}kC_2\to 0\cdots)$ is clearly zero in~$\KA$ hence $\rwz(\cone(\omega\otimes A))=0$ as claimed.
\end{proof}

\begin{Def}
\label{Def:DEU}%
Consider the open piece of~$\Spc(\DE)$
\[
U=U(\cone(\omega))=\SET{\cP\in\Spc(\DE)}{\cone(\omega)\in\cP}
\]
`cut out by the section' $\omega\colon \unit\to \invert$ of the invertible~$\invert=\pwz(\invertpur)(1)$ of \Cref{Def:omega}, that is, $U$ is the open complement of~$\supp(\cone(\omega))$. Let us denote by
\begin{equation}
\label{eq:DEU}%
\xymatrix{\DEfil\ar@{->>}[r]^-{\quo}&\frac{\Displ\DEfil}{\Displ\ideal{\cone(\omega)}}=:\DEfil\restr{U}}
\end{equation}
the corresponding Verdier quotient.\,(\footnote{\,Technically, there is an idempotent-completion in the general definition of~$\cK\restr{U}$ but we are going to prove that the quotient $\DEfil/\ideal{\cone(\omega)}$ is already idempotent-complete.})

\end{Def}

\begin{Prop}
\label{Prop:DEU}%
The quotient $\quo\colon\DE\onto \DEU$ is a central localization in the sense of~\cite{balmer:sss} and~\cite[\S\,5]{gallauer:tt-fmod}, that is, it is the tensor-category $\DE[\omega\inv]$ initial among those receiving~$\DE$ and inverting~$\omega$. Explicitly, morphisms in~$\DEU$ are given by
\[
\Hom_{\DEU}(A,B)\cong\colim_{n\to \infty}\Hom_{\DE}(A,\invert\potimes{n}\otimes B)
\]
for all~$A,B\in\DE$, where the transition morphisms $\Hom_{\DE}(A,\invert\potimes{n}\otimes B)\to \Hom_{\DE}(A,\invert\potimes{(n+1)}\otimes B)$ are given by postcomposition with $\omega\otimes 1\colon \invert\potimes{n}\otimes B\to \invert\potimes{(n+1)}\otimes B$. To $f\colon A\to \invert\potimes{n}\otimes B$ in~$\DE$ in the colimit (on the right) corresponds the morphism $\xymatrix@C=2em{A \ar[r]^-{f} & \invert\potimes{n}\otimes B \ar[rr]^-{(\omega\potimes{n}\otimes 1)\inv} && B}$ in~$\DEU$ (on the left).
\end{Prop}

\begin{proof}
The $\otimes$-invertible object $\invert$ in~$\DE$ has the property that the swap of factors $(12)\colon \invert\otimes \invert\isoto \invert\otimes \invert$ is the identity (it is given by an element of~$k^\times$ of square one and $\mathrm{char}(k)=2$). Hence $\omega\otimes 1\colon \invert\potimes{n}\to \invert\potimes{(n+1)}$ can equally be $1\otimes\omega\otimes 1$ with $\omega$ in any place. By \cite[Thm.\,2.15]{balmer:sss}, we have in $\DE$ that
\begin{equation}
\label{eq:<cone(omega)>}%
\ideal{\cone(\omega)}=\SET{C\in \DE}{\omega\potimes{n}\otimes C=0\textrm{ for }n\gg0}.
\end{equation}
It easily follows (as in \cite[Lem.\,3.8]{balmer:sss}) that $\DEU$, which is by definition~$\DE/\ideal{\cone(\omega)}$, is also the localization~$\DE[\mathcal{S}\inv]$ with respect to the class of maps $\mathcal{S}=\SET{\omega\potimes{n}\otimes B}{n\ge 0,\ B\in \DE}$. The description of the latter as in the above statement is then a general fact; see~\cite[Prop.\,5.1]{gallauer:tt-fmod}.
\end{proof}

\begin{Thm}
\label{Thm:DEU}%
Let us denote by~$\barpwz$ the canonical tensor-triangulated functor
\begin{equation}
\label{eq:barpwz-equivalence}
  \vcenter{\xymatrix{
\barpwz\colon \ \Kb(\cA) \ar[r]^-{\pwz}
& \DEfil \ar@{->>}[r]^-{\quo}
& \DEU
}}
\end{equation}
composed of~$\pwz$ (`pure weight-zero', see~\Cref{Exa:filt-0}, applied degreewise) and the central localization $\quo$ of \Cref{Def:DEU}. Then $\barpwz$ is an equivalence.
\end{Thm}

\begin{proof}
Let us first show that $\barpwz$ is fully faithful. Let $M,N\in\KA$. We have
\[\kern-0.4em\begin{array}{ll}
\Hom_{\DEU}(\barpwz(M),\barpwz(N)) &
\\
\quad\cong \colim_{n\to \infty}\Hom_{\DE}(\pwz(M),\invert\potimes{n}\otimes \pwz(N)) & \textrm{by \Cref{Prop:DEU}}
\\
\quad\cong \colim_{n\to \infty}\Hom_{\KA}(M,\rwz(\invert\potimes{n}\otimes \pwz(N))) & \textrm{by \Cref{Prop:rwz}}
\\
\quad\cong \colim_{n\to \infty}\Hom_{\KA}(M,\rwz(\invert\potimes{n})\otimes N) & \textrm{by projection formula~\eqref{eq:projection-formula-rwz}}
\\
\quad\cong \colim_{n\to \infty}\Hom_{\KA}(M,N) & \textrm{by \Cref{Lem:omega-effective}}
\\
\quad = \Hom_{\KA}(M,N)\,.&
\end{array}\]
A detailed verification shows that the above isomorphism $\Hom_{\KA}(M,N)\isoto \Hom_{\DEU}(\barpwz(M),\barpwz(N))$ is indeed induced by~$\barpwz$, \ie our functor $\barpwz$ is fully faithful. So it suffices to show that the essential image of~$\barpwz$ contains generators of~$\DEU$ as a tt-category. Such generators can be chosen in~$\DE$, namely $\unit(1)$, $\unit(-1)$, $\E_0$ and $\E_1$ (\Cref{Cor:generators-DE}). Clearly, $\E_0=\pwz(kC_2)$ by definition. Also, since $\cone(\iota\colon \E_1\to \E_0(1)) \cong \cone(\omega)$ by \Cref{Lem:cone-omega}\,\eqref{it:cone-omega-a}, we have an isomorphism $\iota\colon \E_1\isoto \E_0(1)=\barpwz(kC_2)(1)$ in~$\DEU$. Therefore it suffices to prove that we have the following isomorphism (see \Cref{Prop:L^n} for $\invertpur$)
\begin{equation}
\label{eq:aux-barpwz-L}%
\barpwz(\invertpur\potimes{-1})\cong \unit(1)
\end{equation}
in~$\DEU$, which automatically implies $\barpwz(\invertpur)\cong\unit(-1)$ since $\barpwz$ is a tensor functor. To prove~\eqref{eq:aux-barpwz-L}, consider the following fraction in~$\KE$:
\[
\xymatrix@C=1.7em@R=1em{
\unit(1)=
&\cdots 0 \ar[r]
& \unit(1) \ar[r]^-{} \ar[d]_-{\eta}
& 0 \ar[r]^-{} \ar[d]^-{}
& \cdots
\\
B:=\ar@<.7em>@{<-}[u]\ar@<-.7em>@{<-}[d]
&\cdots 0 \ar[r]
& \E_1 \ar[r]^-{\eps} 
& \unit \ar[r] 
& 0 \cdots
\\
\pwz(\invertpur\potimes{-1})=
&\cdots 0 \ar[r]
& \E_0 \ar[r]^-{\eps} \ar[u]^-{\iota_0}
& \unit \ar[r] \ar@{=}[u]^-{}
& 0 \cdots
}
\]
The top map $\unit(1)\to B$ is already an isomorphism in~$\DEfil$ since its cone is the fundamental exact sequence~$\unit(1)\into \E_1\onto \unit$. The bottom map~$\pwz(\invertpur\potimes{-1})\to B$ becomes an isomorphism in~$\DEU$ for its cone is a shift of~$\cone(\iota_0\colon \E_0\to \E_1)$ which goes to zero in~$\DEU$ by~\Cref{Lem:cone-omega}\,\eqref{it:cone-omega-b}.
\end{proof}

We want to describe the inverse of the equivalence $\barpwz\colon \KA\to \DEU$.

\begin{Lem}
  \label{Lem:adjunction-localization}
  Let $F:\cK\adjto\cL:G$ be an adjunction of $\otimes$-categories with $F$ a $\otimes$-functor. Let $\omega:\unit\to u$ be a map in $\cL$, where $u$ is $\otimes$-invertible with trivial switch, \ie $(12)=\id_{u\otimes u}$.
Assume that for every $b\in\cL$, the following sequence is stationary in~$\cK$, meaning that transition maps become isomorphisms for~$n\gg0$:
\begin{equation}
\label{eq:G(u^n-b)}%
\xymatrix@C=1.8em{
G(b) \ar[rr]^-{G(\omega\otimes 1)} && \cdots \ar[r]
& G(u\potimes{n}\otimes b) \ar[rr]^-{G(\omega\otimes 1)}
&& G(u\potimes{(n+1)}\otimes b) \ar[r]
& \cdots }
\end{equation}
Consider the localization $\cL\onto \cL[\omega\inv]$ as a tensor-category. Then the composite $\overline{F}:\cK\xto{F}\cL\xto{\quo}\cL[\omega\inv]$ has a right adjoint given by
\[
\widetilde{G}(b):=\colim_n G(u\potimes{n}\otimes b);
\]
in other words, $\widetilde{G}(b)$ is $G(u\potimes{n}\otimes b)$ for $n\gg0$ large enough (depending on~$b$).
\end{Lem}
\begin{proof}
The category $\cL[\omega\inv]$ is the localization of $\cL$ with respect to the class of morphisms $b\xto{\omega}u\otimes b$. It is clear by construction that $\widetilde{G}$ inverts these morphisms and therefore passes to a well-defined functor on the localization. Let $a\in\cK$ and~$b\in\cL$. There are natural isomorphisms
\begin{align*}
  \Hom_{\cK}(a,\widetilde{G}(b))&\cong\Hom_{\cK}(a,\colim_n G(u\potimes{n}\otimes b))\cong\colim_n\Hom_{\cK}(a,G(u\potimes{n}\otimes b))\\
                     &\cong\colim_n\Hom_{\cL}(F(a),u\potimes{n}\otimes b)\cong\Hom_{\cL[\omega\inv]}(\overline{F}(a),b).
\end{align*}
The second isomorphism uses that the sequence~\eqref{eq:G(u^n-b)} is stationary.
\end{proof}

\begin{Cor}
\label{Cor:barpwzinv}%
The quasi-inverse $\barpwzinv:\DEU\to\Kb(\cA)$ to the equivalence $\barpwz:\Kb(\cA)\isoto\DEU$ of~\eqref{eq:barpwz-equivalence} is given for all~$B\in \DEU$ by
\begin{equation}
\label{eq:barpwzinv}%
\barpwzinv(B)= \colim_n \big(\cdots \rwz(\invert\potimes{n}\otimes B)\xto{\rwz(\omega\otimes1)}\rwz(\invert\potimes{(n+1)}\otimes B)\to\cdots\big)\kern-1em
\end{equation}
which is the colimit of a stationary sequence, \ie simply $\rwz(\invert\potimes{n}\otimes B)$ for $n\gg0$. More precisely, it suffices to take~$n$ such that $B(n)$ is effective in each degree. In particular, if $B$ is effective in each degree then $\barpwzinv(B)\cong \rwz(B)$.
\end{Cor}
\begin{proof}
To apply \Cref{Lem:adjunction-localization}, it suffices to show that for each $B\in\DE$ fixed, the morphism $\rwz(\omega\otimes \invert\potimes{n}\otimes B)$ is an isomorphism in~$\KA$ for $n\gg0$. Since $\invert\potimes{n}\cong\pwz(\invertpur\potimes{n})(n)$ by definition, we see that for $n\gg0$ large enough $A:=\invert\potimes{n}\otimes B$ is a complex of effective objects, to which we can apply \Cref{Lem:omega-effective}.
\end{proof}

\begin{Rem}
  \label{Rem:trwz-generators}%
  For any bounded complex $B\in\Ch(\Efil)$ there exists $n\gg 0$ such that $B(n)$ is effective (in each degree) and \Cref{Cor:barpwzinv} combined with the definition of~$\invert$ in~\eqref{eq:u} and the projection formula~\eqref{eq:projection-formula-rwz} yield in~$\KA$
\begin{equation}
\label{eq:twist-untwist}%
    \barpwzinv(B)\cong\rwz(\invert\potimes{n}\otimes B)\cong\rwz(B(n)\otimes\pwz(\invertpur)\potimes{n})\cong\rwz(B(n))\otimes \invertpur\potimes{n}.
\end{equation}
In particular, the images in $\Kb(\cA)$ of our favorite objects under the tt-equivalence $\barpwzinv$ may easily be computed using~\Cref{Rem:bamma-basics}:
  \begin{enumerate}[\rm(a)]
  \item $\barpwzinv(\unit(n))\cong \invertpur\potimes{-n}$, for all $n\in\bbZ$. See also~\eqref{eq:aux-barpwz-L}.
  \item $\barpwzinv(\E_{\ell})\cong \rwz(\E_{\ell})\cong kC_2\oplus (\fundpur^{\ell-1}[-\ell])$, where $\fundpur^{\inv}=\fundpur^0=0$ and $\fundpur^{m}$ (for $m\ge 1$) is the $m^\textrm{th}$ iterated splice of the extension $\fundpur$ from~\eqref{eq:fund-pur}:
    \[
      \cdots 0\to k\xto{\eta} kC_2\xto{\eta\eps} \cdots \xto{\eta\eps} kC_2 \xto{\eps} k \to 0 \cdots
    \]
    in homological degrees from $m+1$ down to~$0$ (hence with~$kC_2$ in $m$ places).
  \item $\barpwzinv(\cone(\beta\colon\unit\to \unit(1)))\cong \rwz(\cone(\beta))\cong \cone(k\xto{\eta} \invertpur\potimes{-1})\cong \fundpur[-1]$.
  \end{enumerate}
\end{Rem}

To wrap up this section, let us try to build some conceptual understanding of the equivalence $\barpwzinv\colon \DEU\isoto \KA$ of \Cref{Cor:barpwzinv}. Perhaps the first property is that $\barpwzinv$ as a refinement of the functor forgetting the filtration.

\begin{Cor}
\label{Cor:tfgt}%
The following diagram commutes up to isomorphism
\[
\xymatrix{
\DE \ar@{->>}[r]^-{\quo} \ar[d]_-{\fgt}
& \DEU \ar[d]^-{\barpwzinv}
\\
\DA
& \KA \ar@{->>}^-{\quo}[l]
}
\]
where $\fgt\colon \DE\to \DA$ is induced by the exact functor $\fgt\colon \Efil\to \cA$.
\end{Cor}

\begin{proof}
Recall from \Cref{Rem:fgt-gr-exact} that $\fgt\colon \Efil\to \cA$ is exact when the target has its abelian category structure (not the split one), and so is $(-)^0\colon \Efil\to \cA$. Furthermore, for every $A\in\Efil$, we have $\fgt(A)=(A(n))^0$ for $n\gg0$. Finally, in~$\DA$, we have $\invertpur\potimes{n}\cong\unit$ since $\invertpur\potimes{n}$ is a resolution of~$k$ (see~\Cref{Prop:L^n}). Hence the image of the isomorphism~\eqref{eq:twist-untwist} in~$\DA$ gives us
\[
\barpwz\inv(B)\cong\rwz(B(n))\otimes \invertpur\potimes{n}\cong\rwz(B(n))\cong \fgt(B)
\]
for $n\gg0$.
\end{proof}

\begin{Def}
\label{Def:tfgt}%
In view of \Cref{Cor:tfgt}, we call the composite tt-functor
\[
\tfgt:=\barpwzinv\circ\quo\colon \DE\onto \DEU\isoto \KA
\]
the \emph{twisted forgetful functor}.
\end{Def}

\begin{Rem}
\label{Rem:tfgt}%
Let us consider another heuristic for~$\tfgt\colon \DE\to \KA$, building on the initial intuition proposed at the beginning of the section. In \Cref{Rem:pwz-(-)^0}, we saw that the pure-weight-zero tt-functor $\pwz\colon \cA\to \Efil$ admits a right adjoint $(-)^0\colon \Efil\to \cA$ that is unfortunately not exact. Right-deriving this right adjoint led us in \Cref{Prop:rwz} to the functor $\rwz=\mathrm{R}(-)^0\colon \DE\to \KA$, itself right adjoint to~$\pwz\colon \KA\to \DE$. In other words, $\rwz\colon \DE\to \KA$ has its own justification, independently of the open~$U$ and the map~$\omega\colon \unit\to \invert$ of the subsequent \Cref{Def:omega}. Unfortunately, $\rwz$ is not a tensor functor. For instance, $\rwz(\unit(-1))=0$. (The functor $\rwz$ is at least lax-monoidal.) On the other hand, $\tfgt\colon \DE\to \KA$ \emph{is} a tensor functor and there exists a natural transformation (of lax-monoidal functors)
\[
\omega^\infty\colon \rwz\to \tfgt
\]
given by~$\rwz\to \colim_{n\ge0}\rwz(\omega\potimes{n}\otimes -)$. See~\eqref{eq:barpwzinv}. Also, $\omega^\infty_A\colon \rwz(A)\to \tfgt(A)$ is an isomorphism whenever~$A$ is a degreewise effective complex in~$\Efil$ by~\Cref{Lem:omega-effective}.

These properties characterize~$\tfgt$. Indeed, suppose that $\alpha\colon \rwz\to G$ is a tensorial approximation of~$\rwz$, that is, $\alpha$ is a natural transformation and $G\colon \DE\to \KA$ is a tt-functor. Suppose furthermore that $\alpha\colon \rwz(A)\isoto G(A)$ is an isomorphism on degreewise effective complexes~$A$. (This last property is expected of a functor `forgetting the filtration': On effective objects it should agree with $(-)^0$ and we know that $\rwz$ is the `derived version' of~$(-)^0$.) Then $G$ is isomorphic to~$\tfgt$. More precisely, the top horizontal sequence in the following diagram
\[
\xymatrix@C=1.7em{
\rwz(A) \ar[rr]^-{\rwz(\omega\otimes1)} \ar[d]_-{\alpha}
&& \cdots \ar[rr]^-{\rwz(\omega\otimes1)}
&& \rwz(\invert\potimes{n}\otimes A)\ar[rr]^-{\rwz(\omega\otimes1)} \ar[d]_-{\alpha}
&& \rwz(\invert\potimes{(n+1)}\otimes A) \ar[r] \ar[d]_-{\alpha}
& \cdots
\\
G(A) \ar[rr]^-{G(\omega\otimes1)}_-{\simeq}
&& \cdots \ar[rr]^-{G(\omega\otimes1)}_-{\simeq}
&& G(\invert\potimes{n}\otimes A)\ar[rr]^-{G(\omega\otimes1)}_-{\simeq}
&& G(\invert\potimes{(n+1)}\otimes A) \ar[r]_-{\simeq}
& \cdots
}
\]
becomes stationary by~\Cref{Lem:omega-effective}. We claim that the bottom sequence consists of isomorphisms. Indeed, since $G$ is a tensor functor, it suffices to check that $G(\omega)$ is an isomorphism. Now $\omega\colon \unit\to \invert=\pwz(L)(1)$ is a map between effective complexes and $\rwz(\omega)$ is an isomorphism (\Cref{Lem:omega-effective}), so our assumption about~$\alpha$ forces~$G(\omega)$ to be one too. Finally, for $n\gg0$, the above vertical maps~$\alpha\colon \rwz(\invert\potimes{n}\otimes A)\to G(\invert\potimes{n}\otimes A)$ become isomorphisms by assumption (since again $\invert\potimes{n}\otimes A$ becomes effective). Taking the colimit of the above (stationary) sequences yields a natural isomorphism $\tfgt(A)\cong\colim_n\rwz(\invert\potimes{n}\otimes A)\isoto \colim_n G(\invert\potimes{n}\otimes A)\simeq G(A)$.

In other words, if we follow the intuition that a forget-the-filtration functor $\DE\to \KA$ should be exact, tensorial, and should agree as much as possible with~$(-)^0$ on effective objects, then we naturally construct the tt-functor~$\tfgt$.
\end{Rem}

\bigbreak\goodbreak
\section{Main representation-theoretic result}
\label{sec:Spc-DE}%
\medbreak

We are now ready to put the pieces of \Cref{part:I} together and determine the space $\Spc(\DEfil)$. Recall that $\cA$ stands for~$\cA=\mmod{kC_2}$ and that $\Efil$ is the category of filtered objects in~$\cA$ with the Frobenius tensor-exact structure of \Cref{sec:filt-Frobenius}.

\begin{Rem}
\label{Rem:gr-Q}%
The strategy will rely on the interplay between the two tt-functors $\gr$ and~$\tfgt\colon\DE\to \KA$ displayed on the left-hand diagram\,:
\[
\vcenter{\xymatrix@R=1em@C=2em{
&& \KA
\\
\DE \ar[rru]^-{\gr} \ar@{->>}[rd]_-{\quo} \ar@{->>}@/^.8em/[rrd]^-{\tfgt}
\\
& \DEU
& \KA \ar@{->}[l]^(.45){\barpwz^{\vphantom{I}}}_(.45){\simeq}
}}
\overset{\Spc}{\leadsto}\quad
\vcenter{\xymatrix@R=1em@C=2em{
& \Spc(\KA)
\\
\Spc(\DE) \ar@{<-}[ru]^-{\Spc(\gr)} \ar@{<-}[rd]_-{\Spc(\tfgt)}
\\
& \Spc(\KA)
}}
\]
Here $\gr$ is the total-graded, as in \Cref{Lem:gr-conservative-and-section}, and $\tfgt=(\barpwz)\inv\circ \quo$ is the `twisted-forgetful functor' of \Cref{Def:tfgt}, \ie the central localization corresponding to the open~$U=U(\cone(\omega))$ of \Cref{Def:DEU} followed by the inverse of the equivalence $\barpwz$ of \Cref{Thm:DEU}. (See heuristics about~$\tfgt$ in \Cref{Rem:tfgt}.)

Applying the contravariant functor~$\Spc(-)$ we get the above right-hand diagram of spaces. We know the source $\Spc(\KA)$ of those two maps, by \Cref{Thm:Spc(Kb(kC_2))}. The gist of the argument is that the images of those two maps form a partition of~$\Spc(\DE)$. More precisely, the bottom one has image~$U$ (that is easy, by \Cref{Rmd:Spc-quo}) and the top map, $\Spc(\gr)$ has closed image equal to the complement of~$U$ (that will require proof). Then we need to understand how the open piece and the closed complement attach together topologically.

Some critical filtered objects in~$\Efil$ are $\E_0=\pwz(kC_2)$, which is just~$kC_2$ in pure weight zero, and $\E_1$ which is $k\hook kC_2$ in weight one and zero; see~\eqref{eq:e_l}. Recall also $\fund_0$ which is the complex $0\to \unit \into \E_0\onto \unit \to 0$, \ie the basic extension $\fundpur$ of~\eqref{eq:fund-pur}, in pure weight zero, which is not exact in~$\Efil$ and therefore defines a complex in~$\DE$, that we place in homological degrees two, one and zero.
\end{Rem}

Let us summarize the basic geography:
\begin{Prop}
\label{Prop:geography}%
We have a set partition
\[
\Spc(\DE)=U\sqcup Z
\]
where the open $U=U(\cone(\omega))=\SET{\cP}{\cone(\omega)\in\cP}$ is the complement of the closed $Z=\supp(\cone(\omega))$ for the morphism~$\omega\colon \unit\to \pwz(\invertpur)(1)$ described in~\eqref{eq:omega}. Moreover, this closed subset~$Z$ is also the support of the object
\begin{equation}
\label{eq:T}%
\T=\cone(\beta\colon \E_1\to \E_1(1))=\cone(\beta\colon \unit\to \unit(1))\otimes \E_1
\end{equation}
for the map~$\beta\colon \Id\to (1)$ as in \Cref{Not:beta}.
\end{Prop}

\begin{proof}
The decomposition $\SpcK=U(A)\,\sqcup\,\supp(A)$ holds for any object $A$ in any tt-category~$\cK$. The two objects $\cone(\omega)$ and $\cone(\beta\colon \E_1\to \E_1(1))$ have the same support because they generate the same tt-ideal by \Cref{Lem:cone-omega}\,\eqref{it:cone-omega-c}.
\end{proof}

The technical crux of the matter is the following result which will allow us to show that $\gr:\DEfil\to\Kb(\cA)$ catches all the points in $Z$.
\begin{KLem}
\label{Lem:gr-nil}%
Let $f\colon \unit\to A$ be a morphism in~$\DEfil$ such that $\gr^0(f)\colon k \to \gr^0(A)$ is zero in~$\Kb(\cA)$. Then $f\potimes{2}\otimes\T$ is zero, where~$\T$ is as in~\eqref{eq:T}.
\end{KLem}

\begin{proof}
The morphism $f$ is represented by a fraction $\unit\xrightarrow{f'}A'\xleftarrow{s}A$ with $f'$ and $s$ maps in $\Ch(\Afil)$ and the complex $\cone(s)$ is acyclic in the exact structure~$\Efil$. In particular, $\gr^0(s)$ is an isomorphism and we deduce that $\gr^0(f')=0$ as well. Moreover, if $(f')\potimes{2}\otimes\T$ is zero in $\DEfil$ then so is $f\potimes{2}\otimes\T$. Hence we will assume without loss of generality that $f$ is represented by a morphism $\unit\to A$ in $\Ch(\Afil)$.

The statement is a consequence of the following claim: The hypothesis $\gr^0(f)=0$ forces the composite at the top of the following diagram to factor in~$\Kb(\Afil)$ via $\beta$
\begin{equation}
\label{eq:aux-nil-reduce}%
\vcenter{\xymatrix@R=1.5em{
\E_1\otimes \E_1^\vee \ar[r]^-{\ev} \ar@{-->}[rrd]
& \unit \ar[r]^-{f\potimes{2}}
& A\potimes{2}
\\
&& A\potimes{2}(-1) \ar[u]_-{\beta}
}}
\end{equation}
Indeed, by duality this is equivalent to a factorization as on the left-hand side below
\[
\vcenter{\xymatrix@R=1.5em{
\E_1\ar[r]^-{1\otimes f\potimes{2}} \ar@{-->}[rd]
& \E_1\otimes A\potimes{2}
\\
& \E_1\otimes A\potimes{2}(-1) \ar[u]_-{1\otimes\beta}
}}
\Longrightarrow\quad
\vcenter{\xymatrix@C=4em@R=1.5em{
\E_1\otimes\cone(\beta)\ar[r]^-{1\otimes f\potimes{2}\otimes1} \ar@{-->}[rd]
& \E_1\otimes A\potimes{2}\otimes\cone(\beta)
\\
& \kern-2em \E_1\otimes A\potimes{2}(-1)\otimes\cone(\beta) \ar[u]_-{1\otimes\beta\otimes 1}
}}
\]
from which we get a factorization as on the right-hand side by tensoring with~$\cone(\beta)$.
But the vertical arrow in this last diagram is zero in $\Kb(\Afil)$, because the map of complexes $\beta\otimes\cone(\beta):\cone(\beta)(-1)\to\cone(\beta)$ is null-homotopic (with $\id$ as homotopy). We then conclude that the top map $\E_1\otimes f\potimes{2}\otimes \cone(\beta)$ is zero as well, as wanted. So we are indeed reduced to prove the claimed factorization in~\eqref{eq:aux-nil-reduce}.

Since $\E_1\otimes \E_1^\vee\cong \E_1\oplus \E_1(-1)$ and every map $\E_1(-1)\to \unit$ factors through~$\beta$, it suffices to prove that the following horizontal composite factors in~$\Kb(\Afil)$ through~$\beta$:
\[
\xymatrix@R=1.5em{
\E_1 \ar[r]^-{\eps} \ar@{-->}[rrd]
& \unit \ar[r]^-{f\potimes{2}}
& A\potimes{2}
\\
&& A\potimes{2}(-1). \ar[u]_-{\beta}
}
\]
Note that we reduced to the homotopy category $\Kb(\Afil)$ of the tensor category~$\Afil$. In particular we do not use the exact category structure~$\Efil$ in the rest of the proof.

For~$B\in\Ch(\Afil)$, with differential~$d\colon B_i\to B_{i-1}$, we have explicit descriptions of maps of complexes from~$\unit$ and from~$\E_1$ to~$B$, and what it means to factor via $\beta$:
\begin{enumerate}[(1)]
\item
\label{it:f:1->y}%
A morphism $f\colon\unit=\pwz(k)\to B$ amounts to picking an element $a\in B_0^0$ such that $(1+\sigma)a=0$ (to be $kC_2$-linear) and such that $d(a)=0$ (to be a morphism of complexes).
\item
\label{it:f:e_1->y}%
A morphism $\E_1\to B$ amounts to picking an element $a\in B_0^0$ such that $(1+\sigma)a\in B_0^1$ (so that $k=(\E_1)^1$ maps to weight 1) and still such that $d(a)=0$. For $f\colon \unit \to B$ as in~\eqref{it:f:1->y}, the morphism $f\eps\colon \E_1\to B$ is given by the same~$a$.
\item
\label{it:f:e_1-beta-y}%
A morphism $\E_1\to B$ as in~\eqref{it:f:e_1->y} factors via $\beta\colon B(-1)\to B$ if (and only if) $a\in B_0^1$ and $(1+\sigma)\in B_0^2$. Note that the condition $d(a)=0$ holds automatically.
\item
\label{it:f-homotopy}%
For two morphisms $\E_1\to B$ given by $a,a'\in B_0^0$ as in~\eqref{it:f:e_1->y}, a homotopy $h$ between them amounts to picking $h\in B_1^0$ such that $(1+\sigma)h\in B_1^1$ (that is just a morphism $\E_1\to B_1$) with the property that $a=a'+d(h)$.
\end{enumerate}

Our $f\colon \unit\to A$ is given, via~\eqref{it:f:1->y} for $B=A$, by an element $a\in A_0^0$ such that
\begin{equation}
\label{eq:aux-a}%
(1+\sigma)a=0
\qquadtext{and}
d(a)=0.
\end{equation}
Note right away that the morphism $f\potimes{2}\colon \unit\to A\potimes{2}$ is simply given by $a\otimes a\in (A\potimes{2})_0^0$ when we apply~\eqref{it:f:1->y} for $B=A\potimes{2}$. Similarly, $f\potimes{2}\circ\eps\colon \E_1\to A\potimes{2}$ is also given by~$a\otimes a$ in the description of~\eqref{it:f:e_1->y} for $B=A\potimes{2}$.

Now let us unpack the information about $\gr^0(f)\colon k \to \gr^0(A)$ being zero in~$\KA$. This homotopy amounts to the existence of $\bar b\in A_1^0/A_1^1$ such that $(1+\sigma)\bar b=0$ and such that $d(\bar b)=\bar a$ in~$A_0^0/A_0^1$. Picking $b\in A_1^0$ representing~$\bar b$, this information reads
\begin{equation}
\label{eq:aux-b}%
(1+\sigma)b\in A_1^1
\qquadtext{and}
c:=a+d(b)\in A_0^1.
\end{equation}
(We use characteristic~2 and do not write signs.) Note that we have $a=d(b)+c$. Also note that $d(c)=d(a)+d^2(b)=0$.

Consider now the element
\[
h=(b\otimes a)+(c\otimes b)
\]
in~$(A\potimes{2})_1$. Its homological degree is indeed~1 because $b$ is degree~1 and $a$ and~$c$ are degree~0. For the moment, we consider~$h$ as `effective', that is, as an element of~$(A\potimes{2})_1^0$. We claim that $(1+\sigma)h$ is of strictly positive weight, as in~\eqref{it:f-homotopy} for the object~$B=A\potimes{2}$. Note that $c\otimes b$ already is of strictly positive weight since~$c$ is and $b$ is effective. So, to show that $(1+\sigma)h$ is of strictly positive weight, it suffices to check this for~$(1+\sigma)(b\otimes a)$. Using that $\sigma a=a$ and that $\sigma$ acts diagonally on the tensor, we have $(1+\sigma)(b\otimes a)=((1+\sigma)b)\otimes a$ which is indeed of weight~$\geq 1$ since $(1+\sigma)b$ is by~\eqref{eq:aux-b} and $a$ is effective. In short, $h$ defines a homotopy for morphisms~$\E_1\to A\potimes{2}$. Let us now modify $f\potimes{2}\circ \eps\colon \E_1\to A\potimes{2}$, given by~$a\potimes{2}$, with the homotopy given by~$h$, as in~\eqref{it:f-homotopy}. We compute using Leibniz (without signs), together with $d(a)=0$ and $d(c)=0$, and finally $a+db=c$:
\begin{align*}
a\potimes{2}+dh & =
a\potimes{2}+d(b\otimes a)+d(c\otimes b)= a\otimes a+db\otimes a+c\otimes db
\\
& = (a+db)\otimes a+c\otimes db = c\otimes (a+db)=c\potimes{2}.
\end{align*}
In other words, the morphism $f\potimes{2}\circ \eps\colon \E_1\to A\potimes{2}$ is homotopic to the morphism $\E_1\to A\potimes{2}$ given by~$c\potimes{2}$. Now since $c\in A_0^1$, we see that $c\potimes{2}$ belongs to $A_0^2$, \ie is of weight~$\geq 2$. In particular, $(1+\sigma)c\potimes{2}$ also is of weight~$\geq 2$. In other words, by~\eqref{it:f:e_1-beta-y} for the object $B=A\potimes{2}$, the morphism given by~$c\potimes{2}$ does factor via $\beta\colon A\potimes{2}(-1)\to A\potimes{2}$.
\end{proof}

To use \Cref{Lem:gr-nil}, we need the following extension of \cite[Thm.\,1.3]{balmer:surjectivity}.
\begin{Cor}
  \label{Cor:surjectivity-support}%
  Let $F:\cK\to\cL$ be a tt-functor between tt-categories and assume that $\cK$ is rigid. Let $t\in\cK$ be an object and assume that $F$ \emph{detects $\otimes$-nilpotence on~$t$}, in the sense that if $f$ is a morphism in $\cK$ and $F(f)=0$ then $f\potimes{n}\otimes t=0$ for some $n\geq 1$. Then the image of $\Spc(F):\Spc(\cL)\to\Spc(\cK)$ contains $\supp(t)$.
\end{Cor}
\begin{proof}
Consider the tt-functor $F'\colon \cK\to (\cK/\ideal{t})\times \cL$ given by~$\quo\colon \cK\onto \cK/\ideal{t}$ in the first component and by~$F$ in the second. If~$f\colon x\to y$ is such that $F'(f)=0$ then in particular $f\mapsto 0$ in~$\cK/\ideal{t}$, hence it factors $f=(x\xto{g}z\xto{h}y)$ via an object~$z\in\ideal{t}$. On the other hand, $F(f)=0$. Thus by hypothesis the morphism~$f$ is nilpotent on~$t$, and therefore on any object of~$\ideal{t}$, like our~$z$; see~\cite[Prop.\,2.12]{balmer:sss}. It follows that~$f\potimes{(n+1)}$, which factors as follows
\[
\xymatrix{x\potimes{(n+1)} \ar[r]^-{1\otimes g}
& x\potimes{n}\otimes z \ar[rr]^-{f\potimes{n}\otimes 1_z}
&& y\potimes{n}\otimes z \ar[r]^-{1\otimes h}
& y\potimes{(n+1)}}
\]
is zero for~$n\gg0$. In short, $F'(f)=0$ forces $f\potimes{n}=0$ for $n\gg0$, \ie the functor $F'$ detects nilpotence. Hence by \cite[Thm.\,1.3]{balmer:surjectivity}, the spectrum $\SpcK$ is covered by the image under $\Spc(F')$ of $\Spc((\cK/\ideal{t})\times\cL)=\Spc(\cK/\ideal{t})\sqcup\Spc(\cL)$. We know by \Cref{Rmd:Spc-quo} that the first component $\Spc(\cK/\ideal{t})\isoto U(t)\subset\SpcK$ misses $\supp(t)$ entirely. Hence it must be the other component of~$\Spc(F')$, namely $\Spc(F)\colon \Spc(\cL)\to \SpcK$, that covers~$\supp(t)$.
\end{proof}

We can now prove our main result, on the representation-theoretic side.

\begin{Thm}
\label{Thm:Spc-DE}%
The spectrum $\Spc(\DEfil)$ of the tt-category $\cK=\DE$ is the following six-point topological space (see \Cref{Rmd:Spc}):
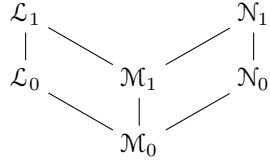
\[
\xymatrix@R=1em{
\prLs
&& \prNs
\\
\prL \ar@{-}[u]
& \prMs \ar@{-}[lu] \ar@{-}[ru]
& \prN \ar@{-}[u]
\\
& \prM \ar@{-}[u] \ar@{-}[lu] \ar@{-}[ru]
}
\]
More precisely, using notation as in~\Cref{Rem:gr-Q} and \Cref{Prop:geography}, we have
\begin{align}
\label{eq:M_0}%
\prM=\ideal{\T,\E_0,\fund_0}
&\quadtext{is the kernel of}
\rsd{\prM}\colon \cK \xto{\tfgt} \KA \xto{\rsd{\cM}}\mmod{k}
\\
\label{eq:L_0}%
\prL=\ideal{\T,\E_0}
&\quadtext{is the kernel of}
\rsd{\prL}\colon \cK \xto{\tfgt} \KA \xto{\rsd{\cL}}\Db(k)
\\
\label{eq:N_0}%
\prN=\ideal{\T,\fund_0}
&\quadtext{is the kernel of}
\rsd{\prN}\colon \cK \xto{\tfgt} \KA \xto{\rsd{\cN}}\Db(k)
\\
\label{eq:M_1}%
\prMs=\ideal{\E_0,\fund_0}
&\quadtext{is the kernel of}
\rsd{\prMs}\colon \cK \xto{\gr} \KA \xto{\rsd{\cM}}\mmod{k}
\\
\label{eq:L_1}%
\prLs=\ideal{\E_0}
&\quadtext{is the kernel of}
\rsd{\prLs}\colon \cK \xto{\gr} \KA \xto{\rsd{\cL}}\Db(k)
\\
\label{eq:N_1}%
\prNs=\ideal{\fund_0}
&\quadtext{is the kernel of}
\rsd{\prNs}\colon \cK \xto{\gr} \KA \xto{\rsd{\cN}}\Db(k)
\end{align}
where the `tt-residue fields' $\rsd{\cL}$, $\rsd{\cM}$ and $\rsd{\cN}$ are those of \Cref{Cor:fields-KA}.
\end{Thm}
\begin{proof}
We refer to the basic geography $\Spc(\DE)=U\sqcup\supp(\T)$ of \Cref{Prop:geography}. The open piece~$U$ is straightforward to describe in view of \Cref{Rmd:Spc-quo} and the equivalence $\barpwz\colon \KA\isoto\DEU$ of \Cref{Thm:DEU}. We have $U=\SET{\tfgt\inv(\cP)}{\cP\in \Spc(\KA)}$, that is, $U$ consists of three points $\prL:=\tfgt^{-1}(\cL)=\Ker(\rsd{\cL}\circ \tfgt)$, $\prM:=\tfgt^{-1}(\cM)=\Ker(\rsd{\cM}\circ \tfgt)$ and $\prN:=\tfgt^{-1}(\cN)=\Ker(\rsd{\cN}\circ \tfgt)$ with the specialization relations between them as depicted. Furthermore, as $\tfgt\colon \DE\onto \KA$ is (equivalent to) the localization at~$\ideal{\cone(\omega)}=\ideal{\T}$ and since we have generators of the quotients~$\prL/\ideal{\T}=\cL=\ideal{\E_0}$, $\prM/\ideal{\T}=\cM=\ideal{\E_0,\fund_0}$ and $\prN/\ideal{\T}=\cN=\ideal{\fund_0}$, we have obvious generators of~$\prM$, $\prL$ and~$\prN$ as in~\eqref{eq:M_0}-\eqref{eq:N_0}.

The closed complement of~$U$ is~$\supp(\T)=\SET{\cP}{\T\notin\cP}$ by \Cref{Prop:geography}. We want to show that this closed complement is exactly $\Img(\Spc(\gr))$. The functor $\gr:\DEfil\to\Kb(\cA)$ has a section (\Cref{Lem:gr-conservative-and-section}). Hence the map $\Spc(\gr):\Spc(\Kb(\cA))\to\Spc(\DEfil)$ is a homeomorphism onto its image, which consists of three points $\prLs:=\gr^{-1}(\cL)=\Ker(\rsd{\cL}\circ\gr)$, $\prMs:=\gr^{-1}(\cM)=\Ker(\rsd{\cM}\circ\gr)$, $\prNs:=\gr^{-1}(\cN)=\Ker(\rsd{\cN}\circ\gr)$ with the specialization relations between them as depicted. Computing $\gr(\E_0)=kC_2$ and $\gr(\fund_0)=\fundpur$, we easily have
\begin{equation}
\label{eq:aux-LMN}%
\prLs\supseteq \ideal{\E_0},\qquad
\prMs\supseteq \ideal{\E_0,\fund_0}
\quadtext{and}
\prNs\supseteq \ideal{\fund_0}
\end{equation}
but we do not yet know that these are equalities. Since $\gr(\T)\cong \gr(\cone(\beta))\otimes \gr(\E_1)\cong (k\oplus k[1])\otimes (k\oplus k)$ is a sum of invertibles, these points $\prLs$, $\prMs$, $\prNs$ do not contain~$\T$, hence belong to $\supp(\T)$. To show that this inclusion $\Img(\Spc(\gr))\subseteq\supp(\T)$ is an equality, we can use \Cref{Cor:surjectivity-support}, \ie it suffices to verify that $\gr$ `detects $\otimes$-nilpotence on~$\T$'. Thus let $f:B\to A$ be a morphism in $\DEfil$ such that $\gr(f)=0$. We would like to show that $f$ is $\otimes$-nilpotent on~$\T$. By rigidity, we may assume $B=\unit$ in which case the statement is proved in the \Cref{Lem:gr-nil}.

At this stage, we know that the spectrum of $\DEfil$ has exactly six points and at least the specialization relations as follows:
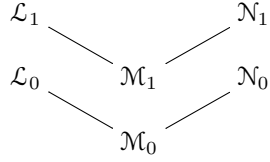
\[
\xymatrix@R=1em{
\prLs
&& \prNs
\\
\prL
& \prMs \ar@{-}[lu] \ar@{-}[ru]
& \prN
\\
& \prM \ar@{-}[lu] \ar@{-}[ru]
}
\]
(We also know that inside each of the two ``V-shapes'' in this picture there are no other specialization relations.) It remains to prove the `vertical' specialization relations and to prove that the inclusions in~\eqref{eq:aux-LMN} are equalities.

By \Cref{Lem:gr-conservative-and-section}, the functor $\gr$ is conservative  and thus the image of the map $\Spc(\gr)$ contains all closed points of $\Spc(\DEfil)$, by~\cite[Thm.\,1.2]{balmer:surjectivity}. We conclude that $\prL$ and $\prN$ are not closed points. Since $\E_0\in \prLs\smallsetminus\prN$ we see that $\prLs\notin\adhpt{\prN}$ and, \afortiori, $\prMs\notin\adhpt{\prN}$. Hence $\overline{\{\prN\}}=\{\prN,\prNs\}$. Similarly, $\fund_0\in\prL\smallsetminus\prNs$ and we deduce that $\overline{\{\prL\}}=\{\prL,\prLs\}$.

To prove $\prMs\in\adhpt{\prM}$, that is $\prMs\subset\prM$, we provide an intermediate tt-category, between $\DE$ and the residue fields~$\kappa(\prMs)=\mmod{k}=\kappa(\prM)$ through which both residue functors $\rsd{\prMs}$ and $\rsd{\prM}$ factor and in which the corresponding primes are included. That intermediate category is obtained from the Frobenius quasi-abelian category~$\Aqab$ discussed in \Cref{Rem:quasi-abelian}. As we saw in \Cref{Rem:Aqab-Frobenius}, the projective-injectives in~$\Aqab$ are given by~$\add^{\otimes}(\E_0)$. Consequently the subcategory of perfect complexes in~$\Db(\Aqab)$ consists of the tt-ideal~$\ideal{\E_0}$ and the associated Verdier quotient is equivalent to the stable category $\stab(\Aqab)$
\[
\Sta\colon \Db(\Aqab)\onto\Db(\Aqab)/\ideal{\E_0}\cong\stab(\Aqab)
\]
as in \Cref{Prop:Rickard}. There is a commutative diagram of tt-functors
\begin{equation}
\label{eq:factorization-res-M01}%
\vcenter{\xymatrix{
& \Kb(\cA) \ar@{->>}[r]^-{\quo}
& \Db(\cA) \ar@{->>}[r]^-{\Sta}
& \stab(\cA) \kern-1em \ar@{}[r]|-{\cong}
& \kern-1em \mmod{k}
 \ar @{<-} `u[ll] `[lllld]_-{\rsd{\prMs}} [lllld]
\\
\DE \ar@{->>}[r]^-{\quo} \ar[ru]^-{\gr} \ar@{->>}[rd]^-{\tfgt}
& \Db(\Aqab) \ar@{->>}[r]^-{\Sta} \ar[ru]^-{\gr} \ar[rd]^-{\fgt}
& \stab(\Aqab) \ar[ru]^-{\gr} \ar[rd]^-{\fgt}
\\
& \Kb(\cA) \ar@{->>}[r]^-{\quo}
& \Db(\cA) \ar@{->>}[r]^-{\Sta}
& \stab(\cA) \kern-1em \ar@{}[r]|-{\cong}
& \kern-1em \mmod{k}
 \ar @{<-} `d[ll] `[llllu]^-{\rsd{\prM}} [llllu]
}}
\end{equation}
The commutativity of the top part is straightforward (Remarks~\ref{Rem:fgt-gr-exact} and~\ref{Rem:quasi-abelian}) and so is the bottom-right (slanted) square, in which $\fgt$ means everywhere `forget the filtration', \ie is the functor induced by~$\fgt\colon \Afil\to \cA$. (\Cref{Rem:fgt-gr-exact} again.) Commutativity of the bottom-left square in~\eqref{eq:factorization-res-M01} follows from \Cref{Cor:tfgt} and the fact that $\fgt\colon \DE\to\DA$ factors via $\quo\colon \DE\to \Db(\Aqab)$.

In order to deduce $\prMs\subset \prM$ from the factorizations of $\rsd{\prMs}$ and $\rsd{\prM}$ given in~\eqref{eq:factorization-res-M01}, it suffices to prove that in $\stab(\Aqab)$ we have $\Ker(\gr)\subset \Ker(\fgt)$. Now, $\stab(\Aqab)$ is a very simple category: Every object is a direct sum of~$\unit(n)$ and $\E_\ell(n)$ for $n\in\bbZ$ and $\ell\ge 1$ by \Cref{Prop:KRS-Efil} (we have not changed the underlying Krull-Schmidt category~$\Afil$ and $\ell=0$ can be removed since the $\E_0(n)$ are projective, hence zero in that stable category). Now, taking the total-graded of any non-zero object in this list $\{\unit(n),\E_\ell(n)\}$ remains non-zero in~$\stab(\cA)\cong\mmod{k}$. In other words, the prime $\Ker(\gr\colon \stab(\Aqab)\to \stab(\cA))$ is zero (\ie $\stab(\Aqab)$ is local). Hence we have the wanted inclusion of primes in~$\stab(\Aqab)$, namely $\Ker(\gr)=(0)\subset \Ker(\fgt)$, whatever the latter is. (It is $\ideal{\E_1}$ but this is not essential.)

Thus we have completely determined the space $\Spc(\DE)$ and we only need to provide generators for $\prLs$, $\prMs$ and~$\prNs$, \ie we need to show that the inclusions in~\eqref{eq:aux-LMN} are equalities. But now that we know the spectrum, it suffices to consider the supports of those tt-ideals. A direct verification shows that they coincide:
$\supp(\prLs)=\{\prN,\prNs\}=\supp(\ideal{\E_0})$, $\supp(\prMs)=\{\prL,\prLs,\prN,\prNs\}=\supp(\ideal{\E_0,\fund_0})$ and $\supp(\prNs)=\{\prL,\prLs\}=\supp(\ideal{\fund_0})$. Hence we do have equalities in~\eqref{eq:aux-LMN}.
\end{proof}

\begin{Rem}
\label{Rem:tfgt-res}%
The `twisted forgetful functor' $\tfgt\colon \DE\to\KA$ of \Cref{Def:tfgt} appears in the tt-residue functors $\rsd{\pL}$, $\rsd{\pM}$ and~$\rsd{\pN}$ of~\eqref{eq:M_0}--\eqref{eq:N_0}. However, it is only strictly necessary for~$\rsd{\pL}$. Indeed, by \Cref{Rem:res-fields}, both $\rsd{\cM}\colon \KA\to \kappa(\cM)$ and $\rsd{\cN}\colon \KA\to \kappa(\cN)$ factor via~$\quo\colon \KA\onto\DA$, hence using that $\quo\circ\tfgt=\fgt$ by \Cref{Cor:tfgt}, we get:
\begin{align*}
\rsd{\pM}=\Sta\circ\fgt\colon & \DE\to \DA\to \kappa(\pM)=\kappa(\cM)=\mmod{k}
\\
\rsd{\pN}=\res^{C_2}_1\circ\fgt\colon & \DE\to \DA\to \kappa(\pN)=\kappa(\cN)=\Db(k).
\end{align*}
This explains our comment about $\pL$ being the most elusive prime among the six. We return to $\rsd{\pL}$ in \Cref{Rem:tfgt-summary}.
\end{Rem}

\bigbreak\goodbreak
\section{Applications}
\label{sec:applications}%
\medbreak

Knowing $\Spc(\DE)$ by \Cref{Thm:Spc-DE}, we can describe all tt-ideals and consequently a number of localizations of~$\cK=\DE$. Direct inspection gives us the 14 closed subsets of~$\SpcK$  listed in~\eqref{eq:lattice-closed}. Note that they are all `Thomason', \ie their complement is quasi-compact, simply because~$\SpcK$ is finite. Applying \cite[Thm.\,4.10]{balmer:spectrum} gives the 14 tt-ideals of \Cref{Cor:14-tt-ideals}. (By rigidity of~$\cK$, every tt-ideal is $\otimes$-radical.) Let us now describe objects with the various possible supports. In the following pictures, we illustrate subsets $Y\subseteq\Spc(\DE)$ of
\[
\Spc(\DE)=
\vcenter{\xymatrix@R=.2em@C=1em{
\prLs\ar@{-}[d]\ar@{-}[rd] && \prNs \ar@{-}[ld]\ar@{-}[d]
\\
\prL \ar@{-}[rd] & \prMs \ar@{-}[d] & \prN \ar@{-}[ld]
\\
& \prM
}}
=
\vcenter{\xymatrix@R=.4em@C=1em{
\bullet \ar@{-}[d]\ar@{-}[rd] && \bullet \ar@{-}[ld]\ar@{-}[d]
\\
\bullet \ar@{-}[rd] & \bullet \ar@{-}[d] & \bullet\ar@{-}[ld]
\\
& \bullet
}}
\]
by writing $\bullet$ for the primes that do belong to~$Y$ and $\circ$ for those not in~$Y$.

\begin{Exas}
\label{Exa:supp-in-DE}%
As every prime~$\cP$ is the kernel of some $\rsd{\cP}\colon \DE\to \kappa(\cP)$ by~\eqref{eq:M_0}--\eqref{eq:N_1}, we compute $\supp(A)\overset{\textrm{def}}{=}\SET{\cP}{A\notin\cP}$ as~$\SET{\cP}{\rsd{\cP}(A)\neq 0}$.
\begin{enumerate}[\rm(a)]
\item
\label{it:supp-E_0}%
The object $\E_0=\pwz(kC_2)$ has $\gr(\E_0)=kC_2$ and $\tfgt(\E_0)=kC_2$, hence
\[
\supp(\E_0)=\{\prN,\prNs\}=\adhpt{\prN}=
\vcenter{\xymatrix@R=.2em@C=1em@H=.5em{
{\color{Gray}\circ} \ar@[Gray]@{-}[d]\ar@[Gray]@{-}[rd] && \bullet \ar@[Gray]@{-}[ld]\ar@{-}[d]
\\
{\color{Gray}\circ} \ar@[Gray]@{-}[rd] & {\color{Gray}\circ} \ar@[Gray]@{-}[d] & \bullet\ar@[Gray]@{-}[ld]
\\
& {\color{Gray}\circ}
}}
\]
\smallbreak
\item
\label{it:supp-E_1}%
The object $\E_1$ of~\eqref{eq:e_l} has $\gr(\E_1)=k\oplus k$ and $\tfgt(\E_1)=kC_2$, hence
\[
\supp(\E_1)=\{\prLs,\prMs,\prN,\prNs\}=\adhpt{\prMs}\cup\adhpt{\prN}=
\vcenter{\xymatrix@R=.2em@C=1em@H=.5em{
\bullet \ar@[Gray]@{-}[d]\ar@{-}[rd] && \bullet \ar@{-}[ld]\ar@{-}[d]
\\
{\color{Gray}\circ} \ar@[Gray]@{-}[rd] & \bullet \ar@[Gray]@{-}[d] & \bullet\ar@[Gray]@{-}[ld]
\\
& {\color{Gray}\circ}
}}
\]
\smallbreak
\item
For $\ell\ge 2$, the object $\E_\ell$ of~\eqref{eq:e_l} has $\gr(\E_\ell)=k\oplus k$ and $\tfgt(\E_\ell)=kC_2\oplus \fundpur^{\ell-1}$ as in \Cref{Rem:trwz-generators}. Hence
\[
\supp(\E_\ell)=\Spc(\DE)\smallsetminus\{\prM\}=
\vcenter{\xymatrix@R=.2em@C=1em@H=.5em{
\bullet \ar@{-}[d]\ar@{-}[rd] && \bullet \ar@{-}[ld]\ar@{-}[d]
\\
\bullet \ar@[Gray]@{-}[rd] & \bullet \ar@[Gray]@{-}[d] & \bullet\ar@[Gray]@{-}[ld]
\\
& {\color{Gray}\circ}
}}
\]
\smallbreak
\item
\label{it:supp-cone(beta)}%
The object~$\cone(\beta\colon \unit\to \unit(1))$ has $\gr(\cone(\beta))=(\cdots 0\to k\xto{0} k\to 0\cdots)=k[0]\oplus k[1]$ and $\tfgt(\cone(\beta))=\fundpur[-1]$ by \Cref{Rem:trwz-generators}. Therefore
\[
\supp(\cone(\beta))=\{\prL,\prLs,\prMs,\prNs\}=\adhpt{\prMs}\cup\adhpt{\prL}=
\vcenter{\xymatrix@R=.2em@C=1em@H=.5em{
\bullet \ar@{-}[d]\ar@{-}[rd] && \bullet \ar@{-}[ld]\ar@[Gray]@{-}[d]
\\
\bullet \ar@[Gray]@{-}[rd] & \bullet \ar@[Gray]@{-}[d] & {\color{Gray}\circ}\ar@[Gray]@{-}[ld]
\\
& {\color{Gray}\circ}
}}
\]
\smallbreak
\item
\label{it:supp(fund_0)}%
The object~$\fund_0=\pwz(\fundpur)$ has $\gr(\fund_0)=\fundpur$ and $\tfgt(\fund_0)=\fundpur$ hence
\[
\supp(\fund_0)=\{\prL,\prLs\}=\adhpt{\prL}=
\vcenter{\xymatrix@R=.2em@C=1em@H=.5em{
\bullet \ar@{-}[d]\ar@[Gray]@{-}[rd] && {\color{Gray}\circ} \ar@[Gray]@{-}[ld]\ar@[Gray]@{-}[d]
\\
\bullet \ar@[Gray]@{-}[rd] & {\color{Gray}\circ} \ar@[Gray]@{-}[d] & {\color{Gray}\circ}\ar@[Gray]@{-}[ld]
\\
& {\color{Gray}\circ}
}}
\]
\smallbreak
\item
The object~$\T=\cone(\beta)\otimes \E_1$ featured prominently in \Cref{sec:Spc-DE}. Its support is the complement of~$U=U(\cone(\omega))$, that is
\[
\supp(\cone(\beta)\otimes \E_1)=\supp(\cone(\omega))=\{\prLs,\prMs,\prNs\}=\adhpt{\prMs}=
\vcenter{\xymatrix@R=.2em@C=1em@H=.5em{
\bullet \ar@[Gray]@{-}[d]\ar@{-}[rd] && \bullet \ar@{-}[ld]\ar@[Gray]@{-}[d]
\\
{\color{Gray}\circ} \ar@[Gray]@{-}[rd] & \bullet \ar@[Gray]@{-}[d] & {\color{Gray}\circ}\ar@[Gray]@{-}[ld]
\\
& {\color{Gray}\circ}
}}\kern-2em
\]
\smallbreak
\item
\label{it:supp-singleton}%
One can combine the above to get the closed points, for instance
\[
\supp(\cone(\beta)\otimes \E_0)=
\vcenter{\xymatrix@R=.2em@C=1em@H=.5em{
{\color{Gray}\circ} \ar@[Gray]@{-}[d]\ar@[Gray]@{-}[rd] && \bullet \ar@[Gray]@{-}[ld]\ar@[Gray]@{-}[d]
\\
{\color{Gray}\circ} \ar@[Gray]@{-}[rd] & {\color{Gray}\circ} \ar@[Gray]@{-}[d] & {\color{Gray}\circ}\ar@[Gray]@{-}[ld]
\\
& {\color{Gray}\circ}
}}
\quadtext{and}
\supp(\fund_0\otimes \E_1)=
\vcenter{\xymatrix@R=.2em@C=1em@H=.5em{
\bullet \ar@[Gray]@{-}[d]\ar@[Gray]@{-}[rd] && {\color{Gray}\circ} \ar@[Gray]@{-}[ld]\ar@[Gray]@{-}[d]
\\
{\color{Gray}\circ} \ar@[Gray]@{-}[rd] & {\color{Gray}\circ} \ar@[Gray]@{-}[d] & {\color{Gray}\circ}\ar@[Gray]@{-}[ld]
\\
& {\color{Gray}\circ}
}}\kern-2em
\]
\smallbreak
\item
Direct sums of objects as in~\eqref{it:supp-E_0}, \eqref{it:supp(fund_0)} and~\eqref{it:supp-singleton} will provide representatives of the four remaining (disconnected) closed subsets listed in~\eqref{eq:lattice-closed}.
\end{enumerate}
\end{Exas}

We can then give new generators for the primes in~$U$, for instance.

\begin{Cor}
\label{Cor:primes-generators}%
We have $\prL=\ideal{\E_1}$, $\prN=\ideal{\cone(\beta)}$ and $\prM=\ideal{\E_2}$.
\end{Cor}
\begin{proof}
Check the supports of those tt-ideals; see \Cref{Thm:Spc-DE} and \Cref{Exa:supp-in-DE}.
\end{proof}

\begin{Not}
\label{Not:rho}%
Let $\rho\colon \unit\to \unit(1)[1]$ in~$\DE$ be the map associated to the fundamental exact sequence~\eqref{eq:fund_1}, namely
\begin{equation}
\label{eq:rho}%
\vcenter{\xymatrix@C=1.8em@R=1em{
\unit=
&\cdots 0 \ar[r]
& 0 \ar[r]
& \unit \ar[r]^-{}
& 0 \cdots
\\
\kern-.3em\tilde\unit=\ar@<.6em>[u]_-{\simeq} \ar@<-.6em>[d]^-{}
&\cdots 0 \ar[r]
& \unit(1) \ar[r]^-{\eta} \ar[u] \ar@{=}[d]
& \E_1 \ar[r]^-{} \ar[u]_-{\eps} \ar[d]^-{}
& 0 \cdots
\\
\unit(1)[1]= \ar@{<-} `l[u] `[uu]^-{\rho} [uu]
&\cdots 0 \ar[r]
& \unit(1) \ar[r]^-{}
& 0 \ar[r]
& 0 \cdots
}}
\end{equation}
Note right away that the cone of~$\rho$ is the cone of the lower map, that is,
\begin{equation}
\cone(\rho)=\E_1[1]\,.
\end{equation}
Its support is the 4-point subset $\{\pLs,\pMs,\pNs,\pN\}$ described in \Cref{Exa:supp-in-DE}\,\eqref{it:supp-E_1}.
\end{Not}

\begin{Rem}
\label{Rem:primes-generators-invertibles}%
Let us continue the `Koszul objects' thread of \Cref{Rem:Koszul-Kb(A)}. In addition to $\rho\colon \unit\to \unit(1)[1]$ from~\eqref{eq:rho}, whose cone is~$\E_1[1]$, we can use other invertibles in $\cK=\DEfil$, different from the `obvious' $\unit[1]$ and $\unit(1)$. By \Cref{Prop:L^n}, we have a third interesting invertible object, namely $\pwz(\invertpur)$. By \Cref{Rem:Koszul-Kb(A)}, the maps $\pwz(\tilde\eta\colon\unit\to \invertpur\potimes{-1})$ and $\pwz(\upsilon\colon\unit\to \invertpur\potimes{-1}[1])$ have respective cones $\fund_0[-1]$ and~$\E_0$. Let us write $\unit^{a,b,c}=\pwz(\invertpur\potimes{-c})(b)[a]$. With this notation, we may describe all the prime ideals of $\DEfil$ as generated by Koszul objects, as follows:
\[
\xymatrix@C=1em@R=1.5em{
\ideal{\cone(\pwz(\upsilon):\unit\to\unit^{1,0,1})}\kern-2em
&& \kern-2em\ideal{\cone(\pwz(\tilde\eta):\unit\to\unit^{0,0,1})}
\\
\ideal{\cone(\rho:\unit\to\unit^{1,1,0})} \ar@{-}[u]
& \ideal{\cone(\pwz(\tilde\eta \upsilon):\unit\to\unit^{1,0,2})} \ar@{-}[lu] \ar@{-}[ru]
& \ideal{\cone(\beta:\unit\to\unit^{0,1,0})} \ar@{-}[u]
\\
& \ideal{\cone(\beta\rho:\unit\to\unit^{1,2,0})} \ar@{-}[u] \ar@{-}[lu] \ar@{-}[ru]
}
\]
\end{Rem}

\begin{Rem}
For every object~$A\in\cK=\DE$, we know (\Cref{Rmd:Spc-quo}) that the open complement $U(A)=\SpcK\smallsetminus\supp(A)$ of its support is homeomorphic via $\Spc(\quo)$ to the spectrum of the Verdier quotient $\cK/\ideal{A}$. For instance, for the objects~$A$ whose supports are described in \Cref{Exa:supp-in-DE}, we obtain the spectra of several Verdier quotients of~$\DE$ by looking at the points marked~$\circ$. Let us isolate the following three special cases of interest, for which we can identify the corresponding localizations as something meaningful.
\[
{\xy
(0,0)*{\pLs};
(0,-10)*{\pL};
(20,-10)*{\pMs};
(20,-20)*{\pM};
(40,0)*{\pNs};
(40,-10)*{\pN};
{\ar@{-} (0,-8)*{};(0,-2)*{}};
{\ar@{-} (20,-18)*{};(20,-12)*{}};
{\ar@{-} (40,-8)*{};(40,-2)*{}};
{\ar@{-} (16,-8)*{};(4,-2)*{}};
{\ar@{-} (16,-18)*{};(4,-12)*{}};
{\ar@{-} (24,-8)*{};(36,-2)*{}};
{\ar@{-} (24,-18)*{};(36,-12)*{}};
{\ar@{..} (-4,-4)*{};(24,-18)*{};};
{\ar@{..} (-4,-4)*{};(-4,-12)*{};};
{\ar@{..} (-4,-12)*{};(24,-26)*{};};
{\ar@{..} (24,-18)*{};(24,-26)*{};};
{\ar@{..>}@/_1em/ (-1,-14)*{};(-20,-15)*{\Spc\big(\stabfil\big)};};
{\ar@{..} (44,-4)*{};(16,-18)*{};};
{\ar@{..} (44,-4)*{};(44,-12)*{};};
{\ar@{..} (44,-12)*{};(16,-26)*{};};
{\ar@{..} (16,-18)*{};(16,-26)*{};};
{\ar@{..>}@/^1em/ (41,-14)*{};(65,-15)*{\Spc\big(\Db(kC_2)\big)};};
{\ar@{--} (10,5)*{};(50,5)*{};};
{\ar@{--} (50,5)*{};(50,-28)*{};};
{\ar@{--} (50,-28)*{};(10,-28)*{};};
{\ar@{--} (10,-28)*{};(10,5)*{};};
{\ar@{-->}@/^1em/ (50,-5)*{};(65,0)*{\Spc\big(\Db(\Aqab)\big)};};
\endxy}
\]
\end{Rem}
\smallskip

\begin{Cor}[Inverting~$\beta$]
\label{Cor:inverting-beta}%
Recall the morphism~$\beta\colon \unit\to \unit(1)$ from~\Cref{Not:beta}. The (central) localization $\DE[\beta\inv]=\DE/\ideal{\cone(\beta)}$ is canonically equivalent to the derived category~$\Db(\mmod{kC_2})$ of the abelian category~$\cA$. In particular, its spectrum is the subset~$\{\pM,\pN\}$ of~$\Spc(\DE)$, with $\prN\in\adhpt{\prM}$.
\end{Cor}

\begin{proof}
The localization $\DE[\beta\inv]=\DE/\ideal{\cone(\beta)}$ has spectrum the open complement~$\{\pM,\pN\}$ of the closed subset $\supp(\cone(\beta))=\{\pL,\pLs,\pMs,\pNs\}$ of \Cref{Exa:supp-in-DE}\,\eqref{it:supp-cone(beta)}. The latter contains~$\{\pLs,\pMs,\pNs\}=\SpcK\smallsetminus U$, hence our localization is a localization of~$\DEU$ from~\eqref{eq:DEU}. We proved in \Cref{Thm:DEU} that $\DEU\cong\KA$. Our localization $\DE[\beta\inv]$ is therefore the localization of~$\KA$ away from the remaining point~$\{\pL\}$, corresponding to~$\{\cL\}=\supp(\Kbac)$ in~$\KA$. This localization is nothing but~$\DA$. (See \Cref{Rem:localisations-of-KA}.)
\end{proof}

\begin{Rem}
As in \Cref{Cor:tfgt}, the localization functor $\DE\onto \Db(\cA)$ isolated above is simply the one induced by the exact forgetful functor $\fgt\colon \Efil\to\cA$.
\end{Rem}

For the next case, recall that $\Efil$ is Frobenius (\Cref{Cor:Afil-Frobenius}).

\begin{Cor}[Inverting~$\rho$]
\label{Cor:inverting-rho}%
Recall the morphism~$\rho\colon \unit\to \unit(1)[1]$ from~\Cref{Not:rho}. The (central) localization $\DE[\rho\inv]=\DE/\ideal{\cone(\rho)}$ is canonically equivalent to the stable category~$\stabfil$ of the Frobenius exact category~$\Efil$. In particular, its spectrum is the subset~$\{\pM,\pL\}$ of~$\Spc(\DE)$, with $\prL\in\adhpt{\prM}$.
\end{Cor}

\begin{proof}
By \Cref{Prop:Rickard}, the stable category~$\stabfil$ can also be obtained as the Verdier quotient of the derived category of~$\Efil$ by the tt-ideal~$\ideal{\E_1}$ generated by the projectives of~$\Efil$ (see \Cref{Prop:proj-Efil}). So it suffices to apply \Cref{Rmd:Spc-quo} to the object~$\E_1$, whose support was computed in \Cref{Exa:supp-in-DE}\,\eqref{it:supp-E_1}.
\end{proof}

\begin{Rem}
\label{Rem:inverting-rho}%
Since $\supp(\cone(\rho))=\supp(\E_1)\supset \SpcK\smallsetminus U$, the localization $\DE\onto \DE[\rho\inv]$ that we just identified to be~$\Sta\colon \DE\onto \stab(\Efil)$ is a localization of~$\DEU\cong \KA$. It is easy to trace the kernel $\KA\onto \stab(\Efil)$ as having support $\supp(\E_1)\cap U=\{\pN\}$, corresponding to~$\{\cN\}=\supp(kC_2)$ in~$\Spc(\KA)$. In other words, we have an equivalence $\KA/\ideal{kC_2}\cong \stab(\Efil)$ making the following diagram commute\,:
\[
\xymatrix{
\DE \ar@{->>}[d]_-{\Sta} \ar@{->>}[r]^-{\tfgt}
& \KA \ar@{->>}[d]^-{\quo}
\\
\stab(\Efil) \ar@{<-}[r]^-{\cong}
& \KA/\ideal{kC_2}
}
\]
\end{Rem}

\begin{Rem}
\label{Rem:tfgt-summary}%
Summarizing our analysis of the tt-functor~$\tfgt\colon \DE\to \KA$ of \Cref{Def:tfgt}, we saw that if we post-compose it with the two localizations $\KA\onto \DA$ and $\KA\onto \KA/\ideal{kC_2}$ discussed in \Cref{Rem:localisations-of-KA} we obtain respectively $\fgt\colon \DE\to\DA$ by \Cref{Cor:tfgt} and $\Sta\colon \DE\onto \stab(\Efil)$ by \Cref{Rem:inverting-rho}. In terms of the residue tt-functor~$\rsd{\pL}\colon \DE\to \kappa(\pL)$ of~\eqref{eq:L_0}, we can complement \Cref{Rem:tfgt-res} and obtain the factorization
\[
\rsd{\pL}\colon \DE\ \xonto{\Sta}\ \stab(\Efil)\cong \KA/\ideal{kC_2}\xto{\rsd{\cL}'}\kappa(\pL)=\kappa(\cL)=\Db(k)
\]
using the functor $\rsd{\cL}'$ of \Cref{Rem:res-fields}, induced by $\Kb(\sta)\colon\KA\to \Kb(k)$.
\end{Rem}

For the last localization, recall the quasi-abelian structure~$\Aqab$ from \Cref{Rem:quasi-abelian}.
\begin{Cor}
The derived category $\Db(\Aqab)$ of $\Afil$ with its maximal (quasi-abelian) structure is a localization of~$\DE$ with kernel the tt-ideal~$\ideal{\fund_0}$. Hence
\[
\Spc(\Db(\Aqab))=
\vcenter{\xymatrix@R=.2em@C=1em{
& \prNs \ar@{-}[ld]\ar@{-}[d]
\\
 \prMs \ar@{-}[d] & \prN \ar@{-}[ld]
\\
 \prM
}}
\quad\overset{\Spc(\quo)}\hook\quad
\vcenter{\xymatrix@R=.2em@C=1em{
\prLs\ar@{-}[d]\ar@{-}[rd] && \prNs \ar@{-}[ld]\ar@{-}[d]
\\
\prL \ar@{-}[rd] & \prMs \ar@{-}[d] & \prN \ar@{-}[ld]
\\
& \prM
}}
=\Spc(\DE)
\]
\end{Cor}

\begin{proof}
The localization $\DE\onto \Db(\Aqab)$ is clear since both categories are localizations of~$\Kb(\Afil)$ and there are less acyclics for~$\Efil$ than for the maximal exact structure~$\Aqab$. Hence the kernel of $\DE\onto \Db(\Aqab)$ consists of complexes which are acyclic for~$\Aqab$, like~$\fund_0$ certainly is, see~\eqref{eq:fund_0}. If the support of this kernel was larger than $\supp(\fund_0)=\adhpt{\prL}$, it would contain the closed point~$\prNs$. So to get the result it suffices to show that $\prNs$ belongs to~$\Spc(\Db(\Aqab))$, \ie that $\rsd{\prNs}\colon\DE\to\Db(k)$ factors via~$\Db(\Aqab)$. This is easy to see\,:
\[
\xymatrix@R=.7em{
\DE \ar[rr]^-{\rsd{\prNs}} \ar[rd]_-{\gr} \ar@{->>}[dd]_-{\quo}
&& \Db(k)
\\
& \Kb(\cA) \ar[ru]_-{\res^{C_2}_1} \ar@{->>}[rd]_-{\quo}
\\
\Db(\Aqab) \ar[rr]_-{\exists\,\gr}
&& \Db(\cA) \ar[uu]_-{\exists\,\res^{C_2}_1}
\\
}
\]
The top part of the diagram commutes by definition of~$\rsd{\prNs}$, see~\eqref{eq:N_1}. The rest is straightforward (the non-trivial part is the existence of the functors).
\end{proof}

\begin{Rem}
Since $\Aqab$ is itself Frobenius (\Cref{Rem:Aqab-Frobenius}), one can go one step further and identify $\Spc(\stab(\Aqab))$ as~$\{\prM,\prMs\}$. Indeed, $\stab(\Aqab)$ is the quotient of~$\Db(\Aqab)$ by its tt-ideal of perfect complexes (\Cref{Prop:Rickard}) which we already know is~$\ideal{\E_0}$ (\Cref{Rem:Aqab-Frobenius}), and $\supp(\E_0)=\{\prN,\prNs\}$. That localization $\DE\onto\stab(\Aqab)$ already appeared in the proof of \Cref{Thm:Spc-DE}, see~\eqref{eq:factorization-res-M01}.
\end{Rem}

\bigbreak\goodbreak
\part{Artin-Tate motives}
\label{part:II}%
\medbreak

\bigbreak\goodbreak
\section{Artin-Tate motives and filtered representations}
\label{sec:translation}%
\medbreak

We now turn to algebraic geometry. The first goal, which is the subject of the present section, is to provide the dictionary necessary to translate the results obtained in \Cref{part:I} to the theory of motives. This dictionary is due to Positselski~\cite{positselski:artin-tate-motives}, and relies on Voevodsky's resolution of the Milnor conjecture.
\begin{Not}
  \label{Not:voevodsky-motives}
Let~$\FF$ be a field of characteristic zero, and set $k=\bbZ/2$. (Of course, many of the constructions below apply in greater generality.) Recall that Voevodsky constructed in~\cite{Voevodsky00} a category $\DM(\FF;k)$ of (geometric, mixed) motives over~$\FF$ with coefficients in~$k$. It is defined by starting with the homotopy category of bounded complexes of finite $k$-linear correspondences of smooth schemes of finite type over~$F$, localizing it to force homotopy invariance and Mayer-Vietoris, then idempotent completing it, and finally inverting the Tate object~$k(1)$. The resulting $\DM(\FF;k)$ is an essentially small, rigid tt-category. For a smooth~$\FF$-scheme of finite type we denote its motive in $\DM(\FF;k)$ by~$\mot(X)$. (If $X=\Spec(A)$ is affine we write $\mot(A)$ instead.) The motive of the base is the unit $\unit=\mot(\FF)$ for the tensor product. By definition of the Tate object we have $\mot(\bbP^1)=\unit\oplus\unit(1)[2]$. The notation $M(i)$ is short for $M\otimes \unit(i)$.

Of particular interest to us will be three thick triangulated subcategories of $\DM(\FF;k)$ with the following sets of generators:
  \begin{itemize}
  \item Artin motives $\DAM(\FF;k)=\thick\{\mot(E)\mid E/\FF\text{ finite}\}$;
  \item Tate motives $\DTM(\FF;k)=\thick\{\unit(n)\mid n\in\bbZ\}$;
  \item Artin-Tate motives $\DATM(\FF;k)=\thick\{\mot(E)(n)\,|\,E/\FF\text{ finite}, n\in\bbZ\}$.
  \end{itemize}
  All these subcategories are in fact rigid tt-categories, by~\cite[Thm.\,4.3.2]{Voevodsky00}.
\end{Not}

Now, let us fix a real closed field~$\FF$ with algebraic closure $\bar{\FF}=\FF(\sqrt{-1})$. As in the first part, we denote by $\kGfil$ the category of filtered $kC_2$-modules with the exact structure of \Cref{sec:filt-Frobenius}.
\begin{Prop}[{\cite{positselski:artin-tate-motives}}]%
  \label{Prop:positselski}
  There is an equivalence of triangulated categories
  \begin{equation}
    \label{eq:positselski}%
    \pos:\Db(\kGfil)\xrightarrow{\sim}\DATM(\FF;k)
  \end{equation}
  which induces a homeomorphism on spectra:
  \begin{align*}
    \Spc(\DATM(\FF;k))&\xrightarrow{\sim}\Spc(\Db(\kGfil))\\
    \cP&\mapsto \pos^{-1}(\cP)\,.
  \end{align*}
\end{Prop}

Let us explain this result. The \'{e}tale realization functor induces an equivalence of exact $\otimes$-categories
\begin{equation}
\label{eq:mini-positselski}%
  \cF(\FF;k)\xrightarrow{\sim}\kGfil
\end{equation}
where $\cF(\FF;k)\subset\DATM(\FF;k)$ denotes the smallest full subcategory containing $\mot(\FF)(n)$ and $\mot(\bar{\FF})(n)$ for all $n\in\bbZ$, and closed under extensions. The filtration is induced by the weight filtration on $\DATM(\FF;k)$. For the proof of this we refer to~\cite[\S\,3]{positselski:artin-tate-motives} or~\cite[\S\,7]{gallauer:tt-dtm-algclosed}. This part is where the Milnor conjecture is used. By the argument in~\cite[Proposition~7.7]{gallauer:tt-dtm-algclosed} (or~\cite[Appendix~D]{positselski:artin-tate-motives}), the functor $\kGfil\isofrom \cF(\FF;k)\hook \DATM(\FF;k)$ of~\eqref{eq:mini-positselski} extends to an exact functor\,(\footnote{As the underlying functor of an exact morphism of stable derivators, \eqref{eq:exact-positselski} is in fact unique up to unique isomorphism. See~\cite[Proposition~7.7]{gallauer:tt-dtm-algclosed}.\label{fnt:uniqueness-derivators}})
\begin{equation}
\label{eq:exact-positselski}%
  \pos\colon\DE\to\DATM(\FF;k)
\end{equation}
which is not known to be tensor except on the heart. The latter, however, is enough to deduce the second statement of \Cref{Prop:positselski} from the first, because tt-ideals on both sides of~\eqref{eq:positselski} are thick subcategories closed under tensoring with certain objects in the heart. We again refer to~\cite[\S\,7]{gallauer:tt-dtm-algclosed} for details.

It remains to explain the equivalence~\eqref{eq:positselski}. Instead of invoking Koszulity of the cohomology algebra as in~\cite{positselski:artin-tate-motives} we will give a direct proof of this fact, using some of the results from \Cref{part:I}.

Consider the invertible objects $\unit[1]$ and $\unit(1)$ in $\DE$ and in $\DATM(\FF;k)$, respectively. They give rise to bigraded endomorphism rings, denoted $R^{\sbull,\sbull}$ and $\Hm^{\sbull,\sbull}=\Hm^{\sbull,\sbull}(\FF;k)$, which is motivic cohomology, defined for all~$n,m\in \bbZ$ by
\begin{equation}
\label{eq:bigraded-rings}%
  R^{n,m}=\Hom_{\DE}(\unit,\unit(m)[n]) \quadtext{and} \Hm^{n,m}=\Hom_{\DATM}(\unit,\unit(m)[n]).
\end{equation}
Without shifts, that is, for $n=0$, the equivalence~\eqref{eq:mini-positselski} gives us $R^{0,m}\cong \Hm^{0,m}$.

\begin{Lem}
  \label{Lem:pos-ring-morphism}
  The exact functor $\pos$ of~\eqref{eq:exact-positselski} induces a morphism of bigraded rings
  \begin{equation*}
    \pos:R^{n,m}\to\Hm^{n,m}.
  \end{equation*}
  Moreover, the latter is a bijection in bidegrees with $n\leq 1$.
\end{Lem}
\begin{proof}
  Since $\pos$ in~\eqref{eq:positselski} is an exact functor, the map $\pos:R^{\sbull,\sbull}\to\Hm^{\sbull,\sbull}$ is a morphism of bigraded abelian groups. For multiplication, recall that given two homogeneous elements $f:\unit\to\unit(m)[n]$ and $g:\unit\to\unit(m')[n']$ their product can be viewed as the image of the pair $(g,f)$ under the map
  \begin{multline}
    \label{eq:multiplication-nontensor}%
    \Hom(\unit,\unit(m')[n'])\times\Hom(\unit,\unit(m)[n])\xto{(\unit(m)[n]\otimes-)\times\id}\\
    \Hom(\unit(m)[n],\unit(m)[n]\otimes\unit(m')[n'])\times\Hom(\unit,\unit(m)[n])\xto{\simeq}\\
    \Hom(\unit(m)[n],\unit(m+m')[n+n'])\times\Hom(\unit,\unit(m)[n])\xto{\circ}\\
    \Hom(\unit,\unit(m+m')[n+n']).
  \end{multline}
  But the functor
  \begin{equation*}
    \unit(m)[n]\otimes\pos(-)\simeq(\unit(m)\otimes\pos(-))[n]
  \end{equation*}
  is equivalent to
  \begin{equation*}
    \pos(\unit(m)\otimes -)[n]\simeq\pos(-(m))[n]
  \end{equation*}
  since $\unit(m)$ belongs to the heart~$\Efil\cong\cF(\FF;k)$ and the latter is a \emph{tensor} subcategory of $\DATM(\FF;k)$. (This is the universal property of the bounded derived category~\cite[Theorem~2.17]{porta:stable-derivators-universal} alluded to in \Cref{fnt:uniqueness-derivators}; see the proof of~\cite[Proposition~7.7]{gallauer:tt-dtm-algclosed} for details.) This shows that the map on hom groups induced by the functor $\pos$ is compatible with the first two arrows in~\eqref{eq:multiplication-nontensor}. Compatibility with the last arrow is functoriality of $\pos$. This completes the proof that $\pos:R^{\sbull,\sbull}\to\Hm^{\sbull,\sbull}$ is a bigraded ring morphism.

   The groups $\Hom_{\mathcal{C}}(A,B[n])$ vanish for $n<0$ and all $A,B\in\Efil\cong\cF(\FF;k)$, for both $\mathcal{C}=\DE$ and~$\mathcal{C}=\DATM(\FF;k)$. Since $\pos:\DE\to\DATM(\FF;k)$ is an equivalence on $\Efil$ and its image~$\pos(\Efil)=\cF(\FF;k)$ is closed under extensions in $\DATM(\FF;k)$, the last part of the statement follows, by~\cite{dyer:exact-triangulated}.
\end{proof}

\begin{Lem}
  \label{Lem:pos-reduction}
  The following statements are equivalent:
  \begin{enumerate}[\rm(a)]
  \item \label{it:pos-red-a}%
  The functor $\pos:\DE\to\DATM(\FF;k)$ is an equivalence.
  \item \label{it:pos-red-b}%
  The morphism $\pos:R^{\sbull,\sbull}\to\Hm^{\sbull,\sbull}$ is an isomorphism.
  \end{enumerate}
\end{Lem}
\begin{proof}
  Let $\mathcal{D}=\DATM(\FF;k)$. It is generated as a thick subcategory by $\mot(E)(m)$ which are in the image of the functor $\pos$. Indeed, $\mot(\FF)(m)=\pos(\unit(m))$ while $\mot(\bar{\FF})(m)=\pos(\E_0(m))$. Given that $\DE$ is idempotent complete, to prove~\eqref{it:pos-red-a} it therefore suffices to prove that~$\pos$ is fully faithful.

  Fix two complexes $A,B\in\DE$. We want to prove that
  \begin{equation}
    \label{eq:pos-hom-groups}
    \pos:\Hom_{\DE}(A,B)\to\Hom_\mathcal{D}(\pos(A),\pos(B))
  \end{equation}
  is a bijection. By induction on the length of the complexes and the five-lemma (recall \Cref{Rem:truncation}) we reduce to $A$ and $B$ shifts of objects in $\Efil$. It therefore suffices to prove~\eqref{eq:pos-hom-groups} is a bijection for $A\in\Efil$ and $B=C[n]$ with $C\in\Efil$ and $n\in\bbZ$.

  By induction on the filtration amplitude and the five-lemma we reduce to $A$ and $C$ of the form $\unit(m)$ or $\E_0(m)$, some $m$. Since twisting is an equivalence, we may assume $A$ is pure of weight zero.

  As already remarked above, the groups $\Hom_{\mathcal{C}}(A,C[n])$ vanish for $n<0$ and $\mathcal{C}\in\{\DE,\mathcal{D}\}$, and~\eqref{eq:pos-hom-groups} is a bijection for $n=0$. If $C=\E_0(m)$ is projective-injective, then $\Hom_{\DE}(A,\E_0(m)[n])=0$ for $n>0$. The same is true for the hom-groups in $\mathcal{D}$:
  \begin{equation*}
    \Hom_{\mathcal{D}}(\unit,\mot(\bar{\FF})(m)[n])=\Hm^{n,m}(\bar{\FF};k)=0.
  \end{equation*}
  We may therefore assume $C=\unit(m)$. Similarly, we may assume $A=\unit$ and are reduced to~\eqref{it:pos-red-b}. This proves the Lemma since~\eqref{it:pos-red-a}$\Rightarrow$\eqref{it:pos-red-b} is trivial.
\end{proof}

\begin{Rem}
  \label{Rem:motivic-cohomology-point}
   Recall that, by the Beilinson-Lichtenbaum conjecture with $\bbZ/2$-coefficients~\cite{voevodsky:milnor-conjecture},
   \begin{equation*}
     \Hm^{\sbull,\sbull}(\FF;k)\cong k[\beta,\rho]
   \end{equation*}
   where:
   \begin{itemize}
   \item $\beta\colon \unit\to \unit(1)$ is the (motivic) Bott element of~\cite{levine:bott,haesemeyer-hornbostel:bott}\,(\footnote{\,This element corresponds to $\tau$ in~\cite{voevodsky:motivic-power-operations}.}), that is, the non-trivial element $-1$ in $\Hm^{0,1}(\FF;k)\cong \mu_2(\FF)=\{\pm1\}$;
   \item the map $\rho\colon \unit\to \unit(1)[1]$ is the non-trivial element in $\Hm^{1,1}(\FF;k)\cong K^{\mathrm{M}}_1(\FF)/2=\FF^\times/(\FF^\times)^2$, induced by a morphism $\Spec(\FF)\to\bbG_{\mathrm{m},\FF}$ corresponding to a negative element of $\FF$.
\end{itemize}

   It follows from the fact that $\Hm^{\sbull,\sbull}$ is generated by elements in $\Hm^{\leq 1,\sbull}$ and from \Cref{Lem:pos-ring-morphism} that $\pos:R^{n,m}\to\Hm^{n,m}$ is an epimorphism. Thus we only need to establish injectivity for $n\geq 2$. This will follow from the following result.
 \end{Rem}

 \begin{Prop}
   \label{Prop:extensions-computation} For~$R^{n,m}$ as in~\eqref{eq:bigraded-rings}, we have for all $n\geq 0$ and $m\in \bbZ$
   \begin{equation*}
     R^{n,m}=
     \begin{cases}
       k&\textrm{ if }n\leq m\\
       0&\textrm{ if }n>m.
     \end{cases}
   \end{equation*}
 \end{Prop}
 \begin{proof}
   Recall from \Cref{Rem:rwz} the injective resolution $\injres$ of $\unit$ in $\Efil$. We may compute $R^{n,m}$ as $R^{n,m}=\Hm_{-n}(\Hom_{\Efil}(\unit,\injres(m)))$ which is the homology in degree~$-n$ of the complex
\begin{equation*}
  0\to k\xto{0}k\xto{0}\cdots\xto{0}k\to 0
\end{equation*}
with non-zero objects in homological degrees from~$0$ down to~$-m$. (If $m<0$ then this is the zero complex.) The claim follows immediately.
 \end{proof}

We may now finish the proof of \Cref{Prop:positselski}. By \Cref{Lem:pos-reduction} and \Cref{Rem:motivic-cohomology-point}, we are reduced to prove injectivity of $R^{n,m}\onto\Hm^{n,m}$ in degrees $n\geq 2$. By \Cref{Prop:extensions-computation}, the $k$-dimensions of $R^{n,m}$ and $\Hm^{n,m}$ coincide for all $n,m$, and we conclude.

\bigbreak\goodbreak
\section{Spectrum of mod-2 real Artin-Tate motives}
\label{sec:Spc-DATM(R;Z/2)}%
\medbreak

Having established the necessary dictionary in the previous section, we are now in a position to apply the results of \Cref{part:I} to the theory of motives. The following theorem, our main result, follows directly from \Cref{Prop:positselski} and \Cref{Thm:Spc-DE}.
\begin{Thm}
\label{Thm:main}%
Let~$\FF$ be a real closed field. The spectrum of the tt-category $\DATM(\FF;\bbZ/2)$ is the following space:
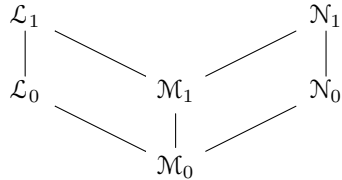
\[
{\xy
(0,0)*{\pLs};
(0,-10)*{\pL};
(20,-10)*{\pMs};
(20,-20)*{\pM};
(40,0)*{\pNs};
(40,-10)*{\pN};
{\ar@{-} (0,-8)*{};(0,-2)*{}};
{\ar@{-} (20,-17.5)*{};(20,-13)*{}};
{\ar@{-} (40,-8)*{};(40,-2)*{}};
{\ar@{-} (16,-8)*{};(4,-2)*{}};
{\ar@{-} (16,-18)*{};(4,-12)*{}};
{\ar@{-} (24,-8)*{};(36,-2)*{}};
{\ar@{-} (24,-18)*{};(36,-12)*{}};
\endxy}
\]
As before, a line indicates that the lower prime specializes to the higher prime.
\qed
\end{Thm}

Our next goal is to interpret motivically some of the functors used in \Cref{part:I} to catch the prime ideals.

\begin{Rem}
  \label{Rem:semisimplification}%
  Objects of the conjectural abelian category of mixed motives should possess a weight filtration whose $n$th graded piece belongs to the subcategory of pure motives of weight $n$. It is also expected that the pure motives span the subcategory of semi-simple objects, so that the total weight graded functor can be thought of as a \emph{semi-simplification}. It is then natural to view the functor $\gr$ which in \Cref{part:I} was used to detect the top three points (\cf \eqref{eq:subdivision}) as a triangulated analogue of semi-simplification, detecting the `pure' primes.
\end{Rem}
  Independently of these considerations, there is another functor defined on Voevodsky motives in great generality, and which, on $\DATM(\FF;\bbZ/2)$ for a real closed field $\FF$, catches the same three primes. To discuss this functor, we need to fix some notation regarding Chow motives.
\begin{Not}
  \label{Not:chow-motives}
 We denote by $\Chow(\FF;k)$ the classical category of Chow motives over~$\FF$ with coefficients in~$k$. The Tate motive is denoted by $\unit\{1\}$, and as before we denote by $\mot(X)$ the Chow motive of a smooth projective~$\FF$-scheme~$X$ and abbreviate $\mot(X)\{m\}=\mot(X)\otimes\unit\{m\}$. In particular, we have for such $X,Y$
  \begin{equation}
    \label{eq:chow-hom}
    \Hom_{\Chow(\FF;k)}(\mot(X)\{m\},\mot(Y)\{n\})=\mathrm{CH}^{\dim(X)+n-m}(X\times_{\FF} Y;k).
  \end{equation}
  As with mixed motives we consider the following three subcategories of~$\Chow(\FF;k)$:
  \begin{itemize}
  \item Artin Chow motives $\AM(\FF;k)=\add(\mot(E)\mid E/\FF\text{ finite})$;
  \item Tate Chow motives $\TM(\FF;k)=\add(\unit\{n\}\mid n\in\bbZ)$;
  \item Artin-Tate Chow motives $\ATM(\FF;k)=\add(\mot(E)\{n\}\,|\,E/\FF\text{ finite}, n\in\bbZ)$.
  \end{itemize}
  These are rigid idempotent-complete $\otimes$-categories, and embed as full $\otimes$-exact subcategories (endowed with the split exact structure) of their mixed triangulated analogues defined in \Cref{Not:voevodsky-motives}, by~\cite[\S\,2.2]{Voevodsky00}.
\end{Not}

\begin{Rem}
  \label{Rem:identification-pure-motives}%
From now on, let us fix a real closed field~$\FF$ with algebraic closure $\bar{\FF}=\FF(\sqrt{-1})$ and the coefficients~$k=\bbZ/2$ as in \Cref{Not:voevodsky-motives}. Using~\eqref{eq:chow-hom}, one checks easily that \'etale cohomology induces canonical equivalences of $\otimes$-categories
  \[
    \AM(\FF;k)\simeq \mmod{kC_2},\quad \TM(\FF;k)\simeq \grmod{k},\quad \ATM(\FF;k)\simeq \grmod{kC_2}
  \]
  where $\grmod{R}$ denotes the category of (finitely generated, as always) $\bbZ$-graded $R$-modules. In particular, these are in fact abelian categories and coincide with the categories of pure motives of Artin, Tate, and Artin-Tate type, respectively. (In other words, for a zero-dimensional variety, all adequate equivalence relations on algebraic cycles coincide.) Under those equivalences, $\mot(\FF)$ corresponds to~$k$ and $\mot(\bar\FF)$ corresponds to~$kC_2$.
\end{Rem}

\begin{Rem}
\label{Rem:chow-weight}%
The category of Voevodsky motives $\DM(\FF;k)$ admits a weight structure in the sense of~\cite{bondarko-weight}, called the \emph{Chow weight structure}, whose (additive) heart is $(\DM(\FF;k))^{w-\heartsuit}=\Chow(\FF;k)$ the category of Chow motives. The associated conservative weight complex functor constructed by Bondarko,
  \begin{equation}
    \label{eq:weight-complex}
    w_{\Chow}:\DM(\FF;k)\to\Kb(\Chow(\FF;k)),
  \end{equation}
  is a tt-functor (\cite[Lemma~20]{bachmann:invertible-quadric} or~\cite{Aoki:weight-complex-tensor}). Restricted to $\DATM(\FF;k)$, this functor factors through the homotopy category of Artin-Tate Chow motives and yields a tt-functor by summing over all weights
  \begin{equation*}
    \gr^{w_{\Chow}}:\DATM(\FF;k)\xto{w_{\Chow}}\Kb(\ATM(\FF;k))\xto{\oplus}\Kb(\AM(\FF;k)).
  \end{equation*}
  Note that the category $\Kb(\AM(\FF;k))$ is canonically equivalent to the triangulated category of Artin motives $\DAM(\FF;k)$ (\cite[Prop.~3.4.1]{Voevodsky00}) and $\gr^{w_{\Chow}}$ therefore provides a retraction to the inclusion of Artin motives into $\DATM(\FF;k)$.

  This proves that the map on spectra,
  \begin{equation*}
    \Spc(\gr^{w_{\Chow}}):\Spc(\DAM(\FF;k))\to\Spc(\DATM(\FF;k)),
  \end{equation*}
  is injective. (Recall from \Cref{Thm:Spc(Kb(kC_2))} that the domain is the V-shaped topological space.) Also, since $\gr^{w_{\Chow}}$ is conservative, the map on spectra catches all closed points, by~\cite[Theorem~1.2]{balmer:surjectivity}. Finally, the objects corresponding to the generators of the bottom three primes in \Cref{part:I}, $\cone(\beta)$ and $\cone(\rho)$, are sent to sums of invertibles by $\gr^{w_{\Chow}}$, and we conclude that $\gr^{w_{\Chow}}$ catches the same points as the total graded of \Cref{part:I}, depicted at the top of the following diagram (the other functors will be discussed subsequently).
\begin{equation}
  \label{eq:3-functors}%
\vcenter{\xy
(0,0)*{\pLs};
(0,-10)*{\pL};
(20,-10)*{\pMs};
(20,-20)*{\pM};
(40,0)*{\pNs};
(40,-10)*{\pN};
{\ar@{-} (0,-8)*{};(0,-2)*{}};
{\ar@{-} (20,-18)*{};(20,-12)*{}};
{\ar@{-} (40,-8)*{};(40,-2)*{}};
{\ar@{-} (16,-8)*{};(4,-2)*{}};
{\ar@{-} (16,-18)*{};(4,-12)*{}};
{\ar@{-} (24,-8)*{};(36,-2)*{}};
{\ar@{-} (24,-18)*{};(36,-12)*{}};
{\ar@{..} (-4,-4)*{};(24,-18)*{};};
{\ar@{..} (-4,-4)*{};(-4,-12)*{};};
{\ar@{..} (-4,-12)*{};(24,-26)*{};};
{\ar@{..} (24,-18)*{};(24,-26)*{};};
{\ar@{..>}@/_1em/ (-1,-14)*{};(-20,-15)*{\mathrm{Re}_{\bbR}};};
{\ar@{..} (44,-4)*{};(16,-18)*{};};
{\ar@{..} (44,-4)*{};(44,-12)*{};};
{\ar@{..} (44,-12)*{};(16,-26)*{};};
{\ar@{..} (16,-18)*{};(16,-26)*{};};
{\ar@{..>}@/^1em/ (41,-14)*{};(60,-15)*{\mathrm{Re}_{\mathrm{\acute{e}t}}};};
{\ar@{..} (-4,3)*{};(4,3)*{};};
{\ar@{..} (4,3)*{};(20,-5)*{};};
{\ar@{..} (20,-5)*{};(36,3)*{};};
{\ar@{..} (36,3)*{};(44,3)*{};};
{\ar@{..} (44,3)*{};(44,-3)*{};};
{\ar@{..} (44,-3)*{};(20,-15)*{};};
{\ar@{..} (20,-15)*{};(-4,-3)*{};};
{\ar@{..} (-4,-3)*{};(-4,3)*{};};
{\ar@{..>}@/_.5em/ (22,-4)*{};(20,5)*{w_{\Chow}};};
\endxy}
\end{equation}
Although $\gr^{w_{\Chow}}$ and the `semi-simplification' functor $\gr$ are both retractions to the inclusion $\DAM(\FF;k)\into\DATM(\FF;k)$ and catch the same points, they are distinct. They differ in that for example $\gr^{w_{\Chow}}(\unit\{i\}=\unit(i)[2i])=\unit$ while $\gr(\unit(i)[2i])=\unit[2i]$.
\end{Rem}

\begin{Rem}
\label{Rem:etale-realization}
 On the category of Voevodsky motives there is an \'{e}tale realization functor~\cite[Thm.\,4.3, Rem.\,4.8]{ivorra:l-adic-real-1}
  \begin{equation*}
    \DM(\FF;k)\to\Db(kC_2)
  \end{equation*}
and its restriction to Artin-Tate motives is the functor $\mathrm{Re}_{\mathrm{\acute{e}t}}$ of~\eqref{eq:3-functors}
  \begin{equation}
    \label{eq:etale-realization-DATM}
    \mathrm{Re}_{\mathrm{\acute{e}t}}:\DATM(\FF;k)\to\Db(kC_2).
  \end{equation}
  The functor in~\eqref{eq:etale-realization-DATM} is obtained by inverting the motivic Bott element $\beta:\unit\to\unit(1)$ of \Cref{Rem:motivic-cohomology-point}, \ie we have the following statement.
\end{Rem}

\begin{Prop}
  \label{Prop:etale-vs-forget} The following diagram of exact functors commutes,
  \begin{equation*}
    \xymatrix{\DE\ar[r]^-{\pos}_-{\sim}\ar[rd]_{\fgt}&\DATM(\FF;k)\ar[d]^{\mathrm{Re}_{\mathrm{\acute{e}t}}}\\
      &\Db(kC_2)}
  \end{equation*}
  and the \'etale realization induces a canonical equivalence of tt-categories
  \[
    \DATM(\FF;k)[\beta\inv]\simeq\Db(kC_2).
  \]
\end{Prop}
\begin{proof}
  By construction of the equivalence $\pos$, the diagram commutes when restricted to the heart $\Efil$. The first claim then follows from the fact that both functors $\DE\to\Db(kC_2)$ are the underlying functors of an exact morphism of stable derivators (\cf \Cref{fnt:uniqueness-derivators} for $\pos$, \cite[Theorem~4.5.2]{cisinski-deglise:etale-motives} for the \'etale realization, and~\cite{cisinski:categories-derivables} for $\fgt$) and are therefore uniquely determined by their restrictions to the heart~\cite[Theorem~2.17]{porta:stable-derivators-universal}. The second claim then follows from \Cref{Cor:inverting-beta}.
\end{proof}

\begin{Rem}
  \label{Rem:real-realization}%
The functor $\mathrm{Re}_{\bbR}:\DATM(\FF;k)\to\DATM(\FF;k)[\rho^{-1}]$ of~\eqref{eq:3-functors}, that we call \emph{real realization} as in~\cite{bachmann:real}, is obtained by inverting the morphism $\rho:\unit\to\unit(1)[1]$ of \Cref{Rem:motivic-cohomology-point}. This localization was studied in \loccit in the stable $\bbA^{1}$-homotopy category, where the quotient category was identified with the topological stable homotopy category:
  \begin{equation*}
    \SHmot(\FF)[\rho^{-1}]\simeq\SH.
  \end{equation*}
  In our context, the localization was described in \Cref{Cor:inverting-rho}:
  \begin{equation*}
    \DATM(\FF;k)[\rho^{-1}]\simeq\stab(\kGfil)
  \end{equation*}
  as the stable category of the Frobenius exact category of filtered $kC_2$-modules. (The $\rho$ in the motivic setting corresponds to the $\rho$ in the setting of filtered $kC_2$-modules, as defined in \Cref{Not:rho}.) By \Cref{Cor:inverting-rho}, another description of the same tt-category is
  \[
    \DATM(\FF;k)[\rho^{-1}]\simeq\DAM(\FF;k)/\ideal{\mot(\bar{\FF})}\simeq\Kb(\mmod{kC_2})/\ideal{kC_2}.
  \]
\end{Rem}

\begin{Rem}
  \label{Rem:motivic-tfgt}
  It is more mysterious (to us, at least) how to interpret motivically the important central localization with respect to the generalized Koszul complex $\cone(\rho)\otimes\cone(\beta)$ considered in \Cref{sec:central}:
  \begin{equation}
  \label{eq:central-localization-motivic}%
   \tfgt:\DATM(\FF;k)\onto\frac{\DATM(\FF;k)}{\ideal{\cone(\rho)\otimes\cone(\beta)}}\simeq \DAM(\FF;k)
  \end{equation}
  which catches the three bottom (mixed) primes in~\eqref{eq:3-functors}. \textsl{Cf}. \Cref{Rem:tfgt}.
\end{Rem}

\begin{Rem}
  \label{Rem:geometric-base-change}
  Base-change to the algebraic closure induces a tt-functor
  \begin{equation*}
    -\times_{\FF}\bar{\FF}:\DATM(\FF;k)\to\DTM(\bar{\FF};k),
  \end{equation*}
  and a similar argument as in \Cref{Prop:etale-vs-forget} shows that this functor corresponds, in the context of \Cref{part:I}, to forgetting the $C_2$-action:
  \begin{equation*}
    \res^{C_2}_1:\DEfil\to\Db(\fil{\mmod{k}}).
  \end{equation*}
  The spectrum of $\DTM(\bar{\FF};k)$ was determined in~\cite{gallauer:tt-fmod,gallauer:tt-dtm-algclosed} by computing the spectrum of $\Db(\fil{\mmod{k}})$. The functor $-\times_{\FF}\bar{\FF}$ catches the two right-most points in~\eqref{eq:3-functors}, \ie the `geometric' primes of \eqref{eq:subdivision}.
\end{Rem}

\begin{Rem}
\label{Rem:prime-generators-motivic}%
We can deduce from \Cref{sec:applications} generators for each of the six prime ideals in $\DATM(\FF;k)$:
\[
\xymatrix@C=1em@R=1em{
\ideal{\mot(\bar{\FF})}
&& \ideal{\fund_0}
\\
\ideal{\cone(\rho)} \ar@{-}[u]
& \ideal{\mot(\bar{\FF}),\fund_0} \ar@{-}[lu] \ar@{-}[ru]
& \ideal{\cone(\beta)} \ar@{-}[u]
\\
& \ideal{\cone(\beta\rho)} \ar@{-}[u] \ar@{-}[lu] \ar@{-}[ru]
}
\]
Here, $\fund_0$ is the complex of finite correspondences
\begin{equation*}
  \cdots 0 \to 0 \to \Spec(\FF)\xto{\eta}\Spec(\bar{\FF})\xto{\epsilon}\Spec(\FF)\to 0 \to 0 \cdots
\end{equation*}
viewed as an object in $\DATM(\FF;k)$, while the morphisms $\beta:\unit\to\unit(1)$ and $\rho:\unit\to\unit(1)[1]$ are those of \Cref{Rem:motivic-cohomology-point}.
\end{Rem}

\bigbreak\goodbreak
\section{Spectrum of integral real Artin-Tate motives}
\label{sec:Spc-DATM(R;Z)}%
\medbreak
As mentioned in the introduction, for $\FF$ real-closed, the computation of the spectrum $\Spc(\DATM(\FF;\bbZ))$ essentially breaks down into two distinct tasks: First, computing the spectrum for mod-2 coefficients as done in the previous section, and second, computing the spectrum for $\bbZ[1/2]$-coefficients after passing to the algebraic closure $\bar{\FF}$. The latter task was undertaken in~\cite{gallauer:tt-dtm-algclosed}. The goal of this section is to put the solutions to these two tasks together in order to describe, in \Cref{Thm:Spc-DATM(R;Z)}, the space $\Spc(\DATM(\FF;\bbZ))$. We fully achieve this for $\FF$ small enough and provide a conjectural picture for all real closed~$\FF$; \cf \Cref{Rmd:Spc-DATM(R;Q)-conjecture}.

\begin{Prop}
\label{Prop:img-Spc}%
Let $\FF$ be a base field (not necessarily real closed).
\begin{enumerate}[\rm(a)]
\item
\label{it:img-Spc-Z/l}%
Let $\mathcal{D}(\FF;R)$ denote $\DATM(\FF;R)$, or $\DAM(\FF;R)$ or $\DTM(\FF;R)$. Let $\ell$ be a prime and consider the change-of-coefficients functor
\[
\chgcoef_{\ell}^*\colon\mathcal{D}(\FF;\bbZ)\to\mathcal{D}(\FF;\bbZ/\ell).
\]
Then the image of $\Spc(\chgcoef_{\ell}^*)$ is the support of~$\bbZ/\ell=\cone (\unit \xto{\ell}\unit)$ in~$\mathcal{D}(\FF;\bbZ)$.
\smallbreak
\item
\label{it:img-Spc-Fbar}%
Let $\mathcal{D}(\FF;R)$ denote $\DATM(\FF;R)$ or $\DAM(\FF;R)$. Let $E/\FF$ be a finite separable extension and $p\colon \Spec(E)\to \Spec(\FF)$. Consider base-extension
\[
p^*=(E\times_{\FF}-)\colon\mathcal{D}(\FF;R)\to\mathcal{D}(E;R).
\]
Then the image of $\Spc(p^*)$ is the support of~$\mot(E)$ in~$\mathcal{D}(\FF;R)$.
\end{enumerate}
\end{Prop}
\begin{proof}
Both parts follow from~\cite[Thm.\,1.7]{balmer:surjectivity}, that guarantees that the image of the map on spectra induced by a tt-functor~$F\colon \cK\to \cK'$ with a right adjoint~$G$ is the subset $\supp(G(\unit))$, as long as $\cK$ is rigid. For~\eqref{it:img-Spc-Z/l} we use $(\chgcoef_\ell)_*(\unit)\cong\bbZ/\ell$. For~\eqref{it:img-Spc-Fbar} we use $p_*(\unit)\cong\mot(E)$. (The latter does not exist on non-Artin Tate motives.)
\end{proof}

\begin{Cor}
  \label{Cor:DATM-odd}%
  Let~$\FF$ be a real closed field, and let $\ell$ be an odd prime. The change-of-coefficients functor $\chgcoef_{\ell}^*:\DATM(\FF;\bbZ)\to\DATM(\FF;\bbZ/\ell)$ induces on spectra a homeomorphism from the Sierpi\'nski space
  \begin{equation*}
    \Spc(\DATM(\FF;\bbZ/\ell))\quad =\quad \vcenter{\xymatrix@R=1em{\bullet\ar@{-}[d]\\\bullet}}
  \end{equation*}
  onto the support of $\bbZ/\ell=\cone (\unit \xto{\ell}\unit)$.
\end{Cor}
\begin{proof}
Consider the maps induced on spectra by~$p^*$ and~$\chgcoef_{\ell}^*$ as in \Cref{Prop:img-Spc}:
\[
\kern-.3em\xymatrix@C=2.9em@R=.1em{\Spc(\DTM(\bar\FF;\bbZ/\ell)) \ar[r]^-{\Spc(p^*)}
& \Spc(\DATM(\FF;\bbZ/\ell)) \ar[r]^-{\Spc(\chgcoef_{\ell}^*)}
& \Spc(\DATM(\FF;\bbZ)).
}
\]
As the Euler characteristic of $\mot(\bar\FF)$ is 2 and thus invertible in $\bbZ/\ell$ we see that $\unit\in\ideal{\mot(\bar{\FF})}$ in~$\DATM(\FF;\bbZ/\ell)$. Therefore the first map $\Spc(p^*)$ is surjective by \Cref{Prop:img-Spc}\,\eqref{it:img-Spc-Fbar}. The image of the second map (and therefore the image of the composite) is~$\supp(\bbZ/\ell)$ by \Cref{Prop:img-Spc}\,\eqref{it:img-Spc-Z/l}.

The source, $\Spc(\DTM(\bar{\FF};\bbZ/\ell))$, was computed in~\cite[Cor.~8.3]{gallauer:tt-dtm-algclosed} and the prime ideals were found to be 0 and $\ideal{\cone(\beta_\ell)}$ where $\beta_{\ell}:\bbZ/\ell\to\bbZ/\ell(1)$ corresponds to the choice of a primitive $\ell$th root of unity in~$\bar\FF$. To prove the result, it therefore suffices to show that the above composite is injective (forcing the first to be a homeomorphism), \ie that those two points have distinct images.

Note that the primitive $\ell$th root defining $\beta_\ell\colon \bbZ/\ell\to \bbZ/\ell(1)$ already exists in~$\FF$ as $\ell$ is odd and~$\FF$ is real closed. Let $A\in\DATM(\FF;\bbZ)$ be the object $\chgcoef_{\ell,*}(\cone(\beta_\ell))$. By the computation in~\cite[Lem.~B.6]{gallauer:tt-dtm-algclosed}, we see that the image of $A$ under the functor~$p^*\chgcoef^*_{\ell}$ generates the prime ideal $\mathfrak{e}_\ell=\ideal{\cone(\beta_\ell)}$. \Cref{Prop:Spc-map-generators}\,\eqref{it:Spc-map-generators-a} implies that $\Spc(p^*\chgcoef_{\ell}^*)$ is injective as was left to prove.
\end{proof}

\begin{Thm}
  \label{Thm:Spc-DATM(R;Z)} Let $\FF=\bbR_{\mathrm{alg}}=\overline{\bbQ}\cap \bbR$ be the field of real algebraic numbers. Then the spectrum of $\DATM(\bbR_{\mathrm{alg}};\bbZ)$ is the following set with specialization relations
  \begin{equation}
    \label{eq:Spc-DATM(R;Z)}
\vcenter
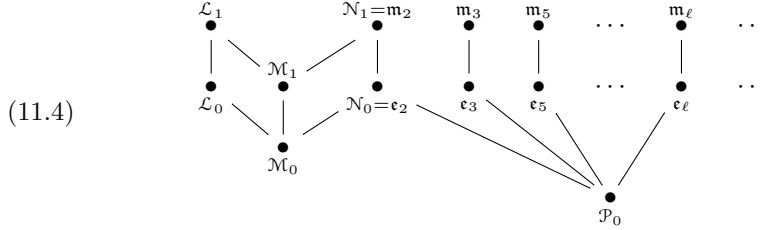
{\xymatrix@C=1em@R=.01em{
\overset{\pLs}{\bullet}&& \overset{\pNs=\mathfrak{m}_2}{\bullet}&\overset{\mathfrak{m}_3}{\bullet}&\overset{\mathfrak{m}_5}{\bullet}&\cdots&\overset{\mathfrak{m}_\ell}{\bullet}& \cdots
\\
\underset{\pL}{\bullet}\ar@{-}[u]&\overset{\pMs}{\bullet}\ar@{-}[ru]\ar@{-}[lu]&\underset{\pN=\mathfrak{e}_2}{\bullet}\ar@{-}[u]&\underset{\mathfrak{e}_3}{\bullet}\ar@{-}[u]&\underset{\mathfrak{e}_5}{\bullet}\ar@{-}[u]&\cdots&\underset{\mathfrak{e}_\ell}{\bullet}\ar@{-}[u]
& \cdots
\\
&\underset{\pM}{\bullet}\ar@{-}[ru]\ar@{-}[lu]\ar@{-}[u]
\\
&&&&&\underset{\cP_0}{\bullet}\ar@{-}[llluu]\ar@{-}[lluu]\ar@{-}[luu]\ar@{-}[ruu]
}}
  \end{equation}
where $\ell$ runs through all prime numbers. The topology is the minimal one with these specialization relations: The closed subsets are
  \begin{enumerate}[\rm(a)]
  \item finite specialization-closed subsets; and
  \item subsets of the form $Z\cup \adhpt{\cP_0}$, where $Z$ is as in~{\rm(a)}.
  \end{enumerate}
\end{Thm}

\begin{proof}
If a prime~$\cP\in\Spc(\DATM(\FF;\bbZ))$ contains~$\cone(2)$, the object $\mot(\bar\FF)$ generates~$\DATM(\FF;\bbZ)/\cP$, since $\mot(\bar\FF)$ has Euler characteristic~$2$. Hence
\begin{equation}
\label{eq:two-closed}%
\Spc(\DATM(\FF;\bbZ))=\supp(\bbZ/2)\cup \supp(\mot(\bar\FF))\,.
\end{equation}
(These two closed subsets are not disjoint.) By \Cref{Prop:img-Spc}, we know that these two closed pieces are the images of the following two maps, respectively
\begin{align}
\label{eq:supp(Z/2)}%
\Spc(\DATM(\FF;\bbZ/2))\xto{\Spc(\chgcoef_2^*)}\Spc(\DATM(\FF;\bbZ)),
\\
\label{eq:supp(Fbar)}%
\Spc(\DATM(\bar\FF;\bbZ))\xto{\Spc(p^*)}\Spc(\DATM(\FF;\bbZ)).
\end{align}
The source of~\eqref{eq:supp(Z/2)} was computed in \Cref{Thm:main} and consists of the six points $\{\pL,\pLs,\pM,\pMs,\pN,\pNs\}$ with the inclusions appearing in~\eqref{eq:Spc-DATM(R;Z)}. On the other hand, the source of~\eqref{eq:supp(Fbar)} was computed in \cite[Thm.\,8.6]{gallauer:tt-dtm-algclosed} and consists of all the primes $\mathfrak{m}_\ell$ and $\mathfrak{e}_\ell$ and the generic~$\cP_0$ exactly as depicted in~\eqref{eq:Spc-DATM(R;Z)}, without the 4 left-most points. (Here we use that~$\bar\FF$ is absolutely algebraic and therefore satisfies Hypothesis~6.6 of \loccit) We need to show that none of those inclusions become equalities in~$\DATM(\FF;\bbZ)$ and that there are no other inclusions or collisions except $\pN=\mathfrak{e}_2$ and $\pNs=\mathfrak{m}_2$. The intersection of the two closed pieces in~\eqref{eq:two-closed} is the image of $\supp(\mot(\bar\FF))$ in $\Spc(\DATM(\FF;\bbZ/2))$, or equivalently the image of $\supp(\bbZ/2)$ in~$\Spc(\DTM(\bar\FF;\bbZ))$, both of which are a Sierpi\'nski space~$\{\pN,\pNs\}=\{\mathfrak{e}_2,\mathfrak{m}_2\}$. This reduces the question of proving injectivity on each of the two maps~\eqref{eq:supp(Z/2)} and~\eqref{eq:supp(Fbar)}.

To show that \eqref{eq:supp(Z/2)} is injective, we use \Cref{Prop:Spc-map-generators}\,\eqref{it:Spc-map-generators-a} and the explicit generators of the prime ideals in $\DATM(\FF;\bbZ/2)$ given in \Cref{Rem:prime-generators-motivic}, most of which already live in $\DATM(\FF;\bbZ)$. The only exceptions are $\pN=\ideal{\cone(\beta)}$ and $\pM=\ideal{\cone(\beta\rho)}=\ideal{\cone(\rho),\cone(\beta)}$. But in these cases, \cite[Lemma~B.6]{gallauer:tt-dtm-algclosed} gives $\chgcoef_2^*\chgcoef_{2,*}\cone(\beta)=\cone(\beta)\oplus\cone(\beta)[1]$ and we can indeed conclude with \Cref{Prop:Spc-map-generators}\,\eqref{it:Spc-map-generators-a}.

To show that \eqref{eq:supp(Fbar)} is injective, we can use \Cref{Cor:DATM-odd}, that shows the images of $\mathfrak{m}_\ell\subset \mathfrak{e}_\ell$ remain proper inclusions in~$\DATM(\FF;\bbZ)$. In the same vein, since $\cP_0\notin\supp(\bbZ/\ell)$ for any prime~$\ell$, the inclusions~$\mathfrak{e}_\ell\subset\cP_0$ remain proper.

At this point we have verified the accuracy of~\eqref{eq:Spc-DATM(R;Z)}, that is the underlying set together with the specialization relations of $\Spc(\DATM(\FF;\bbZ))$. Given a closed subset $Z\subset\Spc(\DATM(\FF;\bbZ))$, let $Z_1=Z\cap \supp(\bbZ/2)$ and $Z_2=Z\cap\supp(\mot(\bar{\FF}))$. Necessarily, $Z_1$ is a finite specialization-closed subset. Since the map $\Spc(p^*)$ is continuous, the (bijective) pullback of $Z_2$ in $\Spc(\DTM(\bar{\FF};\bbZ))$ is closed. It follows from~\cite[Thm.\,8.6]{gallauer:tt-dtm-algclosed} that it is either finite specialization-closed or the whole space. This concludes the proof.
\end{proof}

\begin{Rem}
\label{Rmd:Spc-DATM(R;Q)-conjecture}
  Let~$\FF$ be a general real closed field, and consider the canonical comparison map of~\cite[\S\,5]{balmer:sss}
  \begin{equation*}
    \varrho_{\DATM(\FF;\bbZ)}:\Spc(\DATM(\FF;\bbZ))\to\Spec(\bbZ).
  \end{equation*}
  (The `$\varrho_{\cK}$' notation is that of~\cite{balmer:sss} and is unrelated to the motivic~$\rho$.) The proof of \Cref{Thm:Spc-DATM(R;Z)} in fact identifies all the fibers of this map except at the generic point. (They look precisely as described in~\eqref{eq:Spc-DATM(R;Z)} for $\FF=\bbR_{\mathrm{alg}}$, with our `six points' above~$2\bbZ$.) For the fiber above the generic point of~$\Spec(\bbZ)$, we do not know whether it is a single point for every~$\FF$, as in the case of~$\FF=\bbR_{\mathrm{alg}}$. The issue is that the vanishing hypothesis on the algebraic K-theory of~$\FF$ in~\cite[Hyp.~6.6]{gallauer:tt-dtm-algclosed} is possibly violated. We are therefore not able to determine the spectrum of $\DATM(\FF;\bbQ)$ for general real closed~$\FF$. We conjecture that it is a singleton space in general, and therefore that the spectrum of $\DATM(\FF;\bbZ)$ looks exactly as in the case of $\FF=\bbR_{\mathrm{alg}}$ described in \Cref{Thm:Spc-DATM(R;Z)}.

What we \emph{do} know for general real closed fields $\FF$ is that the $\ell$-adic realization
\begin{equation*}
  \mathrm{Re}_{\ell}:\DATM(\FF;\bbQ)\to\Db(\bbQ_\ell)
\end{equation*}
is conservative, by~\cite[Theorem~1.12]{wildeshaus:interior-motive-pic-surfaces}, and that in particular $\DATM(\FF;\bbQ)$ is a local category, \ie $0$ is a prime ideal. In order to prove the remaining specialization relations depicted in \Cref{Cor:Spc-AT-integral} one can run the argument of~\cite[Theorem~6.10, Lemma~8.5]{gallauer:tt-dtm-algclosed}.
\end{Rem}
\bigbreak\goodbreak
\section{Applications}
\label{sec:Spc-DAM-DTM}%
\medbreak

We will now easily derive the spectra of Artin motives and Tate motives, over a real-closed base field~$\FF$ and with integral coefficients. The inclusions into Artin-Tate motives induce surjective maps on spectra by~\cite[Thm.\,1.3]{balmer:surjectivity}. As in the previous section, nothing new compared to the case of~$\bar\FF$ happens with 2 inverted, because then $\mot(\bar\FF)$ is $\otimes$-faithful. So we concentrate on the coefficients~$k=\bbZ/2$ (but see \Cref{Rem:Spc-DTM/DAM-Z}).

\subsection*{Tate motives}

For the category of Tate motives $\DTM(\FF;k)$, see~\Cref{Not:voevodsky-motives}.
\begin{Thm}
  \label{Thm:Spc-DTM}%
  The spectrum of $\DTM(\FF;k)$ is the following set with specialization relations:
  \begin{equation}
    \label{eq:Spc-DTM}
    \vcenter{\xymatrix@C=1em@R=.1em{&0\\
      \langle \cone(\rho)\rangle\ar@{-}[ru]&&\langle\cone(\beta)\rangle\ar@{-}[ul]\\
      &\langle\cone(\beta\rho)\rangle\ar@{-}[ru]\ar@{-}[lu]}}
  \end{equation}
\end{Thm}
For the map induced by $\DTM(\FF;k)\hook\DATM(\FF;k)$, see \Cref{Rmd:Spc-inclusion-Artin-Tate}.
\begin{proof}
  We know from~\cite[Cor.~1.8]{balmer:surjectivity} that the map $\varphi\colon \Spc(\DATM(\FF;k))\to \Spc(\DTM(\FF;k))$ on spectra, given by intersection with $\DTM(\FF;k)$, is surjective. To determine the map more specifically note that the composition of the horizontal arrows in
  \begin{equation*}
    \xymatrix@C=4em@R=1.5em{\DTM(\FF;k)\ \ar@{.>}[rrd]_{\gr^{w_{\mathrm{Chow}}}_{\mathrm{Tate}}}\ar@{^(->}[r]&\DATM(\FF;k)\ar[r]^-{\gr^{w_{\mathrm{Chow}}}}&\Kb(kC_2)\\
      &&\Kb(k)\ar[u]_-{\triv}}
  \end{equation*}
  factors as indicated. As the three ``pure'' primes $\pLs$, $\pMs$, $\pNs$ are all detected by the functor $\gr^{w_{\mathrm{Chow}}}$ and $\Kb(k)=\Db(k)$ has a single prime ideal, namely $0$, it follows that these three primes are all mapped to the same prime ideal in $\DTM(\FF;k)$, namely $\ker(\gr^{w_{\mathrm{Chow}}}_{\mathrm{Tate}})$. As the functor $\gr^{w_{\mathrm{Chow}}}$ is conservative (\Cref{Rem:chow-weight}), so is $\gr^{w_{\mathrm{Chow}}}_{\mathrm{Tate}}$, and we see that $\varphi$ maps these three primes to~$0$.

The generators of the remaining prime ideals in $\DATM(\FF;k)$, namely $\pL$, $\pM$ and~$\pN$, \cf \Cref{Rem:prime-generators-motivic}, all lie in $\DTM(\FF;k)$ hence they are mapped to distinct primes under $\varphi$ (which are also distinct from $0$) with the same generators, by \Cref{Prop:Spc-map-generators}. For the same reason there can be no additional specialization relations in~\eqref{eq:Spc-DTM}. This completes the proof.
\end{proof}

\subsection*{Artin motives}

For the category of Artin motives $\DAM(\FF;k)$, see \Cref{Not:voevodsky-motives}.

\begin{Thm}
  \label{Thm:Spc-DAM}%
  The spectrum of $\DAM(\FF;k)$ is the following set with specialization relations (with $\fund_0$ as in \Cref{Rem:prime-generators-motivic}):
  \begin{equation}
    \label{eq:Spc-DAM}
    \vcenter{\xymatrix@C=1em@R=1em{\ideal{\mot(\bar{\FF})}&&\ideal{\fund_0}\\
        &\ideal{\mot(\bar{\FF}),\fund_0}\ar@{-}[ru]\ar@{-}[lu]}}
  \end{equation}
\end{Thm}
For the map induced by $\DAM(\FF;k)\hook\DATM(\FF;k)$, see \Cref{Rmd:Spc-inclusion-Artin-Tate}.
\begin{proof}
  As recalled in \Cref{Not:chow-motives}, the category $\DAM(\FF;k)$ is equivalent to $\Kb(kC_2)$ whose spectrum, by \Cref{Thm:Spc(Kb(kC_2))}, indeed has the required shape and description of the primes as in~\eqref{eq:Spc-DAM}.
\end{proof}

\begin{Rem}
\label{Rmd:Spc-inclusion-Artin-Tate}%
  The inclusions of tt-subcategories
\[
\DAM(\FF;k)\hook\DATM(\FF;k)\qquadtext{and}\DTM(\FF;k)\hook\DATM(\FF;k)
\]
induce the following two canonical ``projection'' maps on spectra:
\[
\scalebox{.8}{{\xy
(-25,0)*{\boxed{\Spc(\DATM(\FF;k))}};
(-25,-35)*{\boxed{\Spc(\DAM(\FF;k))}};
(90,-25)*{\boxed{\Spc(\DTM(\FF;k))}};
(0,0)*{\bullet};
(0,-10)*{\bullet};
(15,-5)*{\bullet};
(15,-15)*{\bullet};
(30,0)*{\bullet};
(30,-10)*{\bullet};
{\ar@{-} (0,-8)*{};(0,-2)*{}};
{\ar@{-} (15,-13)*{};(15,-7)*{}};
{\ar@{-} (30,-8)*{};(30,-2)*{}};
{\ar@{-} (12,-4)*{};(3,-1)*{}};
{\ar@{-} (12,-14)*{};(3,-11)*{}};
{\ar@{-} (18,-4)*{};(27,-1)*{}};
{\ar@{-} (18,-14)*{};(27,-11)*{}};
(90,-5)*{\bullet};
(75,-10)*{\bullet};
(90,-15)*{\bullet};
(105,-10)*{\bullet};
{\ar@{-} (78,-9)*{};(87,-6)*{}};
{\ar@{-} (93,-6)*{};(102,-9)*{}};
{\ar@{-} (87,-14)*{};(78,-11)*{}};
{\ar@{-} (93,-14)*{};(102,-11)*{}};
(0,-35)*{\bullet};
(15,-40)*{\bullet};
(30,-35)*{\bullet};
{\ar@{-} (12,-39)*{};(3,-36)*{}};
{\ar@{-} (18,-39)*{};(27,-36)*{}};
(40,-2)*{\Big\}\mapsto};
(50,-10)*{\longrightarrow};
(15,-25)*{\Big\downarrow};
\endxy}}
\]
as can be readily verified on the generators of the relevant prime ideals.
\end{Rem}

\begin{Rem}
If one removes the unique closed point of the space $\Spc(\DTM(\FF;k))$, what remains is precisely $\Spc(\DAM(\FF;k))$:
\[
\scalebox{0.8}{{\xy
(0,-10)*{\bullet};
(20,0)*{\bullet};
(20,-20)*{\bullet};
(40,-10)*{\bullet};
{\ar@{-} (16,-2)*{};(4,-8)*{}};
{\ar@{-} (16,-18)*{};(4,-12)*{}};
{\ar@{-} (24,-2)*{};(36,-8)*{}};
{\ar@{-} (24,-18)*{};(36,-12)*{}};
(-28,-2)*{\Spc\big(\DTM(\FF;k)\big)};
 {\ar@{..>}@/^1em/ (41,-14)*{};(65,-15)*{\Spc\big(\DAM(\FF;k)\big)};};
{\ar@{--} (-5,-7)*{};(45,-7)*{};};
{\ar@{--} (45,-7)*{};(45,-23)*{};};
{\ar@{--} (45,-23)*{};(-5,-23)*{};};
{\ar@{--} (-5,-23)*{};(-5,-7)*{};};
\endxy}}
\]
This geometric observation underlies a categorical fact. The composite
\[
\DTM(\FF;k)\into \DATM(\FF;k)\xto{\tfgt}\DAM(\FF;k)
\]
realizes the inclusion of the bottom three points and therefore factors through the localization (which realizes the removal of the closed point)
\begin{equation}
\label{eq:tate-artin-localization}
\frac{\DTM(\FF;k)}{\ideal{\cone(\rho)\otimes\cone(\beta)}}\to\DAM(\FF;k).
\end{equation}
It turns out that \eqref{eq:tate-artin-localization} is an equivalence of tt-categories.
To see this it suffices, by \Cref{Rem:motivic-tfgt}, to show that
\[
\frac{\DTM(\FF;k)}{\ideal{\cone(\rho)\otimes\cone(\beta)}}\to\frac{\DATM(\FF;k)}{\ideal{\cone(\rho)\otimes\cone(\beta)}}
\] is an equivalence.
This can be checked on the two open subsets $U(\cone(\beta))$ and $U(\cone(\rho))$ (cf.~\cite{balmer-favi:gluing}) but since these localizations are \emph{central} localizations~\cite[\S\,3]{balmer:sss} this is clear.
In particular, we note that in this particular instance, Tate motives carry all the information about Artin motives.
\end{Rem}

\begin{Rem}
\label{Rem:Spc-DTM/DAM-Z}%
With integral coefficients, the spectra of the tt-categories $\DTM(\FF;\bbZ)$ and $\DAM(\FF;\bbZ)$ are respectively
\[
\vcenter{\xymatrix@C.5em@R=.5em{
& {\bullet}&&{\bullet}&{\bullet}&\cdots&{\bullet}& \cdots
\\
{\bullet}\ar@{-}[ru]\ar@{-}[rd]&&{\bullet}\ar@{-}[lu]&{\bullet}\ar@{-}[u]&{\bullet}\ar@{-}[u]&\cdots&{\bullet}\ar@{-}[u]
& \cdots
\\
&{\bullet}\ar@{-}[ru]
\\
&&&&&{\bullet}\ar@{-}[llluu]\ar@{-}[lluu]\ar@{-}[luu]\ar@{-}[ruu]
}}
\qquadtext{and}
\vcenter{\xymatrix@C.5em@R=.5em{
{\bullet}\ar@{-}[rd]&&{\bullet}&{\bullet}&{\bullet}&\cdots&{\bullet}
& \cdots
\\
&{\bullet}\ar@{-}[ru]
\\
&&&&&{\bullet}\ar@{-}[llluu]\ar@{-}[lluu]\ar@{-}[luu]\ar@{-}[ruu]
}}
\]
Details are left to the reader, following the methods of \Cref{sec:Spc-DATM(R;Z)}.
\end{Rem}

\bookmarksetup{startatroot} 


\begin{thebibliography}{DRSS99}

\bibitem[Aok20]{Aoki:weight-complex-tensor}
Ko~Aoki.
\newblock The weight complex functor is symmetric monoidal.
\newblock {\em Adv. Math.}, 368:107145, 10, 2020.

\bibitem[Bac17]{bachmann:invertible-quadric}
Tom Bachmann.
\newblock On the invertibility of motives of affine quadrics.
\newblock {\em Doc. Math.}, 22:363--395, 2017.

\bibitem[Bac18]{bachmann:real}
Tom Bachmann.
\newblock Motivic and real \'{e}tale stable homotopy theory.
\newblock {\em Compos. Math.}, 154(5):883--917, 2018.

\bibitem[Bal05]{balmer:spectrum}
Paul Balmer.
\newblock The spectrum of prime ideals in tensor triangulated categories.
\newblock {\em J. Reine Angew. Math.}, 588:149--168, 2005.

\bibitem[Bal10a]{balmer:sss}
Paul Balmer.
\newblock Spectra, spectra, spectra - tensor triangular spectra versus
  {Z}ariski spectra of endomorphism rings.
\newblock {\em Algebraic and Geometric Topology}, 10(3):1521--63, 2010.

\bibitem[Bal10b]{balmer:icm}
Paul Balmer.
\newblock Tensor triangular geometry.
\newblock In {\em Proc.\ of the {I}nternational {C}ongress of {M}athematicians.
  {V}olume {II}}, pages 85--112. Hindustan Book Agency, New Delhi, 2010.

\bibitem[Bal18]{balmer:surjectivity}
Paul Balmer.
\newblock On the surjectivity of the map of spectra associated to a
  tensor-triangulated functor.
\newblock {\em Bull. Lond. Math. Soc.}, 50(3):487--495, 2018.

\bibitem[Bal20]{balmer:guide-HT-handbook}
Paul Balmer.
\newblock A guide to tensor-triangular classification.
\newblock In {\em Handbook of homotopy theory}, pages 145--162. CRC
  Press/Chapman Hall, Boca Raton, FL, 2020.

\bibitem[BF07]{balmer-favi:gluing}
Paul Balmer and Giordano Favi.
\newblock Gluing techniques in triangular geometry.
\newblock {\em Q. J. Math.}, 58(4):415--441, 2007.

\bibitem[BKS19]{balmer-krause-stevenson:ruminations}
Paul Balmer, Henning Krause, and Greg Stevenson.
\newblock Tensor-triangular fields: ruminations.
\newblock {\em Selecta Math. (N.S.)}, 25(1):25:13, 2019.

\bibitem[Bon10]{bondarko-weight}
M.~V. Bondarko.
\newblock Weight structures vs. {$t$}-structures; weight filtrations, spectral
  sequences, and complexes (for motives and in general).
\newblock {\em J. K-Theory}, 6(3):387--504, 2010.

\bibitem[CD16]{cisinski-deglise:etale-motives}
Denis-Charles Cisinski and Fr{\'e}d{\'e}ric D{\'e}glise.
\newblock {\' E}tale motives.
\newblock {\em Compositio Mathematica}, 152(3):556--666, 2016.

\bibitem[Cis10]{cisinski:categories-derivables}
Denis-Charles Cisinski.
\newblock Cat\'{e}gories d\'{e}rivables.
\newblock {\em Bull. Soc. Math. France}, 138(3):317--393, 2010.

\bibitem[DHM22]{dugger-et-al:C2}
Daniel Dugger, Christy Hazel, and Clover May.
\newblock Equivariant $\underline{{\mathbb{Z}}/\ell}$-modules for the cyclic
  group ${C}_2$.
\newblock Available at \url{https://arxiv.org/abs/2203.05287}, 2022.

\bibitem[DRSS99]{drss:exact-vector}
P.~Dr\"{a}xler, I.~Reiten, S.~O. Smal\o, and \O{}. Solberg.
\newblock Exact categories and vector space categories.
\newblock {\em Trans.\ Amer.\,Math.\,Soc.}, 351(2):647--682, 1999.
\newblock With appendix by B.~Keller.

\bibitem[Dye05]{dyer:exact-triangulated}
Matthew Dyer.
\newblock Exact subcategories of triangulated categories.
\newblock available at \url{https://www3.nd.edu/~dyer/papers/index.html}, 2005.

\bibitem[FHM03]{fausk-hu-may:adjoints}
H.~Fausk, P.~Hu, and J.~P. May.
\newblock Isomorphisms between left and right adjoints.
\newblock {\em Theory Appl. Categ.}, 11:No. 4, 107--131, 2003.

\bibitem[Gal18]{gallauer:tt-fmod}
Martin Gallauer.
\newblock Tensor triangular geometry of filtered modules.
\newblock {\em Algebra Number Theory}, 12(8):1975--2003, 2018.

\bibitem[Gal19]{gallauer:tt-dtm-algclosed}
Martin Gallauer.
\newblock tt-geometry of {T}ate motives over algebraically closed fields.
\newblock {\em Compos. Math.}, 155(10):1888--1923, 2019.

\bibitem[Hap88]{happel:stable-cat}
Dieter Happel.
\newblock {\em Triangulated categories in the representation theory of
  finite-dimensional algebras}, volume 119 of {\em LMS Lecture Note}.
\newblock Cambr.\ Univ.\ Press, Cambridge, 1988.

\bibitem[HH05]{haesemeyer-hornbostel:bott}
Christian Hasemeyer and Jens Hornbostel.
\newblock Motives and etale motives with finite coefficients.
\newblock {\em $K$-Theory}, 34(3):195--207, 2005.

\bibitem[Hu05]{pohu:pic-motivic}
Po~Hu.
\newblock On the {P}icard group of the stable {$\Bbb A^1$}-homotopy category.
\newblock {\em Topology}, 44(3):609--640, 2005.

\bibitem[Ivo07]{ivorra:l-adic-real-1}
Florian Ivorra.
\newblock R\'{e}alisation {$l$}-adique des motifs triangul\'{e}s
  g\'{e}om\'{e}triques. {I}.
\newblock {\em Doc. Math.}, 12:607--671, 2007.

\bibitem[Kel90]{keller:chain-stable}
Bernhard Keller.
\newblock Chain complexes and stable categories.
\newblock {\em Manuscripta Math.}, 67(4):379--417, 1990.

\bibitem[KV87]{keller-vossieck:sous-der}
Bernhard Keller and Dieter Vossieck.
\newblock Sous les cat\'{e}gories d\'{e}riv\'{e}es.
\newblock {\em C. R. Acad. Sci. Paris S\'{e}r. I Math.}, 305(6):225--228, 1987.

\bibitem[Lev00]{levine:bott}
Marc Levine.
\newblock Inverting the motivic {B}ott element.
\newblock {\em $K$-Theory}, 19(1):1--28, 2000.

\bibitem[Nee90]{neeman:D(exact)}
Amnon Neeman.
\newblock The derived category of an exact category.
\newblock {\em J. Algebra}, 135(2):388--394, 1990.

\bibitem[Pau17]{Pauwels17}
Bregje Pauwels.
\newblock Quasi-{G}alois theory in symmetric monoidal categories.
\newblock {\em Algebra Number Theory}, 11(8):1891--1920, 2017.

\bibitem[Pet13]{peter:spectrum-damt}
Tobias~J. Peter.
\newblock Prime ideals of mixed {A}rtin-{T}ate motives.
\newblock {\em Journal of K-Theory}, 11(2):331--349, 004 2013.

\bibitem[{Por}15]{porta:stable-derivators-universal}
Marco {Porta}.
\newblock {Universal property of triangulated derivators via {K}eller's
  towers}.
\newblock {\em ArXiv e-prints}, December 2015.

\bibitem[Pos11]{positselski:artin-tate-motives}
Leonid Positselski.
\newblock Mixed {A}rtin-{T}ate motives with finite coefficients.
\newblock {\em Mosc. Math. J.}, 11(2):317--402, 407--408, 2011.

\bibitem[Ric89]{rickard:der-stab}
Jeremy Rickard.
\newblock Derived categories and stable equivalence.
\newblock {\em J. Pure Appl. Algebra}, 61(3):303--317, 1989.

\bibitem[Sch99]{Schneiders:quasi-abelian}
Jean-Pierre Schneiders.
\newblock Quasi-abelian categories and sheaves.
\newblock {\em Mémoires de la Société Mathématique de France},
  76:III1--VI140, 1999.

\bibitem[Ste18]{Stevenson18}
Greg Stevenson.
\newblock A tour of support theory for triangulated categories through tensor
  triangular geometry.
\newblock In {\em Building bridges between algebra and topology}, Adv. Courses
  Math. CRM Barcelona, pages 63--101. Birkh\"{a}user/Springer, Cham, 2018.

\bibitem[Voe00]{Voevodsky00}
Vladimir Voevodsky.
\newblock Triangulated categories of motives over a field.
\newblock In {\em Cycles, transfers, and motivic homology theories}, volume 143
  of {\em Ann. of Math. Stud.}, pages 188--238. Princeton Univ. Press,
  Princeton, NJ, 2000.

\bibitem[Voe03a]{voevodsky:milnor-conjecture}
Vladimir Voevodsky.
\newblock Motivic cohomology with $\mathbb{Z}/2$-coefficients.
\newblock {\em Publ. Math. Inst. Hautes \'Etudes Sci.}, 98:59--104, 2003.

\bibitem[Voe03b]{voevodsky:motivic-power-operations}
Vladimir Voevodsky.
\newblock Reduced power operations in motivic cohomology.
\newblock {\em Publ. Math. Inst. Hautes \'Etudes Sci.}, (98):1--57, 2003.

\bibitem[Wil15]{wildeshaus:interior-motive-pic-surfaces}
J\"{o}rg Wildeshaus.
\newblock On the interior motive of certain {S}himura varieties: the case of
  {P}icard surfaces.
\newblock {\em Manuscripta Math.}, 148(3-4):351--377, 2015.

\bibitem[Wil16]{wildeshaus:artin-tate}
J\"org Wildeshaus.
\newblock Notes on {A}rtin-{T}ate motives.
\newblock In {\em Autour des motifs---\'Ecole d'\'et\'e {F}ranco-{A}siatique de
  {G}\'eom\'etrie {A}lg\'ebrique et de {T}h\'eorie des {N}ombres {V}ol. {III}},
  volume~49 of {\em Panor. Synth\`eses}, pages 101--131. Soc. Math. France,
  Paris, 2016.

\end{thebibliography}
\end{document}